\documentclass[12pt,twoside]{amsart}
\usepackage{amssymb,amsmath,amsthm, amscd, enumerate, mathrsfs, upgreek}
\usepackage{graphicx, hhline, tikz}
\usepackage[colorlinks=true,pagebackref,hyperindex]{hyperref}
\usepackage{dynkin-diagrams}
\usepackage{textcomp}
\usepackage[all]{xy} 

\allowdisplaybreaks

\usepackage{color}
\usepackage[backrefs, alphabetic, initials]{amsrefs}
\usepackage{color} 
\usepackage{latexsym}
\usepackage[T1]{fontenc} 
\usepackage{fancyhdr}
\usepackage{enumerate}
\usepackage{tikz-cd}
\usepackage{array, longtable}
\usepackage{caption}
\usepackage{mathtools}
\newcolumntype{C}{>{$}c<{$}}
\newcolumntype{L}{>{$}l<{$}}

\title[Fano index $\le 66$]{A canonical Fano threefold has Fano index $\le 66$}

\date{\today}

\author{Chen Jiang}
\address{\rm Shanghai Center for Mathematical Sciences \& School of Mathematical Sciences, Fudan University, Shanghai 200438, China}
\email{chenjiang@fudan.edu.cn}

\author{Haidong Liu}
\address{\rm School of Mathematics, Sun Yat-sen University, Guangzhou 510275, China}
\email{liuhd35@mail.sysu.edu.cn, jiuguiaqi@gmail.com}

\DeclareMathOperator{\Supp}{Supp}

\DeclareMathOperator{\rank}{rank}

\DeclareMathOperator{\mult}{mult}

\DeclareMathOperator{\reg}{reg}
\DeclareMathOperator{\Sing}{Sing}

\DeclareMathOperator{\Cl}{Cl}

\DeclareMathOperator{\diag}{diag}
\DeclareMathOperator{\lcm}{lcm}
\DeclareMathOperator{\LB}{LB}

\newcommand{\qQ}{\text{\rm q}_{\mathbb{Q}}}
\DeclareMathOperator{\red}{red}
\DeclareMathOperator{\cd}{cd}

\newtheorem{thm}{Theorem}[section]
\newtheorem{lem}[thm]{Lemma}
\newtheorem{prop}[thm]{Proposition}

\newtheorem{cor}[thm]{Corollary}

\theoremstyle{definition}
\newtheorem{ex}[thm]{Example}
\newtheorem{set}[thm]{Setting}
\newtheorem{defn}[thm]{Definition}

\newtheorem{claim}[thm]{Claim}
\newtheorem{algo}[thm]{Algorithm}

\theoremstyle{remark}
\newtheorem{rem}[thm]{Remark}

\numberwithin{equation}{section}

\begin{document}

\begin{abstract}
    We show that the $\mathbb{Q}$-Fano index of a canonical weak Fano $3$-fold is at most $66$. This upper bound is optimal and gives an affirmative answer to a conjecture of Chengxi Wang in dimension $3$. During the proof, we establish a new Riemann--Roch formula for canonical $3$-folds and provide a detailed study of non-isolated singularities on canonical Fano $3$-folds, concerning both their local and global properties. Our proof also involves a Kawamata--Miyaoka type inequality and geometry of foliations of rank $2$ on canonical Fano $3$-folds. 
\end{abstract}

\dedicatory{Dedicated to Professor Meng Chen on the occasion of his 60th birthday}

\maketitle 
\tableofcontents

\section{Introduction}\label{sec1}

A normal projective variety is called \emph{Fano} (resp., \emph{weak Fano}) if its anti-canonical divisor is ample (resp., nef and big). According to the minimal model program,  Fano varieties with mild singularities serve as fundamental building blocks of algebraic varieties.
Here we remark that in many references, the definition of a Fano variety (or a $\mathbb{Q}$-Fano variety) comes with extra conditions, such as $\mathbb{Q}$-factoriality or Picard number $1$, but we do not impose any extra condition in this paper; instead, we state the conditions explicitly when they are needed. 

For a weak Fano variety $X$ with mild singularities (e.g., terminal or canonical singularities), its $\mathbb Q$-Fano index $\qQ(X)$ is defined to be 
\begin{align*} 
    \qQ(X) {}&\coloneq \max\{q \mid -K_X \sim_{\mathbb Q}qA, \, A\in \Cl (X) \}. 
\end{align*}
The $\mathbb{Q}$-Fano index of a terminal Fano $3$-fold is quite well-understood, which belongs to the set $\{1,\dots,9, 11, 13,17, 19\}$ (see \cites{suzuki, prokhorov2010}) and it plays a crucial role in the classification of terminal Fano $3$-folds \cites{bs1,bs2,prokhorov2007, prokhorov2013,kasprzyk, liu-liu, liu-liu2024}. However, for canonical Fano $3$-folds, progress has been limited due to the absence of Riemann--Roch formula and the classification of canonical singularities,
see \S\,\ref{sec 1.1} for a detailed explanation. 
Even a good upper bound for $\mathbb{Q}$-Fano indices remains unknown. In literature, for a canonical Fano $3$-fold $X$, we only have a very rough upper bound $\qQ(X)\leq 840^2\cdot 324\approx 2.3\cdot 10^8$ (cf. \cite{Birkar4fold}*{Lemma~2.3}): we know that $rK_X$ is Cartier for some positive integer $r\leq 840$ by \cite{CJ2016}*{Proposition~2.4} and $(-K_X)^3\leq 324$ by \cite{JZ}; if $-K_X \sim_{\mathbb{Q}} qA$ for some ample integral Weil divisor $A$, then 
\[
    q\leq qA\cdot (-rK_X)^2= r^2(-K_X)^3\leq 840^2\cdot 324.
\]

In this paper, we give the optimal upper bound for the $\mathbb{Q}$-Fano index of a canonical weak Fano $3$-fold, which gives an affirmative answer to a conjecture of Wang {\cite{wang}*{Conjecture~3.7}} in dimension $3$.

\begin{thm}\label{main.thm}
    Let $X$ be a canonical weak Fano $3$-fold. Then $\qQ(X)\leq 66$.
\end{thm}

Kasprzyk classified terminal or canonical toric Fano $3$-folds in \cites{kasprzyk, kasprzyk2}. The following are some motivating examples with large $\mathbb{Q}$-Fano indices. In particular, Theorem~\ref{main.thm} is optimal. See Remark~\ref{rem q=66} for some discussion of the case of equality. 

\begin{ex}[{\cite{kasprzyk2}*{Table~3}}]\label{ex 123}
    \begin{enumerate}
        \item The weighted projective space $\mathbb P(5,6,22,33)$ is a $\mathbb Q$-factorial canonical Fano $3$-fold of Picard number $1$ with $\mathbb Q$-Fano index $66$. It has a terminal singularity of type $\frac{1}{5}(1,2,3)$ and $3$ non-isolated canonical singularities which are lines of orbifold points of type $\frac{1}{2}(1,1)$, $\frac{1}{3}(1,2)$, and $\frac{1}{11}(5,6)$. 
        
        \item The weighted projective space $\mathbb P(3,5,11,19)$ is a $\mathbb Q$-factorial Fano $3$-fold of Picard number $1$ with isolated canonical singularities. Its $\mathbb Q$-Fano index is $38$. 
        
        \item The weighted projective space $\mathbb P(5,8,9,11)$ is a $\mathbb Q$-factorial Fano $3$-fold of Picard number $1$ with isolated canonical singularities which is non-Gorenstein at a crepant point. Its $\mathbb Q$-Fano index is $33$. 
    \end{enumerate}
\end{ex}

For a direct application, Theorem~\ref{main.thm} gives an explicit boundedness property for a subgroup of K3 type of the automorphism group of 
a rationally connected $3$-fold (see \cite{Loginov}*{Proposition~1.9}).

\subsection{Difficulties and strategy}\label{sec 1.1}

We briefly explain the main difficulty and the strategy of the proof of Theorem~\ref{main.thm}. After a reduction step (see Proposition~\ref{prop.reducetopicard1}), we may assume that $X$ is a $\mathbb Q$-factorial canonical Fano $3$-fold of Picard number $1$ with torsion-free Weil class group $\Cl(X)$ and $-K_X\sim qA$ where $A$ is the ample generator of $\Cl(X)$. The goal is to show that $q\leq 66$.

The main difficulty comes from the singularities. A common way to deal with a $3$-fold $X$ with canonical singularities is to pass to a partial resolution $f\colon Y\to X$ where $Y$ has only terminal singularities and $f^*K_X=K_Y$, then deal with $Y$. The advantage of working on $Y$ is that there is a complete classification of $3$-dimensional terminal singularities and a powerful tool: Reid's Riemann--Roch formula for integral Weil divisors on terminal $3$-folds. But for the proof of Theorem~\ref{main.thm}, taking such a partial resolution is not very helpful, as $f^*A$ is no longer an integral Weil divisor and most of the information of $A$ is lost on $Y$. So we usually have to deal with $X$ directly. This is more difficult: firstly, $3$-dimensional canonical singularities are much more complicated than $3$-dimensional terminal singularities, we do not have a classification of such singularities and they can be non-isolated; secondly, Reid's Riemann--Roch formula does not work directly for integral Weil divisors on canonical $3$-folds, it only works for integral Weil divisors satisfying Reid's condition, namely, those divisors locally proportional to $K_X$.

Our first ingredient is to establish a new Riemann--Roch formula based on Reid's Riemann--Roch formula. Our idea is to associate an integral Weil divisor $D$ on $X$ to an integral Weil divisor $f^{\lfloor *\rfloor}(D)$ on $Y$ which is called the {\it Weil pullback} of $D$ (see Definition~\ref{def weilpullback}) which preserves cohomologies, and then apply Reid's Riemann--Roch formula to $f^{\lfloor *\rfloor}(D)$ on $Y$.
As a consequence, we show that (see Theorem~\ref{thm.rr.sA}) for any integer $0<s<q$,
\begin{align}
    \begin{split}\label{eq intro 1}
        h^0(X,\mathcal O_X(sA))={}& 
        -\frac{s^2}{2}A^2\cdot K_X+2+\sum_{C\subset \Sing(X)}(-K_X\cdot C)c_C(sA)\\
        {}&-\sum_{Q\in B_X}\frac{(\overline{\mathsf{i}_{f^{\lfloor *\rfloor}(sA), Q}b_Q})_{r_Q}(\overline{-\mathsf{i}_{f^{\lfloor *\rfloor}(sA), Q}b_Q})_{r_Q}}{2r_Q},
    \end{split}
\end{align} 
where the first sum comes from the contribution of non-isolated singularities of $X$ and the second sum comes from the contribution of isolated singularities and {\it dissident} points (see \cite[(0.12)(g)]{reidmm}) of $X$. As long as we know everything about the singularities of $X$, we can use \eqref{eq intro 1} to calculate $h^0(X,\mathcal O_X(sA))$ for $0<s<q$. But often we use \eqref{eq intro 1} in another way: as $h^0(X,\mathcal O_X(sA))$ is always an integer, \eqref{eq intro 1} gives strong restrictions (which we call integrality constraints) on the singularities of $X$, which can be used to determine the contributions of singularities of $X$ to the formula although we do not have a full classification of those singularities. 

The second ingredient is a Kawamata--Miyaoka type inequality established in our earlier joint work with Jie Liu \cite{jiang-liu-liu} (see \S\,\ref{sub.KMineq}) which gives an inequality (see Theorem~\ref{thm Delta>jC})
\begin{align}\label{eq intro 2}
    r_Xc_2(X)\cdot c_1(X)- \frac{q^2+2q-4}{4q^2}r_Xc_1(X)^3 \geq \sum_{C\subset\Sing(X)}\left(j_C-\frac{1}{j_C}\right)(-r_XK_X\cdot C),  
\end{align}
where $r_X$ is the Gorenstein index of $X$, $C$ runs over curves in $\Sing(X)$, and $j_C$ is a local invariant of $C$ (see Definition~\ref{def ecgc}). 
This formula gives another restriction on non-isolated singularities of $X$, as well as a restriction for other numerical data of $X$ such as $r_X$, $q$, $c_1(X)^3$, and $c_2(X)\cdot c_1(X)$. Indeed, only this inequality itself already gives an upper bound for $q$, but not as good as we expect.

In the application of \eqref{eq intro 1} and \eqref{eq intro 2}, a natural technical difficulty is to determine the value or the lower bound of the intersection number $(-K_X\cdot C)$ where $C$ is a curve in $\Sing(X)$. This actually depends on the global geometry of $X$ and $C$ rather than local singularities. Priorly, we only have a naive lower bound $(-K_X\cdot C)\geq \frac{1}{r_X}$ which is not very helpful; actually, we observed in many examples that often $(-K_X\cdot C)\geq 1$ holds (see Example~\ref{ex compute -rKC}), which suggests that the lower bound could be greatly improved under good conditions. 
So we establish a method to give a better lower bound for $(-K_X\cdot C)$ in \S\,\ref{sec.LB} with an algorithm depending on the type of $C$ and singularities of $Y$. The idea is to relate the intersection number $(-K_X\cdot C)$ with Reid's Riemann--Roch formula for exceptional divisors centered at $C$ on $Y$, then we can get some integrality constraints which eventually show that in many cases the denominator of $(-K_X\cdot C)$ is small, sometimes it is just $1$.

Putting all the above ideas together, we are able to give an algorithm to 
find possible numerical data of $X$ with $q>66$. 
It turns out that any $X$ with $q>66$ belongs to a list of $36$ numerical types (see Theorem~\ref{thm.36cases}) with the help of a computer program (which took about 15--30 minutes to run on a laptop). As a comparison, the number of canonical Fano $3$-folds is very huge, for example, there are 674,688 toric canonical Fano $3$-folds \cite{kasprzyk2}. 

Then the final mission is to rule out those $36$ numerical types. In all cases, we have a good understanding of curves $C\subset \Sing(X)$ and 
the intersection numbers $(-K_X\cdot C)$, so we can calculate $h^0(X,\mathcal O_X(sA))$ up to finitely many choices of the second sum. Then we can rule out $24$ numerical types in the list as $h^0(X,\mathcal O_X(sA))$ fails to be an integer for some $s$ (see \S\,\ref{sec 7}). In order to deal with integrality constraints, often we need to treat  congruence equations on square residues.

For the remaining $12$ numerical types, it turns out that $h^0(X,\mathcal O_X(sA))$ can be computed explicitly. An interesting phenomenon is that in all these $12$ cases,
\[
h^0(X, \mathcal O_X(sA))=h^0(\mathbb{P}(5,6,22,33), \mathcal{O}_{\mathbb{P}(5,6,22,33)}(s))
\]
for $0<s<66$. So there is no contradiction from $h^0(X,\mathcal O_X(sA))$. 
In order to rule out the remaining cases, we shall study the geometric structure of $X$ in more detail. In the end, we find a contradiction on the algebraically integrable foliation $\mathcal{F}$ of rank $2$ induced by a refinement of \eqref{eq intro 2}: it turns out that $(-K_\mathcal{F})^2\cdot L\leq 8$ where $L$ is a general leaf of $\mathcal{F}$; on the other hand, by the numerical behavior of $h^0(X,\mathcal O_X(sA))$ and the geometry of the foliation $\mathcal{F}$, we show that $L\sim sA$ for some $s\geq 60$, which implies that $(-K_\mathcal{F})^2\cdot L\geq (-K_\mathcal{F})^2\cdot 60A>8$ and gives a contradiction (see \S\,\ref{sec group C1}--\S\,\ref{sec group C2}). 

\section{Preliminaries}

Throughout this paper, we work over the complex number field $\mathbb C$.
We will freely use the basic notation in \cites{KMM, kollar-mori}.
 
\subsection{Singularities}

Let $X$ be a normal variety such that $K_X$ is $\mathbb Q$-Cartier. The \emph{Gorenstein index} $r_X$ of $X$ is defined as the smallest positive integer $m$ such that $mK_X$ is Cartier. 
A prime divisor $E$ \emph{over} $X$ is a prime divisor on some model $Y$ where  $f\colon Y\to X$ is a proper birational morphism. The image $f(E)$ is called the \emph{center} of $E$ on $X$, and $E$ is said to be \emph{centered at} $f(E)$. 
Write
\[
    K_Y=f^*K_X+\sum_Ea(E,X)E,
\]
where $a(E,X)\in \mathbb Q$ is called the \emph{discrepancy} of $E$. We say that $X$ has \emph{terminal} (resp., \emph{canonical}) singularities if $a(E,X)>0$ (resp., $a(E,X)\geq 0$) for any exceptional divisor $E$ over $X$. Often we just simply say that $X$ is {\it terminal} or {\it canonical}, respectively. 

A \emph{crepant center} of $X$ is the center of an exceptional divisor $E$ over $X$ with $a(E,X)=0$. 
A \emph{crepant curve} (resp., \emph{crepant point})  is a crepant center of dimension $1$ (resp., $0$). 
Denote by $\Sing(X)$ the singular locus of $X$ and by $\Sing_1(X)$ the $1$-dimensional part of $\Sing(X)$. 
Here recall that if $X$ is a canonical $3$-fold, then $\Sing_1(X)$ consists of crepant curves as $3$-dimensional terminal singularities are isolated (\cite{kollar-mori}*{Corollary~5.18}). 

If $X$ is a canonical variety, then by the minimal model program, there exists a \emph{terminalization} $f\colon Y\to X$ such that $Y$ is a $\mathbb{Q}$-factorial terminal variety with $f^*K_X=K_Y$ and $f$ is projective birational (see \cite{BCHM}*{Corollary~1.4.3}).
As $f^*K_X=K_Y$, for a prime exceptional divisor $E$ on $Y$, $f(E)$ is a crepant center; also $X$ is Gorenstein in a neighborhood of $f(E)$ if and only if $Y$ is Gorenstein in a neighborhood of $E$.

\subsection{Fano indices}\label{subsec.fi}

Let $X$ be a canonical weak Fano variety. The \emph{$\mathbb Q$-Fano index} of $X$ is defined by 
\begin{align*} 
    \qQ(X) {}&\coloneq \max\{q \mid -K_X \sim_{\mathbb Q}qA, \, A\in \Cl (X) \}. 
\end{align*}
It is known that $\Cl (X)$ is a finitely generated Abelian group, so $\qQ(X)$ is a positive integer. For more details, see \cite{ip}*{\S\,2} or \cite{prokhorov2010}.

Recall the following theorem which relates Fano index with anti-canonical degree and other invariants. 

\begin{thm}[\cite{jiang-liu-liu}*{Theorem~4.2}]\label{thm.degreeandindex}
    Let $X$ be a canonical Fano $3$-fold. Let $A$ be an ample integral Weil divisor on $X$ such that $-K_X\sim qA$ for some positive integer $q$. Let $J_A$ be the smallest positive integer such that $J_AA$ is Cartier in codimension $2$. Then $J_Ar_X(-K_X)^3$ is an integer and the following divisibility constraints hold: \[J_A\mid q \text{ and } q^2\mid J_Ar_X(-K_X)^3.\] 
\end{thm}

\subsection{Crepant curves on canonical $3$-folds}

We recall basic properties of crepant curves on canonical $3$-folds; generically they behave the same as Du Val singularities which are well-understood.

\begin{defn}\label{def ecgc}
    Let $X$ be a canonical $3$-fold and let $C\subset \Sing(X)$ be a crepant curve. Then at a general point of $C$, $X$ is analytically isomorphic to $\mathbb{A}^1\times S_C$ where $S_C$ is a Du Val singularity (see \cite{reidc3f}*{Corollary~1.14}). 

    We say that $C\subset \Sing(X)$ is {\it of type $\mathsf{T}$} if $S_C$ is a Du Val singularity of type $\mathsf{T}$, or equivalently, for a general hyperplane $H$ on $X$, $H$ has Du Val singularities of type $\mathsf{T}$ in a neighborhood of any point of $H\cap C$. Here $\mathsf{T}\in \{\mathsf{A}_n, \mathsf{D}_m, \mathsf{E}_k\mid n\geq 1, m\geq 4, k=6,7,8\}.$ 

    We associate the following invariants to $C$:
    \begin{enumerate}
        \item $e_C$ is defined to be $1$ plus the number of exceptional curves on the minimal resolution of $S_C$; 
    
        \item $e'_C$ is defined to be $1$ plus the length of the longest chain of exceptional curves on the minimal resolution of $S_C$;
    
        \item $g_C$ is defined to be the order of the local fundamental group of $S_C$;
        
        \item  $j_C$ is defined to be the order of the Weil divisor class group of $S_C$ (see \cite{Kawakita2024}*{Remark~4.2.9}). 
    \end{enumerate} 
    Namely, 
    \begin{align*}
        (e_C, e'_C, g_C, j_C)=
        \begin{cases} (n+1, n+1, n+1, n+1) & \text{ if } C\subset \Sing(X) \text{ is of type } \mathsf{A}_n;\\
        (m+1, m, 4m-8, 4) & \text{ if } C\subset \Sing(X) \text{ is of type } \mathsf{D}_m;\\
        (7, 6, 24, 3) & \text{ if } C\subset \Sing(X) \text{ is of type } \mathsf{E}_6;\\
        (8,7, 48, 2) & \text{ if } C\subset \Sing(X) \text{ is of type } \mathsf{E}_7;\\
        (9, 8, 120,1) & \text{ if } C\subset \Sing(X) \text{ is of type } \mathsf{E}_8.
        \end{cases} 
    \end{align*}
\end{defn}

\begin{defn}[Non-split crepant divisor]\label{def nonsplit}
    Keep the notation in Definition~\ref{def ecgc}. Let $f\colon Y\to X$ be a terminalization and let $S'_C$ be the minimal resolution of $S_C$.
    Then over a general point of $C$, the morphism $f\colon Y\to X$ is analytically isomorphic to $\mathbb{A}^1\times S'_C\to \mathbb{A}^1\times S_C$. 
    For a prime $f$-exceptional divisor $E$ on $Y$ centered at $C$, a general fiber of $E\to C$ is reduced but not necessarily irreducible, and it corresponds to a union of exceptional curves of $S'_C\to S_C$. One should be careful that $E$ is irreducible but $E$ might be analytically reducible over a general point of $C$.

    We say that $C$ is {\it with non-split crepant divisors} if for any prime $f$-exceptional divisor $E$ centered at $C$, a general fiber of $E\to C$ is irreducible. In this case, a prime $f$-exceptional divisor $E$ centered at $C$ corresponds to an exceptional curve (i.e., a $(-2)$-curve) of $S'_C$ over $S_C$, and vice versa. We denote this correspondence by $E\longleftrightarrow \ell_E$. 
\end{defn}

We provide a simple example  to illustrate crepant curves with non-split/split crepant divisors.

\begin{ex}[Non-split versus split]
    Consider $X\coloneq (x^2-y^2+z^3t=0)\subset \mathbb{A}_{x,y,z}^3\times \mathbb{C}^*_t$ and $C=(x=y=z=0)=\Sing(X)$. Then 
    $C$ is of type $\mathsf{A}_2$ and the blowup of $X$ along $C$ gives a resolution of $X$ and produces $2$ prime exceptional divisors over $C$. So $C$ is with non-split crepant divisors. 

    Now consider the free involution $\tau$ on $X$ given by $(x,y, z,t)\mapsto (x,-y,-z,-t)$, then $X'\coloneq X/\tau$ is again canonical and $\Sing(X)=C'\coloneq C/\tau$ is again of type $\mathsf{A}_2$. But as the involution exchanges $x+y$ and $x-y$, it glues the $2$ prime exceptional divisors over $C$ into a single prime exceptional divisor over $C'$, whose general fiber over $C'$ is reducible. This means that $C'$ is not with non-split crepant divisors.  
\end{ex}

The following lemma gives the relation between the intersection of a crepant curve with a divisor and the intersection of exceptional divisors centered at this curve with the pullback of the divisor. 

\begin{lem}\label{lem EEK=EECK} 
    Let $X$ be a projective canonical $3$-fold and let $C\subset \Sing(X)$ be a crepant curve with non-split crepant divisors. Let $f\colon Y\to X$ be a terminalization. Keep the notation in Definition~\ref{def ecgc} and Definition~\ref{def nonsplit}. Then for any prime $f$-exceptional divisors $E, E'$ centered at $C$ and any $\mathbb{Q}$-Cartier $\mathbb{Q}$-divisor $H$ on $X$, 
    \[
        (E\cdot E'\cdot f^*H)=(\ell_E\cdot  \ell_{E'})\cdot (H\cdot C).
    \]
    In particular, 
    \[
        (E\cdot E'\cdot K_Y)=(\ell_E\cdot  \ell_{E'})\cdot (K_X\cdot C).
    \]
\end{lem}

\begin{proof}
    The assertion is linear on $H$, so we may assume that $H$ is a general very ample divisor on $X$. Then $H$ is a projective surface with Du Val singularities by Bertini's theorem. Take $H'\coloneq f^*H=f^{-1}H$ on $Y$. Then $f|_{H'}\colon H'\to H$ is the minimal resolution of $H$. By construction, $E|_{H'}$ (resp., $E'|_{H'}$) is a disjoint union of $(-2)$-curves $C_i$ (resp., $C'_i$) ($1\leq i \leq  (H\cdot C)$) over the points in $H\cap C$. Here $(C_i\cdot C_i')=(\ell_{E}\cdot \ell_{E'})$ for each $i$ via the correspondence $E\longleftrightarrow \ell_E$, and $(C_i\cdot C'_j)=0$ for $i\neq j$. Then it follows that 
    \[
        (E\cdot E'\cdot f^*H)=(E|_{H'}\cdot E'|_{H'})=(\ell_E\cdot  \ell_{E'})\cdot (H\cdot C). \qedhere
    \] 
\end{proof}

\begin{rem}\label{rem kawakita}
    In this subsection, we mainly focus on the relation between the generic point of a crepant curve and a Du Val singularity. The classification of crepant points and special points on crepant curves is a much more complicated task, and we know very little about it. Here we recall an interesting and useful result due to Kawakita \cite{kawakita}*{Theorem~1.1} on crepant points: let $X$ be a canonical $3$-fold with a crepant point $P$, then $r_P\leq 6$ where $r_P$ is the Gorenstein index of $X$ at $P$; moreover, by \cite{kawakita}*{Table~2}, Reid's basket $B_X$ (see \S\,\ref{sec RR3}) contains the following orbifold points:
    \[
        \begin{cases}
            \{4\times (2,1)\} &\text{if } r_P=2;\\
            \{3\times (3,1)\} &\text{if } r_P=3;\\
            \{(2,1), 2\times (4,1)\} &\text{if } r_P=4;\\
            \{(5,1), (5,2)\} &\text{if } r_P=5;\\
            \{(2,1), (3,1), (6,1)\} &\text{if } r_P=6.
        \end{cases}
    \]
\end{rem}

\subsection{Kawamata--Miyaoka type inequality for canonical Fano threefolds}\label{sub.KMineq}

The following theorem is a detailed version of the Kawamata--Miyaoka type inequality of \cite{jiang-liu-liu}*{Theorem~3.8}. Here we refer to \cite{gkpt} or \cite{jiang-liu-liu}*{\S\,3} for the definition of \emph{generalized second Chern class}  (or \emph{orbifold second Chern class}) $\hat{c}_2(X)$ of a $3$-fold $X$. 

\begin{thm}[\cite{jiang-liu-liu}*{Theorem~3.8}]\label{thm.kmineqrefined}
    Let $X$ be a $\mathbb{Q}$-factorial canonical Fano $3$-fold of Picard number $1$ and let $q\coloneq \qQ(X)$ be the $\mathbb{Q}$-Fano index of $X$. Let
    \begin{align}\label{eq.hnfiltration}
        0=\mathcal{E}_0\subsetneq \mathcal{E}_1\subsetneq \dots \subsetneq \mathcal{E}_l=\mathcal{T}_X
    \end{align}
    be the Harder--Narasimhan filtration of the tangent sheaf $\mathcal T_X$ with respect to $c_1(X)$, where $1\leq l\leq 3$. Denote by $r_1\coloneq\rank \mathcal{E}_1$ and take $p$ to be the integer such that $c_1(\mathcal{E}_{l-1})\equiv \frac{p}{q}c_1(X)$. Here $p<q$, and $p>\frac{l-1}{l}q$ if $l>1$. Then we have
    \[
        \frac{c_1(X)^3}{\hat{c}_2(X)\cdot c_1(X)} \leq 
        \begin{dcases}
        3   & \text{if } (l, r_1)=(1,3);\\
        \frac{16}{5}  & \text{if } (l, r_1)=(2,1);\\
        \frac{4q^2}{p(4q-3p)}  \leq \frac{4q^2}{q^2+2q-3}  & \text{if } (l, r_1)=(2,2);\\
        \frac{4q^2}{-4p^2+6pq-q^2}  \leq \frac{4q^2}{q^2+2q-4}  & \text{if } (l, r_1)=(3,1).
        \end{dcases}
    \]
\end{thm}

\begin{proof}
    We refer the reader to \cite{jiang-liu-liu}*{Proof of Theorem~3.8} with some necessary explanations as the following. Recall that in \cite{jiang-liu-liu}*{Proof of Theorem~3.8}, we denote $\mathcal{F}_i\coloneqq (\mathcal{E}_i/\mathcal{E}_{i-1})^{**}$ and denote by $q_i$ the unique positive integer such that $c_1(\mathcal{F}_i)\equiv \frac{q_i}{q}c_1(X)$. Then $p=\sum_{i=1}^{l-1}q_i$. 
    
    If $(l, r_1)=(1,3)$, that is, $\mathcal{T}_X$ is semistable with respect to $c_1(X)$, then the inequality follows directly from the $\mathbb{Q}$-Bogomolov--Gieseker inequality \cite{KMM94}*{Lemma~6.5}.
    
    If $(l, r_1)=(2,1)$, then the inequality follows from \cite{jiang-liu-liu}*{Proof of Theorem~3.8, Case~1}.
    
    If $(l, r_1)=(2,2)$, then $p=q_1$ and $\frac23 q< p\leq q-1$.
    By \cite{jiang-liu-liu}*{Proof of Theorem~3.8, Case~2}, we have  
    \[
        6 \hat{c}_2(X)\cdot c_1(X) - 2c_1(X)^3 \geq -\frac{(3p-2q)^2}{2q^2} c_1(X)^3,
    \]
    which yields the desired inequality.
    
    If $(l, r_1)=(3,1)$, then from \cite{jiang-liu-liu}*{Proof of Theorem~3.8, Case~3}, we have $p=q_1+q_2$ with $2 \leq q_2 \leq q_1- 1 \leq \frac q2- 1$ and
    \begin{align*}
        6\hat{c}_2(X)\cdot c_1(X) - 2 c_1(X)^3 
        & \geq - \left((q_1-q_2)^2 + (2q_1+q_2-q)^2 + (q_1+2q_2-q)^2\right) \cdot \frac{1}{q^2} c_1(X)^3 \\
        & = - \left((2q_1-p)^2 + (q_1+p-q)^2 + (2p-q_1-q)^2\right) \cdot \frac{1}{q^2} c_1(X)^3 \\
        & \geq - \left(6 p^2 - 9pq + \frac{7}{2}q^2\right) \frac{1}{q^2} c_1(X)^3,
    \end{align*}
    which yields the desired inequality. Here for the last inequality, we used the fact that the first two leading terms of $ (2q_1-p)^2 + (q_1+p-q)^2 + (2p-q_1-q)^2$ as a polynomial of $q_1$ is $6q_1^2-6pq_1$, so it attains the maximal value at $q_1=\frac{q}{2}$ as $q_1>\frac{q_1+q_2}{2}=\frac{p}{2}$. 
\end{proof}

The relation between $c_2(X)\cdot c_1(X)$ and $\hat{c}_2(X)\cdot c_1(X)$ is given by the following formula which is a special case of \cite{jiang-liu-liu}*{Theorem~4.6}.

\begin{thm}[{\cite{jiang-liu-liu}*{Theorem~4.6}}]\label{thm.c1c2 diff}
    Let $X$ be a projective canonical $3$-fold. Then
    \[
        c_2(X)\cdot c_1(X)- \hat{c}_2(X)\cdot c_1(X)=\sum_{C\subset \text{\rm Sing}(X)}\left(e_C-\frac{1}{g_C}\right)(-K_X\cdot C),
    \]
    where the sum runs over crepant curves $C\subset \text{\rm Sing}(X)$ and $e_C, g_C$ are defined in Definition~\ref{def ecgc}.
\end{thm} 
 
\subsection{Unmodifiable canonical Fano $3$-folds}

We introduce the concept of an unmodifiable canonical Fano $3$-fold.

\begin{defn}\label{def unmod}
    We say that $(X, A)$ is an {\it unmodifiable} canonical Fano $3$-fold if the following conditions are satisfied:
    \begin{enumerate}
        \item $X$ is a $\mathbb Q$-factorial canonical Fano $3$-fold of Picard number $1$;
        \item $\Cl(X)$ is torsion-free and $A$ is the ample generator of  $\Cl(X)$;
        \item For any terminalization $f\colon Y\to X$ and any prime $f$-exceptional divisor $E$ on $Y$, $\mult_Ef^*A\not\in \mathbb{Z}$.
    \end{enumerate}
\end{defn}

A more geometric interpretation of Definition~\ref{def unmod}(3) will be given in Corollary~\ref{cor A primitive}, where we will show that Definition~\ref{def unmod}(3) implies that any crepant curve $C\subset \Sing(X)$ is of type $\mathsf{A}$ with non-split crepant divisors, and $A$ generates the local Weil class group of $X$ at the generic point of $C$. 

The following proposition allows us to reduce Theorem~\ref{main.thm} to unmodifiable canonical Fano $3$-folds. The idea is that we can use the minimal model program to construct a new model if Definition~\ref{def unmod}(3) fails for some $E$. The terminology ``unmodifiable" means that $(X,A)$ can not be further modified by the minimal model program. 

\begin{prop}[{cf. \cite{jiang-liu-liu}*{Proposition~5.2}}]\label{prop.reducetopicard1}
    Let $X$ be a canonical weak Fano $3$-fold with $\qQ(X)\geq 7$. Then there exists an unmodifiable canonical Fano $3$-fold $(X', A')$ such that $\qQ(X')\geq \qQ(X)\geq 7$.
\end{prop}

\begin{proof}
    For a projective canonical variety $W$, define $\cd(W)$ to be the number of prime divisors $E$ exceptional over $W$ such that $a(E,W)=0$. This is a well-defined non-negative integer (see \cite{kollar-mori}*{Proposition~2.36}).
 
    Denote by $\mathcal{X}$ the set of $\mathbb Q$-factorial canonical Fano $3$-folds $X'$ of Picard number $1$ with $\qQ(X')\geq \qQ(X)\geq 7.$ By \cite{jiang-liu-liu}*{Proposition~5.2}, $\mathcal{X}$ is not empty. 

    Take $X'\in \mathcal{X}$ such that $((-K_{X'})^3, -\cd(X'))$ is maximal in lexicographical order. Then $X'$ is a $\mathbb Q$-factorial canonical Fano $3$-fold of Picard number $1$ by definition, and $\Cl(X')$ is torsion-free by \cite{jiang-liu-liu}*{Lemma~5.1} and the maximality of $((-K_{X'})^3, -\cd(X'))$. 

    Take $A'$ to be the ample generator of  $\Cl(X')$. Then we claim that $(X', A')$ is an unmodifiable canonical Fano $3$-fold. It suffices to check Definition~\ref{def unmod}(3).

    If there exists a terminalization $f\colon Y\to X'$ and a prime $f$-exceptional divisor $E$ on $Y$ such that $\mult_{E}f^*A'\in \mathbb{Z}$, 
    then by \cite{BCHM}*{Corollary~1.4.3}, there is a projective birational morphism $f'\colon Y'\to X'$ such that $Y'$ is $\mathbb{Q}$-factorial and $E$ is the only prime $f'$-exceptional divisor on $Y'$. 
    Then $Y'$ is a $\mathbb{Q}$-factorial canonical weak Fano $3$-fold with $-K_{Y'}\sim \qQ(X')f'^*A'$ where $f'^*A'$ is an integral Weil divisor by the assumption that $\mult_{E}f^*A'\in \mathbb{Z}$. 
    Then we can run a $K$-MMP on $Y'$ which ends up with a Mori fiber space $X''\to T$ where $X''$ is $\mathbb{Q}$-factorial and canonical. Hence $-K_{X''}\sim \qQ(X')A''$ where $A''$ is the strict transform of $f'^*A'$. If $\dim T>0$, then for a general fiber $F$ of $X''\to T$, we have $-K_{F}\sim \qQ(X')A''|_F$ where $F$ is either $\mathbb{P}^1$ or a canonical del Pezzo surface, but this contradicts \cite{wang}*{Proposition~3.3} as $\qQ(X')\geq 7$. Hence $\dim T=0$ and $X''$ is a $\mathbb Q$-factorial canonical Fano $3$-fold of Picard number $1$ with $\qQ(X'')\geq \qQ(X')\geq \qQ(X)$. But from the construction of $X''$, we have 
    \[
        \cd(X')=\cd({Y'})+1\geq \cd(X'')+1
    \]
    by \cite{kollar-mori}*{Lemma~3.38} and 
    \[
        (-K_{X'})^3=(-K_{Y'})^3\leq (-K_{X''})^3
    \]
    by \cite{jiang21}*{Lemma~4.4}. This contradicts the maximality of $((-K_{X'})^3, -\cd(X'))$. 
\end{proof}

The condition in Definition~\ref{def unmod}(3) is closely related to the following observation on Du Val singularities. 

\begin{lem}\label{lem DE type is modifiable}
    Let $S$ be a Du Val singularity and let $\pi\colon S'\to S$ be the minimal resolution. Let $D$ be an integral Weil divisor on $S$. Suppose that one of the following holds:
    \begin{enumerate}
        \item $S$ is not of type $\mathsf{A}$; or

        \item $S$ is of type $\mathsf{A}$ and $D$ is not a generator of $\Cl(S)$.
    \end{enumerate}
    Then there exists a prime exceptional curve $\ell$ on $S'$ such that $\mult_\ell\pi^*D\in \mathbb{Z}$.
\end{lem}

\begin{proof}
    We may write 
    \begin{align}
        \pi^*D=D'+\sum_i a_i\ell_i \label{eq D=D'+E}
    \end{align}
    where $D'$ is the strict transform of $D$ on $S'$ and $\ell_i$ are all exceptional curves on $S'$. The goal is to show that $a_i\in \mathbb{Z}$ for some $i$.

    The statement is clear if $S$ is of type $\mathsf{E}_8$ as in this case $\Cl(S)$ is trivial. Recall that the Weil divisor class group of $S$ is $\mathbb{Z}_2\times\mathbb{Z}_2$, $\mathbb{Z}_4$, $\mathbb{Z}_3$, $\mathbb{Z}_2$ if $S$ is of type $\mathsf{D}_{2m}$, $\mathsf{D}_{2m+1}$, $\mathsf{E}_{6}$, $\mathsf{E}_{7}$ respectively (see \cite{Kawakita2024}*{Remark~4.2.9}).

    If $S$ is of type $\mathsf{D}_{2m}$ or $\mathsf{E}_7$, then $2a_i\in \mathbb{Z}$ for any $i$. We may assume that $\ell_1$ only intersects with $\ell_2$. Intersecting \eqref{eq D=D'+E} with $\ell_1$, we have $-2a_1+a_2\in \mathbb{Z}$, which implies that $a_2\in \mathbb{Z}$. 

    If $S$ is of type $\mathsf{E}_6$, then $3a_i\in \mathbb{Z}$ for any $i$. We may label the first $3$ exceptional curves in the dual graph as the following:
    \[
        \dynkin [scale=2, labels={ \ell_1,, \ell_2,\ell_3 ,}] E6
    \]
    Intersecting \eqref{eq D=D'+E} with $\ell_2-\ell_1$, we have $a_3-3a_2+3a_1\in\mathbb{Z}$, which implies that $a_3\in \mathbb{Z}$. 
 
    If $S$ is of type $\mathsf{D}_{2m+1}$, then $4a_i\in \mathbb{Z}$ for any $i$. We may label the exceptional curves in the dual graph as the following:
    \[
        \dynkin [scale=2, labels={\ell_1,\ell_2, \ell_{2m-2} ,\ell_{2m-1}, \ell_{2m}, \ell_{2m+1}}, label directions={,,,right, ,  }] D{}
    \]
    Intersecting \eqref{eq D=D'+E} with $\sum_{i=1}^m 2\ell_{2m}$, we have 
    \[
        2a_1-\sum_{j=2}^{2m}(-1)^j4a_j\in \mathbb{Z},
    \]
    which implies that $2a_1\in \mathbb{Z}$. 
    Intersecting \eqref{eq D=D'+E} with $\ell_{1}$, we have $-2a_1+a_{2}\in \mathbb{Z}$, which implies that $a_2\in \mathbb{Z}$.

    Finally, suppose that $S$ is of type $\mathsf{A}_{n}$. We may label the exceptional curves in the dual graph as the following:
    \[
        \dynkin [scale=2, labels={ \ell_1,\ell_2,\ell_3,\ell_n}] A{***.*}
    \]  
    As $S$ can be viewed as a quotient singularity of type $\frac{1}{n+1}(1, -1)$, there is an integral Weil divisor $D_0$ on $S$ generating $\Cl(S)\simeq \mathbb{Z}_{n+1}$ such that 
    \[
        \pi^*D_0=D'_0+\sum_{i=1}^n\frac{i}{n+1}\ell_i,
    \]
    where $D'_0$ is the strict transform of $D_0$ on $S'$.
    If $D$ is not a generator of $\Cl(S)$, then $D\sim kD_0$ for some $1\leq k\leq n+1$ such that $\gcd(k,n+1)>1$. This implies that $\mult_{\ell_{i}}\pi^*D\in \mathbb{Z}$ for any $i$ divided by $\frac{n+1}{\gcd(k,n+1)}$.
\end{proof}

As a corollary, crepant curves on an {unmodifiable} canonical Fano $3$-fold are all of type $\mathsf{A}$ with non-split crepant divisors.

\begin{cor}\label{cor A primitive}
    Let $(X, A)$ be an {unmodifiable} canonical Fano $3$-fold and let $C\subset \Sing(X)$ be a crepant curve. Then $C$ is of type $\mathsf{A}$ with non-split crepant divisors, and the Cartier index of $A$ at the generic point of $C$ is exactly $j_C$.
\end{cor}

\begin{proof} 
    Let $f\colon Y\to X$ be a terminalization. Then as in Definition~\ref{def ecgc}, at a general point of $C$, $X$ is analytically isomorphic to $\mathbb{A}^1\times S_C$ where $S_C$ is a Du Val singularity. By \cite{hartshorne}*{Proposition~II.6.6}, locally we have a linear equivalence $A\sim \mathbb{A}^1\times A_0$ where $A_0$ is an integral Weil divisor on $S_C$. 
    
    We claim that $S_C$ is of type $\mathsf{A}$ and $A_0$ is a generator of $\Cl(S_C)$. Otherwise, by Lemma~\ref{lem DE type is modifiable}, there exists an exceptional curve $\ell_0$ on the minimal resolution $\pi\colon  S'_C\to S_C$ such that $\mult_{\ell_0}\pi^*A_0\in \mathbb{Z}$. Then there exists a prime $f$-exceptional divisor $E_0$ on $Y$ centered at $C$ such that $\ell_0$ corresponds to an irreducible component of a general fiber of $E_0\to C$ (see Definition~\ref{def nonsplit}). As a general fiber of $E_0\to C$ is reduced, we know that $\mult_{E_0}f^*A-\mult_{\ell_0}\pi^*A_0\in \mathbb{Z}$ by the local linear equivalence $A\sim \mathbb{A}^1\times A_0$. Then $\mult_{E_0}f^*A\in \mathbb{Z}$. But this contradicts Definition~\ref{def unmod}(3). So $S_C$ is of type $\mathsf{A}$ and $A_0$ is a generator of $\Cl(S_C)$; in particular, $C$ is of type $\mathsf{A}$ and the Cartier index of $A$ at the generic point of $C$ is exactly $j_C$.

    Finally we show that $C$ is with non-split crepant divisors. Suppose that there exists a prime $f$-exceptional divisor $E$ on $Y$ centered at $C$ such that a general fiber of $E\to C$ is not irreducible, then it contains $2$ irreducible curves corresponding to exceptional curves $\ell, \ell'$ on the minimal resolution $\pi\colon S'_C\to S_C$. Then we know that $\mult_{E}f^*A-\mult_{\ell}\pi^*A_0\in \mathbb{Z}$ and $\mult_{E}f^*A-\mult_{\ell'}\pi^*A_0\in \mathbb{Z}$ by the local linear equivalence $A\sim \mathbb{A}^1\times A_0$. In particular, $ \mult_{\ell}\pi^*A_0-\mult_{\ell'}\pi^*A_0\in \mathbb{Z}$. But from the proof of Lemma~\ref{lem DE type is modifiable}, there is an integral Weil divisor $D_0$ on $S_C$ such that 
    \[
        \pi^*D_0=D'_0+\sum_{i=1}^n\frac{i}{n+1}\ell_i,
    \]
    where $D'_0$ is the strict transform of $D_0$ on $S'_C$, $\ell_i$ are exceptional curves, and $A_0\sim kD_0$ for some integer $k$ coprime to $n+1$. Then we have $\ell=\ell_i$ and $\ell'=\ell_j$ for some $i\neq j$. But then $\mult_{\ell}\pi^*A_0-\mult_{\ell'}\pi^*A_0\in \mathbb{Z}$ implies that $\frac{ki}{n+1}-\frac{kj}{n+1}\in \mathbb{Z}$, which is absurd.   
\end{proof}

\section{Reid's Riemann--Roch formula for surfaces and threefolds}

In this section, we recall Reid's Riemann--Roch formula for integral Weil divisors on surfaces with Du Val singularities and for integral Weil divisors satisfying Reid's condition on $3$-folds with canonical singularities. 

\subsection{Reid's Riemann--Roch formula for surfaces with Du Val singularities} 

Let $S$ be a projective surface with Du Val singularities and let $D$ be an integral Weil divisor on $S$. According to Reid \cite{reid}*{(9.1)}, there is a formula
\begin{align}\label{eq RR duval}
    \chi(S, \mathcal{O}_S(D))=\chi(S, \mathcal{O}_S)+\frac12D\cdot (D-K_S)+\sum_{Q\in\Sing(S)}c_Q(D).
\end{align}
Here for a singular point $Q\in\Sing(S)$, $c_Q(D)$ is defined as the following:
\begin{enumerate}
    \item If $(Q\in S, D)$ is a cyclic quotient singularity of type 
    $i(\frac{1}{r}(1, -1))$ for some $0\leq i<r$
    (see \cite{reid}*{(8.3)}), then $c_Q(D)=\frac{-i(r-i)}{2r}$;

    \item In general, we can locally deform $(Q\in S, D)$ into a collection of cyclic quotient singularities $(Q_\alpha\in S_\alpha, D_\alpha)$ of type 
    $i_\alpha(\frac{1}{r_\alpha}(1, -1))$, and define
    \[
        c_Q(D)=\sum_\alpha c_{Q_\alpha}(D_\alpha)=-\sum_\alpha\frac{i_\alpha(r_\alpha-i_\alpha)}{2r_\alpha}.
    \]
\end{enumerate}
Note that $c_Q(D)$ depends only on the local Weil divisor class of $D$ at $Q$.

For a crepant curve on a canonical $3$-fold, we introduce the following invariant.
\begin{defn}\label{def ccD}
    Let $X$ be a canonical $3$-fold and let $D$ be an integral Weil divisor on $X$. Let $C\subset \Sing(X)$ be a crepant curve. Then at a general point $P\in C$, $X$ is analytically isomorphic to a neighborhood of $(0, P_0)\in \mathbb{A}^1\times S_C$, where $P_0\in S_C$ is a Du Val singularity. By \cite{hartshorne}*{Proposition~II.6.6}, locally we have a linear equivalence $D\sim \mathbb{A}^1\times D_0$ where $D_0$ is an integral Weil divisor on $S_C$.
    So we just define
    \[
        c_C(D)\coloneq c_{P_0}(D_0).
    \] 
    Here the definition is well-defined as $c_{P_0}(D_0)$ is independent of the choice of the general point $P\in C$: take a general hyperplane section $H\subset X$ which intersects $C$ transversely at a point $P'$ in an analytic neighborhood of $P$, then $P'\in H$ is locally isomorphic to $P_0\in S_C$ via the natural projection $\mathbb{A}^1\times S_C\to S_C$, which implies that $c_{P'}(D|_H)=c_{P_0}(D_0)$. Then the independency follows by varying $P$ and $P'$ along $C$. 
\end{defn}

\begin{ex}\label{ex cc A123}
    Let $(X, A)$ be an {unmodifiable} canonical Fano $3$-fold and let $C\subset \Sing(X)$ be a crepant curve of type $\mathsf{A}_k$ for some $k\leq 3$. Then by Corollary~\ref{cor A primitive},
    \[
    c_C(A)=
        \begin{dcases}
            -\frac{1}{4} & \text{if } C \text{ is of type } \mathsf{A}_1;\\
             -\frac{1}{3} & \text{if } C \text{ is of type } \mathsf{A}_2;\\
             -\frac{3}{8} & \text{if } C \text{ is of type } \mathsf{A}_3.
        \end{dcases}    
    \]
    More generally,
    \[
    c_C(sA)=
        \begin{dcases}
            -\frac{(\overline{s})_2(\overline{-s})_2}{4} & \text{if } C \text{ is of type } \mathsf{A}_1;\\
             -\frac{(\overline{s})_3(\overline{-s})_3}{6} & \text{if } C \text{ is of type } \mathsf{A}_2;\\
             -\frac{(\overline{s})_4(\overline{-s})_4}{8} & \text{if } C \text{ is of type } \mathsf{A}_3.
        \end{dcases}    
    \]
    Here for a positive integer $r$ and an integer $a$, the symbol $(\overline{a})_r$ means the smallest non-negative residue of $a$ mod $r$. 
\end{ex}

\subsection{Reid's Riemann--Roch formula for canonical $3$-folds}
\label{sec RR3}

Let $X$ be a projective canonical $3$-fold. According to Reid \cite{reid}*{(10.2)}, there is a collection of pairs of integers (permitting weights)
\[
    B_{X}=\{(r_{i}, b_{i}) \mid i=1, \dots, s; 0<b_{i}\leq\frac{r_{i}}{2} ; b_{i} \text{ is coprime to } r_{i}\}
\]
associated to $X$, called {\it Reid's basket}, where a pair $(r, b)$ corresponds to an orbifold point of type $\frac{1}{r}(1, -1, b)$ which comes from locally deforming non-Gorenstein singularities of a terminalization of $X$. In other words, $B_X$ is identified with the collection of virtual orbifold points (a.k.a fictitious singularities) of $X$. For simplicity, we sometimes just call a pair in $B_X$ an orbifold point of $X$ and often formally write 
\[
    B_X=\{Q=(r_Q, b_Q)\}.
\]
By definition, if $Y$ is a terminalization of $X$, then $B_Y=B_X$. Denote by $\mathcal{R}_X$ the collection of $r_i$ (permitting weights) appearing in $B_X$. Note that the Gorenstein index $r_X$ of $X$ is just $\lcm\{r\mid r\in \mathcal{R}_X\}$.

Let $D$ be a $\mathbb{Q}$-Cartier integral Weil divisor on $X$. We say that $D$ satisfies {\it Reid's condition} if locally at any point $P\in X$, $D\sim iK_X$ for some integer $i$ (depending on $P$). Note that if $X$ has only terminal singularities, then $D$ always satisfies Reid's condition by \cite{kawamata-crepant}*{Corollary~5.2}. 

For $D$ satisfying Reid's condition, according to \cite{reid}*{(10.2)}, there exists an orbifold Riemann--Roch formula, called \emph{Reid's formula}:
\begin{align}\label{eq.reidformula}
\begin{split}
   {}&  \chi(X, \mathcal O_X(D))\\={}&\chi(X, \mathcal O_X)+\frac{1}{12}D\cdot (D-K_X)\cdot (2D-K_X)+\frac{1}{12}c_2(X)\cdot D+\sum_{Q\in  B_X}c_Q(D),
\end{split}
\end{align}
where the last sum runs over Reid’s basket of orbifold points. The term $c_2(X)\cdot D$ is defined to be $c_2(W)\cdot \pi^*D$ for any resolution $\pi\colon W\to X$. 

For an orbifold point $Q=(r_Q,b_Q)\in B_X$, $c_Q(D)$ is defined as the following: the local index $\mathsf{i}_{D,Q}$ of $D$ at $Q$ is defined to be 
an integer $i\geq 0$ such that $D\sim iK_X$ around $Q$ after taking terminalization and local deformation, where $\mathsf{i}_{D,Q}$ is determined up to modulo $r_Q$ and well-defined as $D$ satisfies Reid's condition; then 
\[
    c_Q(D)\coloneq -\frac{\mathsf{i}_{D,Q}(r_Q^2-1)}{12r_Q}+\sum_{j=0}^{\mathsf{i}_{D,Q}-1}\frac{(\overline{jb_Q})_{r_Q}(\overline{-jb_Q})_{r_Q}}{2r_Q}.
\]
Here for a positive integer $r$ and an integer $a$, the symbol $(\overline{a})_r$ means the smallest non-negative residue of $a$ mod $r$; we set $\sum_{j=0}^{-1}\coloneq 0$. 

In particular, by \cite{reid}*{(10.3)}, we have
\begin{align}\label{eq.range}
    c_2(X)\cdot c_1(X) + \sum_{r\in \mathcal{R}_X} \left(r-\frac{1}{r}\right)=24\chi(X, \mathcal O_X),
\end{align}
and 
\begin{align}\label{eq.RR-Fano}
    \frac{1}{2}c_1(X)^3+3\chi(X, \mathcal O_X)- \sum_{Q\in B_X}\frac{b_Q(r_Q-b_Q)}{2r_Q}= \chi(X,\mathcal{O}_X(-K_X))\in \mathbb Z.
\end{align}

\begin{rem}\label{rem.posofc2}
    For a canonical weak Fano $3$-fold $X$, $\chi(X, \mathcal{O}_X)=1$ by the Kawamata--Viehweg vanishing theorem \cite{KMM}*{Theorem~1.2.5} and $c_2(X)\cdot c_1(X)>0$ by \cite{ijl}*{Corollary~7.3}. 
\end{rem}

The following lemma is an easy and well-known consequence of \eqref{eq.reidformula}.
\begin{lem}\label{lem.DDKterminal}
    Let $X$ be a projective canonical $3$-fold and let $D$ be a $\mathbb{Q}$-Cartier integral Weil divisor on $X$ satisfying Reid's condition. Then 
    \begin{align*}
        {}&\chi(X,  \mathcal{O}_X(D))-\chi(X, \mathcal{O}_X(D+K_X))\\
        ={}&-\frac{1}{2}D^2\cdot K_X+2\chi(X, \mathcal{O}_X)-\sum_{Q\in B_X}\frac{(\overline{ \mathsf{i}_{D,Q}b_Q})_{r_Q}(\overline{-\mathsf{i}_{D,Q}b_Q})_{r_Q}}{2r_Q}.
    \end{align*}
\end{lem}

\begin{proof}
    By \eqref{eq.reidformula}, 
    \begin{align*}
        {}&\chi(X, \mathcal{O}_X(D))-\chi(X, \mathcal{O}_X(D+K_X))\\
        ={}&-\frac{1}{2}D^2\cdot K_X-\frac{1}{12}c_2(X)\cdot K_X+\sum_{Q\in B_X} (c_Q(D)-c_Q(D+K_X)).
    \end{align*} 
    For an orbifold point $Q \in B_X$, we have 
    \[
        \mathsf{i}_{D+K_X, Q}\equiv \mathsf{i}_{D, Q}+\mathsf{i}_{K_X, Q}\equiv \mathsf{i}_{D, Q}+1 \bmod r_Q.
    \]
    Hence
    \[
        c_Q(D)-c_Q(D+K_X)= \frac{r_Q^2-1}{12r_Q}-\frac{(\overline{ \mathsf{i}_{D,Q}b_Q})_{r_Q}(\overline{-\mathsf{i}_{D,Q}b_Q})_{r_Q}}{2r_Q}.
    \]
    Then the conclusion follows from  \eqref{eq.range}.
\end{proof}

\section{A Riemann--Roch formula for canonical threefolds}

In the study of Fano indices of canonical Fano $3$-folds, we have to consider integral Weil divisors which do not satisfy Reid's condition (e.g., divisors that are not Cartier in codimension $2$), for which Reid's Riemann--Roch formula is not applicable.

In this section, we establish a Riemann--Roch formula for canonical $3$-folds based on Reid's formula. The idea is to take a special transform called the Weil pullback of the divisor on $X$ to a terminalization, and then apply Reid's formula on the terminalization.

\subsection{Sequential terminalizations and Weil pullbacks}\label{sec STWP}
As we mentioned in the introduction, for a terminalization $f\colon Y\to X$ and a $\mathbb{Q}$-Cartier integral Weil divisor $D$ on $X$, in general $f^*D$ is not an integral Weil divisor. To make it an integral Weil divisor we might just consider the round down $\lfloor f^*D\rfloor$, but then the cohomologies of $\lfloor f^*D\rfloor$ usually differ from those of $D$. So in order to get a good integral Weil divisor on $Y$ from $D$, we consider a special kind of terminalization called sequential terminalization 
which is decomposed into a sequence of birational morphisms and do the pullback and round down operations via this sequence. 

\begin{defn}[Sequential terminalization]\label{def seq terminalization}
    Let $X$ be a canonical variety. A projective birational morphism $f\colon Y\to X$ together with a sequence of projective birational morphisms $f_k\colon X_k\to X_{k-1}$ for $1\leq k\leq m$ is called a {\it sequential terminalization} if the following properties are satisfied: 
    \begin{enumerate}
        \item $Y$ is $\mathbb{Q}$-factorial terminal;
        \item $X_0=X$, $X_m=Y$, $f=f_1\circ \dots \circ f_m$;  
        \item for any $1\leq k\leq m$, $f_k$ is crepant, i.e., $f_k^*K_{X_{k-1}}=K_{X_k}$;
        \item for any $1\leq k\leq m$, denote by $E_k$ the sum of exceptional divisors of $f_k$, then either $E_k=0$, or $E_k$ is a $\mathbb{Q}$-Cartier prime divisor and $-E_k$ is $f_k$-nef.
    \end{enumerate} 
    For simplicity, we just say that $f\colon Y\to X$ is a sequential terminalization. 
\end{defn}

\begin{rem}
\begin{enumerate}
    \item We can always construct a sequential terminalization by applying \cite{BCHM}*{Corollary~1.4.3}  to extract one prime divisor at a time. See also \cite{kawakita}*{Corollary~2.6}.

    \item If $f\colon Y\to X$ is a sequential terminalization and $U\subset X$ is an open subset, then $f$ naturally induces a sequential terminalization $f^{-1}(U)\to U$.
\end{enumerate}
    
\end{rem}

\begin{defn}[Weil pullback]\label{def weilpullback}
    Let $X$ be a canonical variety and let $f\colon Y\to X$ be a sequential terminalization. For any $\mathbb{Q}$-Cartier integral Weil divisor $D$ on $X$,  the {\it Weil pullback} $f^{\lfloor *\rfloor}(D)$ of $D$ on $Y$ is defined as the following: keep the notation in Definition~\ref{def seq terminalization}, take $D_0=D$, and define inductively $D_k=\lfloor f_k^*D_{k-1} \rfloor$ for $1\leq k\leq m$, and finally set $f^{\lfloor *\rfloor}(D)\coloneq D_m$. Here we can inductively show that $D_k$ is $\mathbb{Q}$-Cartier since $f_k^*D_{k-1}$ and $E_k$ are $\mathbb{Q}$-Cartier. So $f^{\lfloor *\rfloor}(D)$ is a $\mathbb{Q}$-Cartier integral Weil divisor on $Y$. 
\end{defn}

In general, Weil pullbacks are difficult to compute due to the complicated inductive definition. We illustrate the computation by the following example.

\begin{ex}
  Let $X$ be a Du Val singularity of type $\mathsf{A}_{n}$ and $f\colon Y\to X$ be its minimal resolution. We may label the exceptional curves of $f$ in the dual graph as the following:
    \[
        \dynkin [scale=2, labels={ \ell_1,\ell_2,\ell_3,\ell_n}] A{***.*}
    \]  
    There is an integral Weil divisor $\ell_0$ on $X$ such that 
    \[
        f^*\ell_0 =\ell_0' +\sum_{i=1}^n\frac{n+1-i}{n+1}\ell_i,
    \]
    where $\ell_0'$ is the strict transform of $\ell_0$ on $Y$.
    We may decompose $f$ into
    \[
    Y=X_n\to X_{n-1}\to \dots \to X_1\to X_0=X
    \]
    where $f_i\colon X_i\to X_{i-1}$ is the birational morphism extracting exactly the exceptional curve $\ell_i$ for $1\leq i\leq n$. This can be viewed as a sequential terminalization of $X$. We will compute $f^{\lfloor *\rfloor}(m\ell_0)$ for $1\leq m\leq n$. 
     
    By abusing the notation, we just use $\ell_i \, (0\leq i\leq n)$ to denote their strict transforms on all birational models, then for $0\leq j\leq i-1\leq n-1$,  we have the pullback formula:
    \begin{align}\label{eq filj}
    f_i^*\ell_j=\begin{cases}
        \ell_j & \text{ if } j<i-1;\\
        \ell_{i-1}+\frac{n-i+1}{n-i+2}\ell_{i} & \text{ if } j= i-1.
    \end{cases}
    \end{align}
    For an integer $m$, write \[f^{\lfloor *\rfloor}(m\ell_0)=m\ell_0'+\sum_{i=1}^nb_{m, i}\ell_{i},\] then by \eqref{eq filj}, we have a recursion formula:
    \begin{align}
        b_{m, 0}=m, \quad b_{m, i}=\left\lfloor \frac{(n-i+1)b_{m, i-1}}{n-i+2}\right\rfloor, 1\leq i\leq n.  \label{eq bmi}
    \end{align}
    Now suppose that $1\leq m\leq n$, then $b_{m, 1}=\left\lfloor \frac{nm}{n+1}\right\rfloor=m-1$. By \eqref{eq bmi}, it follows inductively that $b_{m, i}=\max\{m-i, 0\}$ for $1\leq i\leq n$. So for $1\leq m\leq n$, \[f^{\lfloor *\rfloor}(m\ell_0)=m\ell_0'+\sum_{i=1}^m(m-i)\ell_{i}.\]
\end{ex}

One key good property of Weil pullback is that it preserves cohomologies. 

\begin{lem}\label{lem same coh}
    Let $X$ be a canonical variety and let $f\colon Y\to X$ be a sequential terminalization. Then for any $\mathbb{Q}$-Cartier integral Weil divisor $D$ on $X$ and any integer $i\geq 0$, we have 
    \[
        H^i(Y, \mathcal O_Y(f^{\lfloor *\rfloor}(D)))=H^i(X, \mathcal O_X(D)).
    \]
\end{lem}

\begin{proof}
    Keep the notation in Definition~\ref{def seq terminalization} and Definition~\ref{def weilpullback}. It suffices to prove that \[H^i(X_k, \mathcal O_{X_k}(D_k))=H^i(X_{k-1}, \mathcal O_{X_{k-1}}(D_{k-1}))\] for every $1\leq k\leq m$ and $i\geq 0$. 

    By construction, we have $D_k=\lfloor f_k^*D_{k-1} \rfloor=f_k^*D_{k-1}-a_kE_k$ for some $0\leq a_k<1$. So $f_{k*}\mathcal O_{X_k}(D_k)=\mathcal O_{X_{k-1}}(D_{k-1})$ by \cite{nakayama-zariski}*{Lemma~II.2.11}. Moreover, $D_k$ is $f_k$-nef and $f_k$-big and $K_{X_k}$ is $f_k$-trivial as $f_k$ is crepant. Hence the Kawamata--Viehweg vanishing theorem \cite{KMM}*{Theorem~1.2.5} yields that
    \[
        R^jf_{k*}\mathcal O_{X_k}(D_k)=R^jf_{k*}\mathcal O_{X_k}(K_{X_k}+D_k-K_{X_k})=0
    \]
    for $j>0$. By a spectral sequence argument, it follows that 
    \[
        H^i(X_k, \mathcal O_{X_k}(D_k))=H^i(X_{k-1}, f_{k*}\mathcal O_{X_k}(D_k))=H^i(X_{k-1}, \mathcal O_{X_{k-1}}(D_{k-1}))
    \]
    for every $i\geq 0$. 
\end{proof}

The Weil pullback behaves not so well as it is not additive but only superadditive, namely, for $2$ divisors $D$ and $D'$, $f^{\lfloor *\rfloor}(D+D') \geq  f^{\lfloor *\rfloor}(D)+f^{\lfloor *\rfloor}(D')$. We have additivity in the following case which is useful.

\begin{lem}\label{lem pullback of D+G}
    Let $X$ be a canonical variety and let $f\colon Y\to X$ be a sequential terminalization. Let $D$ and $D'$ be $\mathbb{Q}$-Cartier integral Weil divisors on $X$. If the usual pullback $f^*D'$ is an integral Weil divisor on $Y$, then $f^{\lfloor *\rfloor}(D+D') = f^{\lfloor *\rfloor}(D)+f^*D'$; in particular, $f^{\lfloor *\rfloor}(D') = f^*D'$.
\end{lem}

\begin{proof}
    Keep the notation in Definition~\ref{def seq terminalization} and Definition~\ref{def weilpullback}. Take $D'_0=D'$, and define inductively $D'_k= f_k^*D'_{k-1} $ for $1\leq k\leq m$. Then as $f^*D'$ is an integral Weil divisor, $D'_k$ is an integral Weil divisor for $1\leq k\leq m$. This implies that 
    \[
        \lfloor f_k^*(D_{k-1}+D'_{k-1}) \rfloor=\lfloor f_k^*D_{k-1} +f_k^*D'_{k-1}\rfloor=\lfloor f_k^*D_{k-1} \rfloor +f_k^*D'_{k-1}=D_{k}+D'_k 
    \] 
    for $1\leq k\leq m$. So we get $f^{\lfloor *\rfloor}(D+D') = f^{\lfloor *\rfloor}(D)+f^*D'$.
\end{proof}
 
To conclude this subsection, we prove the following proposition which determines the difference between certain intersection numbers of the Weil pullback and the original divisor. It is an application of Reid's Riemann--Roch formula for surfaces.  

\begin{prop}\label{prop.rr.new}
    Let $X$ be a projective canonical $3$-fold and let $f\colon Y\to X$ be a sequential terminalization. Let $D$ be a $\mathbb{Q}$-Cartier integral Weil divisor on $X$. Then for any $\mathbb{Q}$-Cartier $\mathbb{Q}$-divisor $H$ on $X$,
    \begin{align}\label{eq D2H}
        (f^{\lfloor *\rfloor}(D))^2\cdot f^*H-D^2\cdot H=  \sum_{C\subset \Sing (X)}2(H\cdot C)c_C(D).
    \end{align}
\end{prop}

\begin{proof}  
    By linearity, it suffices to prove \eqref{eq D2H} for $H$ very ample. Take a general very ample divisor $H$ on $X$ and denote $H'\coloneq f^*H=f^{-1}_*H$ on $Y$. Denote $S\coloneq f^{\lfloor *\rfloor}(D)$. As $H$ is Cartier, by Lemma~\ref{lem pullback of D+G}, 
    \[
        f^{\lfloor *\rfloor}(D-H)={S}-H'.
    \] 
    So by Lemma~\ref{lem same coh},
    \[
        \chi(X, \mathcal{O}_X(D))-\chi(X, \mathcal{O}_X(D-H))=\chi(Y, \mathcal{O}_Y(S))-\chi(Y, \mathcal{O}_Y(S-H')),
    \]
    which implies that
    \begin{align}
        \chi(H, \mathcal{O}_H(D|_H))=\chi(H', \mathcal{O}_{H'}(S|_{H'})).\label{eq S-D1}
    \end{align}     
    As $H$ is a surface with Du Val singularities and $H'$ is smooth by Bertini's theorem, by \eqref{eq RR duval} and the usual Riemann--Roch formula, we have
    \begin{align}
        \chi(H, \mathcal{O}_H(D|_{H}))={}& \chi(H, \mathcal{O}_H)+\frac{1}{2}D|_H\cdot (D|_H-K_H)+\sum_{Q\in \Sing (H)}c_{Q}(D|_H); \label{eq S-D2}\\
        \chi(H', \mathcal{O}_{H'}(S|_{H'}))={}&\chi(H', \mathcal{O}_{H'})+\frac{1}{2}S|_{H'}\cdot (S|_{H'}-K_{H'}).\label{eq S-D3}
    \end{align}
    As $f_*S=D$, by the projection formula, we have
    \begin{align}
        (D|_H\cdot K_H)= (D\cdot (K_X+H)\cdot H)=(S\cdot (K_Y+H')\cdot H')=(S|_{H'}\cdot K_{H'}).\label{eq S-D4}
    \end{align} 
    Also we have $\chi(H, \mathcal{O}_{H})=\chi(H', \mathcal{O}_{H'})$ as  $H'\to H$ is birational and $H$ has only rational singularities. So combining \eqref{eq S-D1}--\eqref{eq S-D4}, we have
    \[
        S^2\cdot H'- D^2\cdot H= \sum_{Q\in \Sing (H)}2c_{Q}(D|_H).
    \]
    Note that by Bertini's theorem, the singularities of $H$ are exactly those points in $H\cap C$ for all crepant curves $C\subset\Sing(X)$, so by Definition~\ref{def ccD},
    \[
        \sum_{Q\in \Sing (H)}c_{Q}(D|_H)=\sum_{C\subset \Sing(X)}\sum_{Q\in H\cap C}c_{Q}(D|_H)=\sum_{C\subset \Sing(X)}(H\cdot C)c_{C}(D).
    \]
    This proves \eqref{eq D2H}.
\end{proof}

\subsection{A Riemann--Roch formula}

After the preparation of \S\,\ref{sec STWP}, we are able to prove the main theorem of this section, which is a Riemann--Roch formula (in a weak form) for integral 
Weil divisors on canonical $3$-folds.

\begin{thm}\label{thm.rr.new}
    Let $X$ be a projective canonical $3$-fold and let $f\colon Y\to X$ be a sequential terminalization. Let $D$ be a $\mathbb{Q}$-Cartier integral Weil divisor on $X$. Then
    \begin{align*}
        {}&\chi(X, \mathcal{O}_X(D))-\chi(X, \mathcal{O}_X(D+K_X))\\
        ={}&-\frac{1}{2}D^2\cdot K_X+2\chi(X, \mathcal{O}_X)+\sum_{C\subset \Sing(X)}(-K_X\cdot C)c_C(D)\\{}&-\sum_{Q\in  B_X}\frac{(\overline{ \mathsf{i}_{f^{\lfloor *\rfloor}(D), Q} b_Q})_{r_Q}(\overline{-\mathsf{i}_{f^{\lfloor *\rfloor}(D),Q} b_Q})_{r_Q}}{2r_Q}.
    \end{align*} 
    Here recall that $\mathsf{i}_{f^{\lfloor *\rfloor}(D), Q}$ makes sense for the integral Weil divisor $f^{\lfloor *\rfloor}(D)$ on $Y$ as $B_Y=B_X$ and $Y$ is terminal.
\end{thm}

\begin{proof}
    Denote $S\coloneq f^{\lfloor *\rfloor}(D)$. As $f$ is crepant, by Lemma~\ref{lem pullback of D+G}, 
    \[
        f^{\lfloor *\rfloor}(D+K_X)={S}+K_Y.
    \] 
    So by Lemma~\ref{lem same coh} and Lemma~\ref{lem.DDKterminal},
    \begin{align*}
        {}&\chi(X, \mathcal{O}_X(D))-\chi(X, \mathcal{O}_X(D+K_X))\\
        ={}&\chi(Y, \mathcal{O}_Y(S))-\chi(Y, \mathcal{O}_Y(S+K_Y))\\
        ={}&-\frac{1}{2}S^2\cdot K_Y+2\chi(Y, \mathcal{O}_Y)-\sum_{Q\in B_Y}\frac{(\overline{ \mathsf{i}_{f^{\lfloor *\rfloor}(D),Q} b_Q})_{r_Q}(\overline{-\mathsf{i}_{f^{\lfloor *\rfloor}(D),Q}b_Q})_{r_Q}}{2r_Q}.
    \end{align*} 
    The conclusion follows from  Proposition~\ref{prop.rr.new} for $H=K_X$ and the fact that $\chi(Y, \mathcal{O}_Y)=\chi(X, \mathcal{O}_X)$ as $X$ has rational singularities.
\end{proof}

\begin{rem}
    In Theorem~\ref{thm.rr.new}, by Lemma~\ref{lem same coh} and \eqref{eq.reidformula} we can get a more general form of Riemann--Roch formula:
    \begin{align*} 
        \chi(X, \mathcal O_X(D))={}&\chi(X, \mathcal O_X)+\frac{1}{12}f^{\lfloor *\rfloor}(D)\cdot (f^{\lfloor *\rfloor}(D)-K_Y)\cdot (2f^{\lfloor *\rfloor}(D)-K_Y)\\{}&+\frac{1}{12}c_2(Y)\cdot f^{\lfloor *\rfloor}(D)+\sum_{Q\in  B_Y}c_Q(f^{\lfloor *\rfloor}(D)).
    \end{align*}
    But in order to make use of this formula, we need to understand $f^{\lfloor *\rfloor}(D)$ in more detail. For example, we need to calculate the intersection numbers $(f^{\lfloor *\rfloor}(D))^3$ and $c_2(Y)\cdot f^{\lfloor *\rfloor}(D)$. This requires a better understanding of singular points of $X$ rather than generic points of crepant curves. So in Theorem~\ref{thm.rr.new} we consider the form $\chi(X, \mathcal{O}_X(D))-\chi(X, \mathcal{O}_X(D+K_X))$ which cancels those terms that are difficult to compute. 

    One can compare this formula with \cite{BuckleySZ}*{Theorem~2.1} where a Riemann--Roch formula for $3$-folds with only quotient singularities (but not necessarily canonical singularities) is given. 
\end{rem}

As a corollary, we have the following theorem which can be effectively applied to problems on Fano indices of canonical Fano $3$-folds.

\begin{thm}\label{thm.rr.sA}
    Let $X$ be a canonical Fano $3$-fold and let $f\colon Y\to X$ be a sequential terminalization. Suppose that $-K_X\sim_{\mathbb{Q}} qA$ for some positive rational number $q$ and some ample integral Weil divisor $A$.  Then for any integer $s$ with $0<s<q$,
    \begin{align*}
        h^0(X,\mathcal O_X(sA))={}& 
        -\frac{s^2}{2}A^2\cdot K_X+2+\sum_{C\subset \Sing(X)}(-K_X\cdot C)c_C(sA)\\{}&-\sum_{Q\in B_X}\frac{(\overline{\mathsf{i}_{f^{\lfloor *\rfloor}(sA), Q}b_Q})_{r_Q}(\overline{-\mathsf{i}_{f^{\lfloor *\rfloor}(sA), Q}b_Q})_{r_Q}}{2r_Q}.
    \end{align*} 
\end{thm}

\begin{proof}
    Since $-K_X$ is ample, all higher cohomologies of $sA$ and $sA+K_X$ vanish 
    by the Kawamata--Viehweg vanishing theorem \cite{KMM}*{Theorem~1.2.5}. Also $h^0(X, \mathcal{O}_X(sA+K_X))=0$ as $sA+K_X\sim_{\mathbb{Q}} (s-q)A$ and $s<q$. So we have 
    \begin{align*}
        {}&\chi(X, \mathcal{O}_X)=h^0(X, \mathcal{O}_X)=1,\\
        {}&\chi(X, \mathcal{O}_X(sA))=h^0(X, \mathcal{O}_X(sA)),\\
        {}&\chi(X, \mathcal{O}_X(sA+K_X))=0.
    \end{align*}
    Then the conclusion follows directly from Theorem~\ref{thm.rr.new} for $D=sA$.
\end{proof}

\begin{rem}
    In the applications of Theorem~\ref{thm.rr.sA}, the technical difficulty is to determine the contributions of singularities: we have to handle values of $(-K_X\cdot C)$, $c_C(sA)$, and ${\mathsf{i}_{f^{\lfloor *\rfloor}(sA), Q}}$. We will treat the first two in \S\,\ref{sec.LB} and \S\,\ref{sec det C}. On the other hand, determining ${\mathsf{i}_{f^{\lfloor *\rfloor}(sA), Q}}$ is much harder due to the complicated definition of Weil pullback; the problem is that ${\mathsf{i}_{f^{\lfloor *\rfloor}(sA), Q}}$ is not linear in $s$ (see Lemma~\ref{lem pullback of D+G}) and we do not have a good way to calculate it due to the lack of classification of singularities. However, as we can usually determine $B_X$, ${\mathsf{i}_{f^{\lfloor *\rfloor}(sA), Q}}$ takes value in a finite set and often we can determine it from the integrality constraints discussed in \S\,\ref{sec int cons}.
\end{rem}

\subsection{Integrality constraints}\label{sec int cons}

An important phenomenon for the Riemann--Roch formula is that the Euler characteristic is always an integer while the contribution from singularities is only a sum of rational numbers. This provides many integrality constraints which give strong restrictions on singularities. 
In this subsection, we explore several integrality constraints that will be used later.

We will frequently use the fact that for a positive integer  $r$ and an integer $a$, 
\[
    a(r-a)\equiv (\overline{a})_r(\overline{-a})_r\bmod 2r.
\] 

As a direct consequence of Theorem~\ref{thm.rr.new}, we have the following integrality constraint.

\begin{cor}\label{cor RR integer general}
    Let $X$ be a projective canonical $3$-fold and let $f\colon Y\to X$ be a sequential terminalization. Let $D$ be a $\mathbb{Q}$-Cartier integral Weil divisor on $X$. Then for any positive integer $r'$,  
    \begin{align*}
        -\frac{r'}{2}D^2\cdot K_X +\sum_{C\subset \Sing(X)}(-r'K_X\cdot C)c_C(D)-\sum_{Q\in  B_X}\frac{r' \mathsf{i}_{f^{\lfloor *\rfloor}(D), Q} b_Q(r_Q-\mathsf{i}_{f^{\lfloor *\rfloor}(D), Q}b_Q)}{2r_Q}\in \mathbb{Z}.
    \end{align*} 
\end{cor}

\begin{rem}\label{rem ignore part}
    In Corollary~\ref{cor RR integer general}, usually some terms are ignorable by the following facts:

    Suppose that $C$ is of type $\mathsf{A}_{j_C-1}$ and $(-r'K_X\cdot C)\in \mathbb{Z}$, then  $(-r'K_X\cdot C)c_C(D)\in \mathbb{Z}$ holds under one of the following conditions:
    \begin{enumerate}
        \item  $2\nmid j_C$ and $j_C\mid (-r'K_X\cdot C)$;
    
        \item  $2\mid j_C$ and $2j_C\mid (-r'K_X\cdot C)$.
    \end{enumerate}
    For a positive integer $r$ and an integer $a$, $\frac{r'a(r-a)}{2r}\in \mathbb{Z}$ holds under one of the following conditions:
    \begin{enumerate}
        \item $2\nmid r$ and $r\mid r'$; 
    
        \item $2\mid r$ and $2r\mid r'$.
    \end{enumerate}
\end{rem}

The following two corollaries are special cases of Corollary~\ref{cor RR integer general}.

\begin{cor}\label{cor JA integer}  
    Let $X$ be a projective canonical $3$-fold and let $f\colon Y\to X$ be a sequential terminalization. Let $D$ be a $\mathbb{Q}$-Cartier integral Weil divisor on $X$ which is Cartier in codimension $2$. Then
    \begin{align*}
        -\frac{1}{2}D^2\cdot K_X-\sum_{Q\in  B_X}\frac{{ \mathsf{i}_{f^{\lfloor *\rfloor}(D),Q} b_Q}({r_Q-\mathsf{i}_{f^{\lfloor *\rfloor}(D),Q}b_Q})}{2r_Q} \in \mathbb{Z}.
    \end{align*} 
\end{cor}

\begin{proof}
    As $D$ is Cartier in codimension $2$, $c_C(D)=0$ for any crepant curve $C$. So we get the conclusion from Corollary~\ref{cor RR integer general} for $r'=1$.
\end{proof}

\begin{cor}\label{cor.canpart} 
    Let $X$ be a projective canonical $3$-fold and let $f\colon Y\to X$ be a sequential terminalization. Let $D$ be a $\mathbb{Q}$-Cartier integral Weil divisor on $X$. Then 
    \[
        -r_XD^2\cdot K_X+ \sum_{C\subset \Sing(X)}2(-r_XK_X\cdot C)c_C(D)
        \in \mathbb Z.
    \]
\end{cor}

\begin{proof}
    This follows from Corollary~\ref{cor RR integer general} for $r'=2r_X$ and Remark~\ref{rem ignore part}. 
\end{proof}

Besides the integrality constraint provided by Theorem~\ref{thm.rr.new}, we also have integrality constraints provided by Reid's Riemann--Roch formula for exceptional divisors on a (not necessarily sequential) terminalization. Those constraints will be essentially used to estimate the intersection number of a crepant curve with the anti-canonical divisor in \S\,\ref{sec.LB}.
 
\begin{prop}\label{prop.KC in Z}
    Let $X$ be a projective canonical $3$-fold and let $f\colon Y\to X$ be a terminalization. Let $C\subset \Sing(X)$ be a crepant curve with non-split crepant divisors and let $E, E'$ be prime $f$-exceptional divisors centered at $C$. Then we have 
    \begin{align}\label{eq.KCE}
        (K_X\cdot C)-\sum_{Q\in  B_X}\frac{{\mathsf{i}_{E,Q}b_Q}(r_Q-{\mathsf{i}_{E,Q}b_Q})}{2r_Q}\in \mathbb Z,\\
        -(\ell_E\cdot \ell_{E'})\cdot(K_X\cdot C)+\sum_{Q\in B_X}\frac{{\mathsf{i}_{E, Q}\mathsf{i}_{E',Q}b_Q^2}}{r_Q}\in \mathbb Z.\notag
    \end{align} 
\end{prop}

\begin{proof} 
    Applying Lemma~\ref{lem.DDKterminal} to $E$, $E'$, and $E+E'$ on $Y$ respectively, we have 
    \begin{align}
        & -\frac{1}{2}E^2\cdot K_Y-\sum_{Q\in B_X}\frac{{i_Qb_Q}(r_Q-i_Qb_Q)}{2r_Q}\in \mathbb Z,\label{eq 4.14-1}\\
        & -\frac{1}{2}E'^2\cdot K_Y-\sum_{Q\in B_X}\frac{{i'_Qb_Q}(r_Q-i'_Qb_Q)}{2r_Q}\in \mathbb Z,\notag\\
        & -\frac{1}{2}(E+E')^2\cdot K_Y-\sum_{Q\in B_X}\frac{{(i_Qb_Q+i'_Qb_Q)}(r_Q-i_Qb_Q-i'_Qb_Q)}{2r_Q}\in \mathbb Z,\notag
    \end{align}
    where $i_Q\coloneq\mathsf{i}_{E, Q}$ and $i'_Q\coloneq\mathsf{i}_{E', Q}$. It follows that 
    \begin{align}\label{eq 4.14-2}
        -(E\cdot E'\cdot K_Y)+\sum_{Q\in B_X}\frac{{i_Qi'_Qb_Q^2}}{r_Q}\in \mathbb Z.  
    \end{align}
    Then we get the conclusion by \eqref{eq 4.14-1}, \eqref{eq 4.14-2}, and Lemma~\ref{lem EEK=EECK}.
\end{proof}

Similarly, for an exceptional divisor centered over a crepant point, we have the following integrality constraint, which is closely related to Kawakita's result on crepant points \cite{kawakita}. 

\begin{prop}\label{prop crepant center in Z}
    Let $X$ be a projective canonical $3$-fold and let $f\colon Y\to X$ be a terminalization. Let $E$ be a prime $f$-exceptional divisor centered at a crepant point. Then we have 
    \begin{align*} 
        \sum_{Q\in B_X}\frac{{\mathsf{i}_{E,Q}b_Q}(r_Q-{\mathsf{i}_{E,Q}b_Q})}{2r_Q}\in \mathbb Z.
    \end{align*}  
\end{prop}

\begin{proof}
    Applying Lemma~\ref{lem.DDKterminal} to $E$ on $Y$, we have 
    \begin{align*}
        & -\frac{1}{2}E^2\cdot K_Y-\sum_{Q\in B_X}\frac{{\mathsf{i}_{E,Q}b_Q}(r_Q-{\mathsf{i}_{E,Q}b_Q})}{2r_Q} \in \mathbb Z.
    \end{align*}
    Note that $f(E)$ is a point, hence $E^2\cdot K_Y =E^2\cdot f^*K_X =0$ by the projection formula.
\end{proof}

\begin{ex}\label{ex.2222}
    In Proposition~\ref{prop crepant center in Z}, suppose that $f(E)=P$ and the Gorenstein index $r_P$ at $P$ is $2$, then by \cite{kawakita}*{Theorem~1.1, Table~2} (see also Remark~\ref{rem kawakita}), $B_X$ contains $4$ orbifold points $\{4\times (2, 1)\}$, which we denote by $Q_i$ $(1\leq i\leq 4)$. We assume further that $r_Q>2$ for any orbifold point $Q\in B_X$ such that $Q\neq Q_i$ $(1\leq i\leq 4)$. Under this assumption, for $Q\neq Q_i$ $(1\leq i\leq 4)$, we have $\mathsf{i}_{E,Q}=0$ as $f(E)=P$ and $r_Q>r_P$. Then by Proposition~\ref{prop crepant center in Z}, we have
    \[
        \sum_{i=1}^4\frac{\mathsf{i}_{E,Q_i}(2-\mathsf{i}_{E,Q_i})}{4}\in \mathbb{Z},
    \]
    which implies that 
    \[
        \mathsf{i}_{E,Q_1}\equiv \mathsf{i}_{E,Q_2}\equiv \mathsf{i}_{E,Q_3}\equiv \mathsf{i}_{E,Q_4}\bmod 2.
    \]
\end{ex}

\section{Lower bounds of degrees of crepant curves}\label{sec.LB}

In order to apply Theorem~\ref{thm.c1c2 diff} and Theorem~\ref{thm.rr.sA}, one crucial technical problem is to determine the intersection number $(-K_X\cdot C)$ for any crepant curve $C\subset \Sing(X)$. Sometimes we just call $(-K_X\cdot C)$ the {\it degree} of the crepant curve $C$ (with respect to $-K_X$). 
There is a naive lower bound $(-K_X\cdot C)\geq \frac{1}{r_X}$
as $-r_XK_X$ is Cartier and ample, but this is far from being sufficient for our purpose. After calculating many examples, we found that often $(-K_X\cdot C)\geq 1$ holds (see Example~\ref{ex compute -rKC}), which suggests that the lower bound of  $(-K_X\cdot C)$ could be greatly improved. 

The main goal of this section is to give a good estimate on the lower bound of $(-K_X\cdot C)$ (see Theorem~\ref{thm c>=LB}). The key idea is to use the integrality constraint in Proposition~\ref{prop.KC in Z} to control the denominator of $(-K_X\cdot C)$ and the proof is elementary. 
 
First we set up the notation for this section.

\begin{set}\label{set KC}
    \begin{enumerate}
        \item Let $X$ be a canonical Fano $3$-fold. Denote  Reid's basket of $X$ by 
        \[
            B_X=\{(r_1, b_1), \dots,(r_s, b_s)\}
        \]
        and 
        \[
            \mathcal{R}\coloneq \mathcal{R}_X=\{r_1, \dots, r_s\}.
        \]
        Recall that the Gorenstein index of $X$ is $r_X=\lcm\{r_i\in \mathcal{R}\}$. Recall that we have
        \begin{align}\label{eq ri<24}
            \sum_{j=1}^s\left(r_j-\frac{1}{r_j}\right)<24 
        \end{align}
        by \eqref{eq.range} and Remark~\ref{rem.posofc2}.

        \item For a prime number $p$, denote by $\nu_p\colon \mathbb{Q}\to \mathbb{Z}\cup \{\infty\}$ the $p$-adic valuation on $\mathbb{Q}$. We use $n_{p^e}$ to denote the number of $r_i$'s in $\mathcal{R}$ with $\nu_p(r_i)=e$. Note that by \eqref{eq ri<24}, $n_{p^e}=0$ if $p^e>23$.

        \item Let $C\subset \Sing(X)$ be a crepant curve with non-split crepant divisors. Denote $c\coloneq (K_X\cdot C)$, where $r_Xc$ is an integer. 
    
        \item Let $f\colon Y\to X$ be a terminalization. Denote $l\coloneq e'_C-1$. By the definition of $e'_C$ (Definition~\ref{def ecgc}) and Definition~\ref{def nonsplit}, there exist prime exceptional divisors $E_1, \dots, E_{l}$ on $Y$ centered at $C$ such that $\ell_{E_1}, \dots, \ell_{E_{l}}$ form a chain of $(-2)$-curves. Denote by $\mathbb{I}=((\ell_{E_k}\cdot \ell_{E_{k'}}))_{1\leq k,k'\leq l}$ the intersection matrix, namely,
        \[
            \mathbb{I}_{kk'}=
            \begin{cases}
                -2 & \text{if } k=k';\\
                1 & \text{if } |k-k'|=1;\\
                0 & \text{otherwise}.
            \end{cases}
        \]

        \item  For $1\leq k\leq l$ and $1\leq j\leq s$, denote $a_{kj}\coloneq \mathsf{i}_{E_k, (r_j, b_j)}b_j$.  
    \end{enumerate}     
\end{set}

By Proposition~\ref{prop.KC in Z}, we have the following.

\begin{lem}\label{lem rational integer c}
    Keep Setting~\ref{set KC}. Then for any $1\leq k, k'\leq {l}$, 
    \begin{align}
        c-\sum_{j=1}^s\frac{{a_{kj}}(r_j-{a_{kj}})}{2r_j}\in \mathbb Z; \label{eq.KE2}\\
        -c\mathbb{I}_{kk'} +\sum_{j=1}^s\frac{a_{kj}a_{k'j}}{r_j}\in \mathbb Z.\label{eq.KEE'/r}
    \end{align}  
\end{lem} 

Our main goal is to give a lower bound for $-c$. The basic idea is that by the integrality constraints in Lemma~\ref{lem rational integer c}, we can show that the denominator of $c$ will not be as large as $r_X$ if $l$ is large enough. We will study the denominator of $c$ by investigating the $p$-adic valuation of $c$ for each prime number $p$. To see how this idea works, we have the following lemma by a linear algebra argument. 
 
\begin{lem}\label{lem ADA}
    Let $l, N$ be positive integers and let $p$ be a prime number. Let $a_{kj}, b_{kj}\in \mathbb{Z}$ and $c, t_j\in \mathbb{Q}$ for $1\leq k\leq l$ and $1\leq j\leq N$. Suppose that for any $1\leq k,k'\leq l$,
    \begin{align}
        -c\mathbb{I}_{kk'} +\sum_{j=1}^Nt_ja_{kj}b_{k'j}\in \mathbb{Z}.\label{eq lem Ictaa}
    \end{align}
    Suppose that $\nu_p(t_j)\geq -e_0$ for some integer $e_0>0$ and for any $1\leq j\leq N$, and denote \[n\coloneq |\{1\leq j\leq N\mid \nu_p(t_j)=-e_0\}|.\] If $l\geq n+1 +\mathsf{if}(p\mid n+2)$, then $\nu_p(c)\geq -e_0+1$. 
\end{lem}

Here we define the function $\mathsf{if}$ as the following: for a statement $S$, $\mathsf{if}(S)=1$ (resp., $0$) if $S$ is true (resp., false). 

\begin{proof}
    First we consider the case that $n=0$. This implies that $\nu_p(t_j)\geq -e_0+1$ for any $1\leq j\leq N$. If $p>2$, then \eqref{eq lem Ictaa} for $(k, k')=(1,1)$ implies that $\nu_p(c)\geq -e_0+1$. If $p=2$, then $l\geq 2$ and hence \eqref{eq lem Ictaa} for $(k, k')=(1,2)$ implies that $\nu_p(c)\geq -e_0+1$.

    Then we consider the case that $n>0$. We may just assume that $l=n+1 +\mathsf{if}(p\mid n+2)$. Note that $l\geq n+1$ and $p\nmid l+1$. 

    We may take an integer $r$ such that $\nu_p(r)=e_0$ and $rt_j\in \mathbb{Z}$ for $1\leq j\leq N$. Then by \eqref{eq lem Ictaa} for $(k,k')=(1,2)$,  we have $rc\in \mathbb{Z}$. 

    Without loss of generality, we may assume that $\nu_p(t_j)=-e_0$ for $1\leq j\leq n$ and $\nu_p(t_j)>-e_0$ for $j>n$. Then \eqref{eq lem Ictaa} implies that for any $1\leq k,k'\leq l$,
    \begin{align}
        -rc\mathbb{I}_{kk'} +\sum_{j=1}^nrt_ja_{kj}b_{k'j}\equiv 0 \bmod p.\label{eq lem Ictaa p}
    \end{align}

    For an integer $a$, denote by $\overline{a}$ its residue in the finite field $\mathbb{F}_p$. Consider the matrices $A\coloneq(\overline{a_{kj}})_{\substack{1\leq k\leq l\\1\leq j\leq j}}, B\coloneq(\overline{b_{kj}})_{\substack{1\leq k\leq l\\1\leq j\leq j}}\in \mathbb{F}_{p}^{l\times n}$, and $D\coloneq\diag\{\overline{rt_{1}}, \dots,  \overline{rt_{n}}\}\in \mathbb{F}_{p}^{n\times n}$, then \eqref{eq lem Ictaa p} is equivalent to   
    \[
        ADB^{\mathsf{T}}= \overline{rc}\cdot \mathbb{I}\in \mathbb{F}_p^{l\times l}.
    \] 
    As $l\geq n+1$, this matrix is singular, which implies that $\det(\overline{rc}\cdot\mathbb{I})=0\in \mathbb{F}_p$. On the other hand, $\det(\overline{rc}\cdot\mathbb{I})= (-1)^l(l+1)\overline{rc}^l$. So we get $\overline{rc}=0\in \mathbb{F}_p$ as $p\nmid l+1$. This implies that $\nu_p(c)\geq -\nu_p(r)+1=-e_0+1$. 
\end{proof}

\begin{rem}
  We usually apply Lemma~\ref{lem ADA} to Setting~\ref{set KC} by taking $t_j=\frac{1}{r_j}$ and $b_{kj}=a_{kj}$ for $1\leq k\leq l$ and $1\leq j\leq s$ as in Setting~\ref{set KC}(5). In this case $n$ is just $n_{p^{e_0}}$. 
\end{rem}

We can directly apply Lemma~\ref{lem ADA} to Setting~\ref{set KC} to get the following consequence.

\begin{cor}\label{cor -e+1}
    Keep Setting~\ref{set KC}. Fix a prime number $p$. Suppose that $\nu_p(r_X)=e>0$. If ${l}\geq n_{p^e}+1+\mathsf{if}(p\mid n_{p^e}+2)$, then $\nu_p(c)\geq -e+1$. 
\end{cor}
 
\begin{proof}
    This directly follows from Lemma~\ref{lem ADA} and \eqref{eq.KEE'/r}. 
\end{proof}

In particular, Corollary~\ref{cor -e+1} shows that if $\nu_p(r_X)=1$ and $l$ is large enough in Setting~\ref{set KC}, then $\nu_p(c)\geq 0$. By \eqref{eq ri<24}, we have $\nu_p(r_X)\leq 1$ if $p>3$. So Corollary~\ref{cor -e+1} is sufficient to handle prime numbers $p>3$. We will treat $p=2,3$ in more detail in the following two propositions. 

\begin{prop}\label{prop p=3}
    Keep Setting~\ref{set KC}. Suppose that $n_9>0$. Then the following assertions hold.
    \begin{enumerate}
        \item $n_9\leq 2$ and $\nu_3(r_X)=2$.
    
        \item  If $l\geq 3$, then  $\nu_3(c)\geq -1$.  
    
        \item  If  $l\geq 3$ and $l\geq n_3 +1+\mathsf{if}(3 \mid n_3+2)$, then  $\nu_3(c)\geq 0$.  
    \end{enumerate}
\end{prop}

\begin{proof}
    (1) By \eqref{eq ri<24}, it is clear that $n_9\leq 2$ and $\nu_{3}(r_X)<3$. 

    (2) This follows from Corollary~\ref{cor -e+1} for $p^e=9$. 

    (3) Suppose that $l\geq 3$. We may assume that $\nu_{3}(r_j)=2$ for $1\leq j\leq n_{9}$. Then $3\mid \frac{r_X}{r_j}$ for $j>n_9$, and $3\mid r_Xc$ as $\nu_3(c)\geq -1$ by (2). So multiplying \eqref{eq.KEE'/r} by $r_X$ and modulo by $3$, we have for any $1\leq k=k'\leq l$, 
    \begin{align}
        \sum_{j=1}^{n_{9}}\frac{r_X}{r_j} a_{kj}^2\equiv 0 \bmod 3. \label{eq mod 9=0}
    \end{align}  

    \begin{claim}\label{claim 3|a}
        $3\mid a_{kj}$ for any $1\leq k\leq l$ and any $1\leq j\leq n_9$. 
    \end{claim} 
    \begin{proof}
        Recall that $3\nmid \frac{r_X}{r_j}$ for $1\leq j\leq n_9$ by definition.

        If $n_9=1$, then \eqref{eq mod 9=0} implies that $3\mid a_{k1}$ for any $1\leq k\leq l$. 

        If $n_9=2$, then $r_1=r_2=9$ by \eqref{eq ri<24}. So \eqref{eq mod 9=0} implies that for any $1\leq k\leq l$, 
        \[ 
            a_{k1}^2+a_{k2}^2\equiv 0 \bmod 3. 
        \]
        Then it follows that $a_{k1} \equiv a_{k2}\equiv 0 \bmod 3$.
    \end{proof}
    By Claim~\ref{claim 3|a}, we may rewrite \eqref{eq.KEE'/r} as 
    \[
        -c\mathbb{I}_{kk'} +\sum_{j=1}^{n_9}\frac{\frac{a_{kj}}{3}\frac{a_{k'j}}{3}}{\frac{r_j}{9}}+\sum_{j=n_9+1}^s\frac{a_{kj}a_{k'j}}{r_j}\in \mathbb Z.
    \]
    Here $\frac{a_{kj}}{3}, \frac{a_{k'j}}{3},\frac{r_j}{9}\in\mathbb{Z}$ and $\nu_3(\frac{r_j}{9})=0$. 
    By Lemma~\ref{lem ADA} for $e_0=1$ and $n=n_3$, we know that $\nu_3(c)\geq 0$ as long as $l\geq n_3 +1+\mathsf{if}(3 \mid n_3+2)$. 
\end{proof}

For the case $p=2$, it gets more complicated and we have to split the discussion into many cases for later applications. But a good news is that we can apply \eqref{eq.KE2} to get a better estimate than other prime numbers. 

\begin{prop}\label{prop p=2}
    Keep Setting~\ref{set KC}. Suppose that $\nu_2(r_X)=e>0$ and $l\geq 2$. Then the following assertions hold.

    \begin{enumerate}
        \item $e\leq 4$ and $2n_{16}+n_8\leq 3$.
        
        \item  If  $n_{2^{e}}\leq 2$, then $\nu_2(c)\geq -e+1$. 

    \item If $n_{16}=1$, then
    \[
        \nu_2(c)\geq \begin{cases}
        -3& \text{if } n_8\neq 0;\\
        -2& \text{if } n_8= 0 \text{ and } n_4\neq 0;\\
        -1 &\text{if } n_8=n_4=0.
        \end{cases}
    \]

    \item If $n_{16}=0$, $n_8\leq 1$, and $l\geq 2\lfloor\frac{n_4}{2}\rfloor+2$, then $\nu_2(c)\geq -1$. Furthermore, $\nu_2(c)\geq 0$ if moreover one of the following holds:
    \begin{enumerate}
        \item $n_4\leq 1$ and $l\geq 2\lfloor\frac{n_2+n_8}{2}\rfloor+2$;
        \item $n_4=2$, $\{4,4\}\subset \mathcal{R}$, and $l\geq 2\lfloor\frac{n_2+n_8+2}{2}\rfloor+2$;
        \item $n_4=3$, $\{4,4,4\}\subset \mathcal{R}$, and $n_2=n_8=0$. 
    \end{enumerate}

    \item If $n_8=2$, then 
    \[
        \nu_2(c)\geq \begin{cases}
        -2 & \text{if } n_{4}\neq 0;\\
        -1 & \text{if } n_4=0 \text{ and } n_2\neq 0;\\
        0 &\text{if } n_{2}=n_4=0. \\
        \end{cases}
    \]
    \end{enumerate}
\end{prop}

Here we recall that for any integer $n$, $n +1+\mathsf{if}(2 \mid n+2)=2\lfloor\frac{n}{2}\rfloor+2$.

\begin{proof}
    (1) By \eqref{eq ri<24}, it is clear that  $\nu_{2}(r_X)\leq 4$, $n_{16}\leq 1$, and $n_{8}\leq 3$. Moreover, if $n_{16}=1$, then $n_8\leq 1$. 

    \medskip 

    (2) Suppose that $n_{2^{e}}=1$. We may assume that $\nu_2(r_1)=e$
    and $\nu_2(r_j)\leq e-1$ for $j>1$. As $\nu_2(c)\geq -e$, by \eqref{eq.KE2}, $\nu_2(\frac{a_{k1}(r_1-a_{k1})}{2r_1})\geq -e$ for $1\leq k\leq l$. This means that $2\mid a_{k1}$ for $1\leq k\leq l$. By \eqref{eq.KEE'/r} for $(k,k')=(1,2)$, we get $\nu_2(c)\geq -e+1$.
 
    Suppose that $n_{2^{e}}=2$. We may assume that $\nu_2(r_1)=\nu_2(r_2)=e$ and $\nu_2(r_j)\leq e-1$ for $j>2$. As $\nu_2(c)\geq -e$, by \eqref{eq.KE2}, for any $1\leq k\leq l$, we have
    \[
        \nu_2\left(\frac{a_{k1}(r_1-a_{k1})}{2r_1}+\frac{a_{k2}(r_2-a_{k2})}{2r_2}\right)\geq -e.
    \]
    This implies that $a_{k1}\equiv a_{k2}\bmod 2$ for $1\leq k\leq l$. Then 
    \[
        \nu_2\left(\frac{a_{11}a_{21}}{r_1}+\frac{a_{12}a_{22}}{r_2}\right)\geq -e+1
    \]
    as 
    \[
        \frac{r_2}{2^e} a_{11}a_{21}+\frac{r_1}{2^e} a_{12}a_{22}\equiv  a_{11}a_{21}+ a_{12}a_{22}\equiv 2a_{11}a_{21}\equiv 0\bmod 2.
    \]
    So by \eqref{eq.KEE'/r} for $(k,k')=(1,2)$, $\nu_{2}(c)\geq -e+1$. 

    \medskip 

    (3) Suppose that $n_{16}=1$. Then $\nu_2(c)\geq -3$ by (2). So we may assume that $n_8=0$ otherwise there is nothing to prove. 

    We may assume that $r_1=16$ and denote $x_k=a_{k1}$ for $1\leq k\leq l$. Then by \eqref{eq.KE2} and \eqref{eq.KEE'/r}, for any $1\leq k, k' \leq l$,
    \begin{align}
        c- \frac{{x_{k}}(16-{x_{k}})}{32}-\sum_{j=2}^{s}\frac{{a_{kj}}(r_j-{a_{kj}})}{2r_j}\in \mathbb Z; \label{eq.KE2 16}\\
        -c\mathbb{I}_{kk'} +\frac{x_kx_{k'}}{16}+\sum_{j=2}^{s}\frac{a_{kj}a_{k'j}}{r_j}\in \mathbb Z.\label{eq.KEE'/r 16}
    \end{align}
    By the proof of (2), $2\mid x_k$ for any $1\leq k\leq l$. By \eqref{eq.KEE'/r 16} for $(k, k')=(1,2)$, we know that $\nu_2(c) \geq -2$. 

    Now suppose that $n_4=0$, then \eqref{eq.KE2 16} implies that $\nu_2(\frac{{x_{k}}(16-{x_{k}})}{32})\geq -2$ for $1\leq k\leq l$, so $4\mid x_k$ for $1\leq k\leq l$. Then by \eqref{eq.KEE'/r 16} for $(k, k')=(1,2)$, we know that  $\nu_2(c) \geq -1$. 

    \medskip 

    (4) Suppose that $n_{16}=0$ and $n_8\leq 1$. If $24\in \mathcal{R}$, then $\mathcal{R}=\{24\}$ by \eqref{eq ri<24}. By the proof of (2), we have $\nu_2(c)\geq -2$ and $2\mid a_{k1}$ for $1\leq k\leq l$. So by \eqref{eq.KEE'/r} for $(k,k')=(1,2)$, we have $-c+\frac{a_{11}a_{21}}{24}\in \mathbb{Z}$, which implies that $\nu_2(c)\geq -1$. As $c-\frac{a_{11}(24-a_{11})}{48}\in \mathbb{Z}$ by \eqref{eq.KE2}, we conclude that $4\mid a_{11}$ and hence $\nu_2(c)\geq 0$. So in the following we may assume that $24\not\in\mathcal{R}$.
   
    By possibly adding $8$ into $\mathcal{R}$, we can construct a new collection of the form
    \[
        \mathcal{R}'=\{8,r_1,r_2,\dots, r_{s'}\}\supset \mathcal{R}
    \]
    where either $s'=s$, or $s'=s-1$ and $r_s=8$. For any $1\leq k\leq l$, we define 
    \[
        x_k=\begin{cases}
        0 & \text{if } 8\notin \mathcal{R};\\
        a_{ks}& \text{if } 8 \in \mathcal{R}.
        \end{cases} 
    \]
    Then by \eqref{eq.KE2}, for any $1\leq k \leq l$,
    \begin{align}
        c-\frac{{x_{k}}(8-{x_{k}})}{16} -\sum_{j=1}^{s'}\frac{{a_{kj}}(r_j-{a_{kj}})}{2r_j}\in \mathbb Z. \label{eq.KE2 8}
    \end{align}
    By the proof of (2), $\nu_2(c)\geq -2$ and $2\mid x_k$ for $1\leq k\leq l$. We may write $x_k=2y_k$ for some integer $y_k$. 

    Then \eqref{eq.KEE'/r} can be rephrased as: for any $1\leq k, k'\leq l$,
    \begin{align}
        -c\mathbb{I}_{kk'} +\frac{y_ky_{k'}}{2} +\sum_{j=1}^{s'}\frac{a_{kj}a_{k'j}}{r_j}\in \mathbb Z.\label{eq.KEE'/r 8}
    \end{align}
    So by Lemma~\ref{lem ADA} for $e_0=2$ and $n=n_4$, we have $\nu_2(c)\geq -1$ as long as $l\geq 2\lfloor\frac{n_4}{2}\rfloor+2$. 

    From now on, we assume that $l\geq 2\lfloor\frac{n_4}{2}\rfloor+2$ and $\nu_2(c)\geq -1$.

    \medskip

    (4a) If $n_4=0$, then by applying Lemma~\ref{lem ADA} for $e_0=1$ and $n=n_2+n_8$ to \eqref{eq.KEE'/r 8}, we have $\nu_2(c)\geq 0$ as long as  $l\geq 2\lfloor\frac{n_2+n_8}{2}\rfloor+2$.  Here we just recall that $n_8\leq 1$ and if $n_8=0$ then $y_k=0$ for $1\leq k\leq l$.

    If $n_4=1$, we may assume that $r_1=4r'_1$ for some integer $r'_1$ and $4\nmid r_j$ for $j>1$. Then by \eqref{eq.KE2 8}, for any $1\leq k\leq l$,
    \begin{align*}
        c-\frac{{y_{k}}(4-{y_{k}})}{4} - \frac{{a_{k1}}(4r'_1-{a_{k1}})}{8r'_1}-\sum_{j=2}^{s'}\frac{{a_{kj}}(r_j-{a_{kj}})}{2r_j}\in \mathbb Z. 
    \end{align*}
    As $\nu_2(c)\geq -1$, this implies that $2\mid a_{k1}$ for $1\leq k\leq l$. Then by \eqref{eq.KEE'/r 8}, for any $1\leq k, k'\leq l$,
    \begin{align*}
        -c\mathbb{I}_{kk'} +\frac{y_ky_{k'}}{2} +
        \frac{\frac{a_{k1}}{2}\frac{a_{k'1}}{2}}{r'_1}+\sum_{j=2}^{s'}\frac{a_{kj}a_{k'j}}{r_j}\in \mathbb Z. 
    \end{align*}
    By Lemma~\ref{lem ADA} for $e_0=1$ and $n=n_2+n_8$, we have $\nu_2(c)\geq 0$ as long as $l\geq 2\lfloor\frac{n_2+n_8}{2}\rfloor+2$. 

    \medskip

    (4b) If $n_4=2$ and $\{4,4\}\subset\mathcal{R}$, we may assume that $r_1=r_2=4$ and $4\nmid r_j$ for $j>2$. Then by \eqref{eq.KE2 8}, for any $1\leq k\leq l$,
    \begin{align*}
        c-\frac{{y_{k}}(4-{y_{k}})}{4} - \frac{{a_{k1}}(4 -{a_{k1}})}{8}-\frac{{a_{k2}}(4 -{a_{k2}})}{8}-\sum_{j=3}^{s'}\frac{{a_{kj}}(r_j-{a_{kj}})}{2r_j}\in \mathbb Z. 
    \end{align*}
    As $\nu_2(c)\geq -1$, this implies that $ a_{k1}\equiv a_{k2}\bmod 2$ for $1\leq k\leq l$. We may write $a_{k2}=a_{k1}+2b_k$ for some integer $b_k$.
    Then by \eqref{eq.KEE'/r 8}, for any $1\leq k, k'\leq l$, 
    \begin{align*}
        -c\mathbb{I}_{kk'}+\frac{y_ky_{k'}}{2}
        +\frac{(a_{k1}+b_{k})a_{k'1}}{2}+\frac{a_{k1}b_{k'}}{2}+\sum_{j=3}^{s'}\frac{a_{kj}a_{k'j}}{r_j}\in \mathbb Z.\label{eq.KEE'/r 844}
    \end{align*}
    By Lemma~\ref{lem ADA} for $e_0=1$ and $n=n_2+n_8+2$, we have $\nu_2(c)\geq 0$ as long as $l\geq 2\lfloor\frac{n_2+n_8+2}{2}\rfloor+2$. 

    \medskip

    (4c) If $n_4=3$, $\{4,4,4\}\subset \mathcal{R}$, and $n_2=n_8=0$, then we may assume that $r_1=r_2=r_3=4$ and $2\nmid r_j$ for $j>3$. Then by \eqref{eq.KE2 8}, for any $1\leq k\leq l$,
    \begin{align*}
        c-\frac{{a_{k1}}(4 -{a_{k1}})}{8}-\frac{{a_{k2}}(4 -{a_{k2}})}{8}-\frac{{a_{k3}}(4 -{a_{k3}})}{8}-\sum_{j=4}^{s'}\frac{{a_{kj}}(r_j-{a_{kj}})}{2r_j}\in \mathbb Z. 
    \end{align*}
    As $\nu_2(c)\geq -1$, this implies that $a_{k1}^2+ a_{k2}^2+a_{k3}^2\equiv 0\bmod 4$ for $1\leq k\leq l$. Then $ a_{k1}\equiv  a_{k2}\equiv a_{k3} \equiv 0\bmod 2$. So $4\mid a_{1j}a_{2j}$ for $1\leq j\leq 3$ and hence by \eqref{eq.KEE'/r 8} for $(k,k')=(1,2)$,  
    \begin{align*}
        -c  + \sum_{j=4}^{s'}\frac{a_{1j}a_{2j}}{r_j}\in \mathbb Z, 
    \end{align*}
    which implies that $\nu_2(c)\geq 0$. 

    \medskip 

    (5) We may assume that $r_1=r_2=8$. By (2), $\nu_{2}(c)\geq -2$.

    If $n_4=0$, then by \eqref{eq.KE2}, for any $1\leq k\leq l$, 
    \[
        \nu_2\left(\frac{a_{k1}(8-a_{k1})}{16}+\frac{a_{k2}(8-a_{k2})}{16}\right)\geq -2.
    \]
    This implies that $a_{k1}^2+a_{k2}^2\equiv 0\bmod 4$, which means that $a_{k1}\equiv a_{k2} \equiv 0\bmod 2$. Then 
    \[
        \nu_2\left(\frac{a_{11}a_{21}}{8}+\frac{a_{12}a_{22}}{8}\right)\geq -1.
    \]
    So by \eqref{eq.KEE'/r} for $(k,k')=(1,2)$, $\nu_{2}(c)\geq -1$. 

    If further $n_2=0$, then by \eqref{eq.KE2}, for any $1\leq k\leq l$, 
    \[
        \nu_2\left(\frac{a_{k1}(8-a_{k1})}{16}+\frac{a_{k2}(8-a_{k2})}{16}\right) \geq -1.
    \]
    This implies that $a_{k1}^2+a_{k2}^2\equiv 0\bmod 8$, which means that $\frac{a_{k1}}{2}\equiv \frac{a_{k2}}{2} \bmod 2$. Then 
    \[
        \nu_2\left(\frac{a_{11}a_{21}}{8}+\frac{a_{12}a_{22}}{8}\right)\geq 0.
    \] 
    So by \eqref{eq.KEE'/r} for $(k,k')=(1,2)$, $\nu_{2}(c)\geq 0$. 
\end{proof}

Based on Corollary~\ref{cor -e+1}, Proposition~\ref{prop p=3}, and Proposition~\ref{prop p=2}, we have the following algorithm to determine a function $\LB$ which gives a lower bound of $(-r_XK_X\cdot C)$ for any crepant curve $C$ depending on the information of $\mathcal{R}_X$ and $e'_C$. 

\begin{algo}\label{algo LB}
    Let $X$ be a canonical Fano $3$-fold. Keep Setting~\ref{set KC}(1)(2). Fix a positive integer $N\geq 2$. We will define a positive integer $\LB(N)$ and positive integers $f_p(N)$ for each prime number $p\leq 23$ in this algorithm. 

    \begin{enumerate}
        \item (Corollary~\ref{cor -e+1}) For a prime number $5\leq p\leq 23$, we determine $f_p(N)$ as the following: 
        \[
            f_p(N)=\begin{cases}
                p & \text{if } n_p>0 \text{ and } N-1\geq n_{p}+1+\mathsf{if}(p\mid n_{p}+2);\\
                1 & \text{otherwise.}
            \end{cases}
        \]
        
        \item We determine $f_3(N)$ as the following:
        \begin{enumerate}
            \item (Proposition~\ref{prop p=3}) If $n_9>0$ and $N-1\geq 3$, then set 
            \[
                f_3(N)=\begin{cases}
                9 & \text{if } N-1\geq n_3 +1+\mathsf{if}(3 \mid n_3+2);\\
                3 & \text{otherwise.}
                \end{cases}
            \]

            \item (Corollary~\ref{cor -e+1}) If $n_9=0$, $n_3>0$, and $N-1\geq n_{3}+1+\mathsf{if}(3\mid n_{3}+2)$, then set $f_3(N)=3$.
   
            \item Otherwise set $f_3(N)=1$.
        \end{enumerate}

        \item We determine $f_2(N)$ as the following: denote $e\coloneq\nu_2(r_X)$, if $e=0$ or $N=2$, then set $f_2(N)=1$. In the following we always consider $e>0$ and $N-1\geq 2$.
        \begin{enumerate}
            \item (Proposition~\ref{prop p=2}(3)) If $n_{16}=1$, then set
            \[
                f_2(N)=\begin{cases}
                2& \text{if } n_8\neq 0;\\
                4& \text{if } n_8= 0\text{ and } n_4\neq 0;\\
                8 &\text{if } n_8=n_4=0.
                \end{cases}
            \]

            \item (Proposition~\ref{prop p=2}(5)) If $n_8=2$, then set 
            \[
                f_2(N)=\begin{cases}
                   2 & \text{if } n_{4}\neq 0;\\
                    4 & \text{if } n_4=0 \text{ and } n_2\neq 0;\\
                    8 &\text{if } n_{2}=n_4=0. \\
                \end{cases}
            \]

            \item (Proposition~\ref{prop p=2}(4)) If $n_{16}=0$, $n_8\leq 1$, $n_4+n_8>0$, and $N-1\geq 2\lfloor\frac{n_4}{2}\rfloor+2$, 
            \begin{enumerate}
                \item if $n_4\leq 1$ and $N-1\geq 2\lfloor\frac{n_2+n_8}{2}\rfloor+2$, then set $f_2(N)=2^e$;
                
                \item if $n_4=2$, $\{4,4\}\subset \mathcal{R}$, and $N-1\geq 2\lfloor\frac{n_2+n_8+2}{2}\rfloor+2$, then set $f_2(N)=2^e$;
                
                \item if $n_4=3$, $\{4,4,4\}\subset \mathcal{R}$, and $n_2=n_8=0$, 
                then set $f_2(N)=2^e$;

                \item otherwise set $f_2(N)=2^{e-1}$.
            \end{enumerate} 

            \item (Proposition~\ref{prop p=2}(2)) If $n_{16}=n_8=0$, $n_4=2$ and $N-1\in\{2,3\}$, then set $f_2(N)=2$.

            \item (Corollary~\ref{cor -e+1}, Proposition~\ref{prop p=2}(2)) If $n_{16}=n_8=n_4=0$,  then set 
            \[
                f_2(N)=\begin{cases}
                2 & \text{if } 0<n_{2}\leq 2;\\
                2 & \text{if }n_{2}>2 \text{ and } N-1\geq 2\lfloor\frac{n_2}{2}\rfloor+2.
                \end{cases}
            \]
    
            \item Otherwise set $f_2(N)=1$.
        \end{enumerate}

        \item Finally, set 
        \[
            \LB(N)= \prod_{p\leq 23 \text{ prime}} f_p(N). 
        \]
    \end{enumerate}
\end{algo}

\begin{rem}\label{label rem LB}
    We give some easy properties of the function $\LB\colon \mathbb{Z}_{\geq 2}\to \mathbb{Z}_{>0}$ defined in Algorithm~\ref{algo LB}.
    \begin{enumerate}
        \item $\LB$ depends only on $\mathcal{R}_X$, and $\LB$ takes values in the divisors of $r_X$;
        
        \item $\LB$ is non-decreasing; moreover, if $N\leq N'$, then $\LB(N)$ divides $\LB(N')$;
        
        \item $\LB(2)=1$;

        \item If elements in $\mathcal{R}_X$ are pairwisely coprime and square free (i.e., $\nu_p(r_X)\leq 1$ and $n_p\leq 1$ for any prime number $p$), then $\LB(4)=r_X$; if moreover $3\nmid r_X$, then $\LB(3)=r_X$.   
    \end{enumerate}
\end{rem}

The following is the main theorem of this section.

\begin{thm}\label{thm c>=LB}
    Let $X$ be a canonical Fano $3$-fold and let $C\subset \Sing(X)$ be a crepant curve with non-split crepant divisors. Then $\LB(e'_C)$ divides $(-r_XK_X\cdot C)$. In particular,
        $(-r_XK_X\cdot C)\geq \LB(e'_C).$
\end{thm}

\begin{proof}
    Keep Setting~\ref{set KC}. Recall that $l=e'_C-1$ and $c=(K_X\cdot C)$.

    For any prime number $p\leq 23$, by the construction of $f_p$, we have 
    \[
        \nu_p(c)\geq -\nu_p(r_X)+\nu_p(f_p(l+1))=-\nu_p(r_X)+\nu_p(f_p(e'_C))
    \]
    by Corollary~\ref{cor -e+1}, Proposition~\ref{prop p=3}, and Proposition~\ref{prop p=2} (see Algorithm~\ref{algo LB} for more precise references in the construction of $f_p$). So for any prime number $p\leq 23$, $\nu_p(r_Xc)\geq \nu_p(\LB(e'_C))$. By construction, all prime factors of $\LB(e'_C)$ are at most $23$, so  $\LB(e'_C)$ divides $(-r_XK_X\cdot C)$.
\end{proof}

We conclude this section by computing a concrete example which motivated us to formulate Theorem~\ref{thm c>=LB}. 

\begin{ex}\label{ex compute -rKC}
    Consider $X=\mathbb{P}(5,6,22,33)$ in Example~\ref{ex 123}. Then $r_X=5$ and $X$ contains $3$ crepant curves $C_2, C_3, C_{11}$ of types $\mathsf{A}_1, \mathsf{A}_2, \mathsf{A}_{10}$. Set theoretically, $C_{11}=A_5\cap A_6$ where $A_5\in |\mathcal{O}_X(5)|$ and $A_6\in |\mathcal{O}_X(6)|$ are the unique elements. Here note that both $A_5$ and $A_6$ are not Cartier along $C_{11}$, so by a local computation at the generic point of $C_{11}$, we have $C_{11}=11A_5\cdot A_6$ as cycles. Then $(-K_X\cdot C_{11})=(-K_X\cdot 11A_5\cdot A_6)=1$. This calculation matches Theorem~\ref{thm c>=LB}. The same argument shows that $(-K_X\cdot C_{i})=1$ for $i=2,3$.
\end{ex}
 
\section{Candidates of canonical Fano threefolds with large Fano indices}

In this section, we combine various inequalities including the Kawamata--Miyaoka type inequality and lower bounds of degrees of crepant curves to search for possible candidates of canonical Fano $3$-folds with $\qQ>66$. The main theorem of this section is Theorem~\ref{thm.36cases} which gives a list of $36$ numerical types.  

We will often use the following setting in the rest of the paper. 

\begin{set}\label{set search}
    \begin{enumerate}
        \item Let $(X, A)$ be an unmodifiable canonical Fano $3$-fold. In particular, 
     any crepant curve $C\subset \Sing(X)$ is of type $\mathsf{A}$ with non-split crepant divisors by Corollary~\ref{cor A primitive}.

        \item Denote $q\coloneq\qQ(X)$ and take $J_A$ to be the smallest positive integer such that $J_AA$ is Cartier in codimension $2$.
    
        \item Denote 
        \[
            \nabla_X\coloneq r_Xc_2(X)\cdot c_1(X)- \frac{q^2+2q-4}{4q^2}r_Xc_1(X)^3.
        \]
    \end{enumerate}
\end{set}

By the Kawamata--Miyaoka type inequality and Theorem~\ref{thm c>=LB}, we have the following:
\begin{thm}[{cf. \cite{jiang-liu-liu}*{Corollary~4.7}}]\label{thm Delta>jC}
    Keep Setting~\ref{set search}. If $q=\qQ(X)\geq 6$, then
    \begin{align*}
        \nabla_X\geq \sum_{C\subset\Sing(X)}\left(j_C-\frac{1}{j_C}\right)(-r_XK_X\cdot C)\geq \sum_{C\subset\Sing(X)}\left(j_C-\frac{1}{j_C}\right)\LB(j_C).  
    \end{align*}
\end{thm}

\begin{proof}
    By \cite{jiang-liu-liu}*{Theorem~3.8} for $q\geq 6$, we have \[\nabla_X\geq r_Xc_2(X)\cdot c_1(X)-r_X\hat{c}_2(X)\cdot c_1(X).\] Then the conclusion follows from Theorem~\ref{thm.c1c2 diff} and Theorem~\ref{thm c>=LB}. Here note that any crepant curve $C$ is of type $\mathsf{A}$, which implies that $e_C=e'_C=g_C=j_C$.
\end{proof}

We will often use the following lemma to deal with the term $j_C-\frac{1}{j_C}$, which allows us to replace $j_C$ by the collection of its prime power factors.
We omit the proof since it is elementary. 
\begin{lem}\label{lem ab>a+b}
    Let $a, b$ be positive integers. Then 
    \[
        ab-\frac{1}{ab}\geq a-\frac{1}{a}+b-\frac{1}{b}.
    \]
    If furthermore $a\geq 2$ and $b\geq 4$, then 
     \[
        ab-\frac{1}{ab}\geq 2\left(a-\frac{1}{a}\right)+b-\frac{1}{b}.
    \]
\end{lem}

\begin{thm}\label{thm.36cases}
    Let $(X, A)$ be an unmodifiable canonical Fano $3$-fold. Assume that $\qQ(X)> 66$. Then the numerical data of all possible candidates of $X$ are listed in Table~\ref{main tab} consisting of $36$ types.
\end{thm}
{\footnotesize
    \begin{longtable}{LLLLLLLLLL}
        \caption{Candidates for unmodifiable canonical Fano $3$-folds with $\qQ>66$}\label{main tab}\\
        \hline
       \text{\textnumero} & B_X & \qQ  & r_X & r_Xc_1^3 & r_Xc_2c_1 &  \{p^a\}& \{\LB(p^a)\} & \nabla_X & \text{ref.} \\
        \hline
        \endfirsthead
        \multicolumn{4}{l}{{ {\bf \tablename\ \thetable{}} \textrm{-- continued}}}
        \\
        \hline 
        \text{\textnumero} & B_X & \qQ  & r_X & r_Xc_1^3 & r_Xc_2c_1 &  \{p^a\}& \{\LB(p^a)\} & \nabla_X & \text{ref.}  \\
        \hline
        \endhead
        \hline
        \hline \multicolumn{4}{c}{{\textrm{Continued on next page}}} \\ \hline
        \endfoot
        
        \hline \hline
        \endlastfoot
        1& \{(5,1)\} & 84 & 5 & 84 & 96 &3,4,7&5,5,5 &74.52& \S\,\ref{section group A}\\
        2& \{2\times (3,1)\} & 70 & 3 & 70 & 56 &2,5,7&1,3,3 &38.02 & \S\,\ref{section group A}\\
        3& \{(3,1),(11,3)\}& 70 & 33 & 490 & 344 &2,5&1,33 &218.1& \S\,\ref{sec group C1}+\S\,\ref{sec group C2}\\
        4& \{(3,1),(11,5)\}& 80 & 33 & 640 & 344 &2,5&1,33 &180.1& \S\,\ref{sec group C1} \\
        5& \{(5,1),(7,2)\}& 72 & 35 & 288 & 432 &2,9 &1,35 &358.06 & \S\,\ref{section group A}\\
        6& \{(5,1),(11,1)\}& 72 & 55 & 864 & 456 &2,3 &1,55 & 234.17 &\S\,\ref{sec group C1}+\S\,\ref{sec group C2}\\
        7&  \{(5,1),(11,2)\}& 78 & 55 & 1014 & 456 &2,3 &1,55 & 196.17 &\S\,\ref{sec group C1}\\
        8&  \{(5,2),(11,3)\}& 84 & 55 & 1176 & 456 &2,3 &1,55 & 155.17 & \S\,\ref{sec group C1}\\
        9&  \{2\times(2,1),(3,1)\}& 70 & 6 & 70 & 110 &2,5,7 &1,6,6 & 92.02 & \S\,\ref{section group A}\\
        10&  \{2\times(2,1),(9,4)\}& 70 & 18 & 490 & 218 &2,5 &1,18 & 92.1 &\S\,\ref{sec group B}\\
        11&  \{(2,1),(5,2),(11,5)\}& 69 & 110 & 1587 & 747 &3 &110 &339.09&\S\,\ref{sec group C1}+\S\,\ref{sec group C2}\\
        12&  \{(3,1),(5,1),(11,1)\}& 82 & 165 & 3362 & 928 &2 &1 & 67.5 &\S\,\ref{sec group C1}\\
        13&  \{(3,1),(5,1),(11,4)\}& 68 & 165 & 2312 & 928 &2 &1 & 333.5&\S\,\ref{sec group C1}+\S\,\ref{sec group C2}\\
        14&  \{(3,1),(5,2),(11,2)\}& 76 & 165 & 2888 & 928 &2 &1 &187.5 &\S\,\ref{sec group C1}\\
        15&  \{(3,1),(5,2),(11,5)\}& 74 & 165 & 2738 & 928 &2 &1 & 225.5 &\S\,\ref{sec group C1}\\
        16&  \{(3,1),(6,1),(7,2)\}& 75 & 42 & 375 & 363 &3,5 &14,42 &266.82 & \S\,\ref{section group A}\\
        17 &  \{(4,1),(5,1),(5,2)\}& 75 & 20 & 375 & 213 &3,5 &4,20 &116.82& \S\,\ref{section group A}\\
        18 &  \{(5,1),(5,2),(7,3)\} & 90 & 35 & 270 & 264 &2,3,5 &1,7,35 &195.04 & \S\,\ref{section group A}\\
        19&  \{2\times(2,1), (3,1),(5,2)\}& 98 & 30 & 686 & 406 &2,7 &1,30 &231.08& \S\,\ref{section group A}\\
        20 &  \{(2,1),(3,1),(5,1),(6,1)\}& 72 & 30 & 864 & 276 &2,3 &1,10 &54.17 & \S\,\ref{sec group B}\\
        21&  \{(2,1),(3,1),(5,1),(11,2)\}& 67 & 330 & 4489 & 1361 &\emptyset &\emptyset &206.25 & \S\,\ref{sec group C1}+\S\,\ref{sec group C2}\\
        22&  \{(2,1),(3,1),(5,2),(11,1)\}& 71 & 330 & 5041 & 1361 &\emptyset &\emptyset &66.25 & \S\,\ref{sec group C1}+\S\,\ref{sec group C2} \\
        23 & \{(2,1),2\times(4,1), (7,2)\}& 72 & 28 & 432 & 228 &3,4 &{14,14} &117.09 &\S\,\ref{sec group B}\\
        24 &  \{3\times(3,1), (5,1)\}& 72 & 15 & 432 & 168 &3,4 &5,5 &57.09&\S\,\ref{sec group B}\\
        25 &  \{3\times(3,1), (5,2)\}& 84 & 15 & 168 & 168 &2,3,7 &1,5,15 &125.03 & \S\,\ref{section group A}\\
        26 &  \{3\times(3,1), (7,1)\}& 90 & 21 & 270 & 192 &2,3,5 &1,7,21 &123.04& \S\,\ref{section group A}\\
        27&  \{(3,1),(5,2),(6,1), (7,1)\}& 69 & 210 & 1587 & 807 &3 &70 &399.09 &\S\,\ref{sec group B}\\
        28&  \{(3,1),(5,2),(6,1), (7,3)\}& 81 & 210 & 2187 & 807 &3 &70 &247.09 & \S\,\ref{section group A}\\
        29 &  \{4\times(2,1),(5,1)\}& 84 & 10 & 168 & 132 &2,3,7 &1,5,10 &  89.03 & \S\,\ref{section group A}\\
        30 &  \{4\times(2,1),(7,1)\}& 72 & 14 & 432 & 156 &3,4 &7,7 & 45.09 & \S\,\ref{section group A}\\
        31 &  \{4\times(2,1),(7,1)\}& 80 & 14 & 320 & 156 &4,5 &7,7 & 74.05 & \S\,\ref{section group A}\\
        32 &  \{3\times(3,1),(5,1),(7,1)\}& 78 & 105 & 1014 & 456 &2,3 &1,35 &196.17 & \S\,\ref{sec group B}\\
        33 &  \{3\times(3,1),(5,1),(7,2)\}& 72 & 105 & 864 & 456 &2,3 &1,35 &234.17 &\S\,\ref{sec group B}\\
        34 &  \{3\times(3,1),(5,1),(7,2)\}& 72 & 105 & 864 & 456 &3,4 &35,35 &234.17& \S\,\ref{section group A}\\
        35 &  \{4\times(2,1),(5,1),(7,3)\}& 68 & 70 & 1156 & 444 &4 &35 &146.75 & \S\,\ref{sec group B}\\
        36 &  \{6\times(2,1),(3,1)\}& 70 & 6 & 70 & 74 &2,5,7 &1,3,3 &56.02& \S\,\ref{sec group B}
    \end{longtable}
    } 

\begin{proof}
    Keep the notation in Setting~\ref{set search}. Suppose that $q> 66$. Then
    \begin{align}\label{eq.test}
       \frac{4q^2}{q^2+2q-4}r_Xc_2(X)\cdot c_1(X)\geq  r_Xc_1(X)^3\geq q> 66
    \end{align} 
    by Theorem~\ref{thm.degreeandindex} and \cite{jiang-liu-liu}*{Theorem~3.8}. In particular, $4r_Xc_2(X)\cdot c_1(X)>66$.

    Let $J_A=p_1^{a_1}p_2^{a_2}\cdots p_k^{a_k}$ be the prime factorization, where $p_i$ are distinct prime numbers. Then each $p_i^{a_i}$ divides at least one $j_C$ for some  crepant curve $C\subset\Sing(X)$ by the definition of $J_A$. Then 
    \begin{align}\label{eq Delta>pa}
        \nabla_X\geq \sum_{i=1}^k \left(p_i^{a_i}-\frac{1}{p_i^{a_i}}\right)\LB(p_i^{a_i}) 
    \end{align}
    by Theorem~\ref{thm Delta>jC} and Lemma~\ref{lem ab>a+b}. Here recall that $\LB$ is a non-decreasing function. 

    Then we employ a computer program to search for such $(X, A)$ satisfying $q>66$ by the following algorithm. The Python code can be found in \url{https://chenjiangfudan.github.io/home/code/pycode-Fano-index-66.txt}.
    \begin{algo}\label{algo1}
    
        {\bf Step 1}. List all possible $(\mathcal R_X, c_2(X)\cdot c_1(X))$ satisfying \eqref{eq.range} {and $4r_Xc_2(X)\cdot c_1(X)>66$}. By \eqref{eq.range} and Remark~\ref{rem.posofc2}, there are only finitely many candidates.

        \medskip 
        
        {\bf Step 2}. Among the list in Step 1, list all possible 
        $(B_X, r_Xc_1(X)^3, q, J_A)$
        satisfying  the divisibility constraints in Theorem~\ref{thm.degreeandindex}, the integrality constraint in \eqref{eq.RR-Fano}, and \eqref{eq.test}.
        
        \medskip 
     
        {\bf Step 3}. Among the list in Step 2, list all possible candidates satisfying \eqref{eq Delta>pa} where $\LB$ is given by Algorithm~\ref{algo LB}.
        
        The output of this algorithm is a list of numerical data 
        \[
            \left(B_X, \qQ(X), r_X, r_Xc_1(X)^3, r_Xc_2(X)\cdot c_1(X),\{p_i^{a_i}\},\{\LB(p_i^{a_i})\}, \frac{\lceil 100\nabla_X \rceil}{100}\right).
        \]
    \end{algo}
    
    By Algorithm~\ref{algo1}, we get all possible numerical data listed in Table~\ref{main tab}. Here we note that if $J_A=1$, then $\{p_i^{a_i}\}=\emptyset$ and the right-hand side of \eqref{eq Delta>pa} is automatically $0$. 
    \end{proof}
    
In the rest of the paper, our mission is to further rule out all cases in Table~\ref{main tab}. We divide the candidates in Table~\ref{main tab} into $3$ groups:

\begin{itemize}
    \item  Group A: \textnumero 1, \textnumero2, \textnumero5, \textnumero9, \textnumero16--19, \textnumero25, \textnumero26, \textnumero28--31, \textnumero34.
    \item Group B: \textnumero10, \textnumero20, \textnumero23, \textnumero24, \textnumero27, \textnumero32, \textnumero33, \textnumero35, \textnumero36.
    \item Group C: \textnumero3, \textnumero4, \textnumero6--8, \textnumero11--15, \textnumero21, \textnumero22.
\end{itemize}

The division of candidates is based on the difficulty of being ruled out. Indeed, candidates in Group A can be easily ruled out by the integrality constraint in Corollary~\ref{cor RR integer general} (see Proposition~\ref{prop rule out A}). Candidates in Group B can also be ruled out by the integrality constraint, but it is more complicated than Group A. We need to discuss case by case using different methods, where \textnumero24, \textnumero27,  \textnumero35 are the most annoying ones (see \S\,\ref{sec group B}). Candidates in Group C are the most complicated. We can not derive any contradiction from the integrality constraint, instead, we give the explicit form of the formula for $h^0(X, \mathcal{O}_X(sA))$ (see Theorem~\ref{thm RR for Group C}). Interestingly, it turns out that all candidates in Group C share the same formula. Then we introduce a geometric method to rule out Group C by considering foliations (see \S\,\ref{sec group C2}).

\section{Ruling out remaining cases: Group A and Group B}\label{sec 7}

In this section, we will rule out Group A and Group B in Table~\ref{main tab}. The idea is to use the integrality constraint in Corollary~\ref{cor RR integer general} to derive contradictions. 

\subsection{Determining types and degrees of crepant curves}\label{sec det C}

In order to apply Corollary~\ref{cor RR integer general} or Theorem~\ref{thm.rr.sA}, it is crucial to determine $(-r_XK_X\cdot C)$ and $c_C(sA)$ for a crepant curve $C\subset \Sing(X)$. For the former one, we need to know its exact value rather than a lower bound given by Theorem~\ref{thm c>=LB}; for the latter one, we need to know the type of $C$. We will address these two issues in this subsection. 

First we deal with crepant curves that are not of type $\mathsf{A}_1$. In most of the cases, we can determine the types and degrees of such curves by Theorem~\ref{thm Delta>jC}. For example, if the degree of the curve $C$ is strictly larger than the lower bound $\LB(j_C)$ in Table~\ref{main tab}, then 
we will often get a contradiction by Theorem~\ref{thm Delta>jC} as $\nabla_X$ is already known.

\begin{prop}\label{prop determine C and KC}
    Keep Setting~\ref{set search}. Suppose that $J_A>2$. We may write \[J_A=p_0^{a_0}p_1^{a_1}p_2^{a_2}\dots p_k^{a_k}\] where $p_i$ are distinct prime numbers with $2=p_0<p_1<p_2<\dots<p_k$, $a_1, \dots, a_k$ are positive integers and $a_0\geq 0$.
  
    \begin{enumerate}
        \item If $a_0\leq 1$ and
        \begin{align*}
            \nabla_X<\sum_{i=1}^k \left(p_i^{a_i}-\frac{1}{p_i^{a_i}}\right)\LB(p_i^{a_i})+\left(p_1-\frac{1}{p_1}\right)\LB(p_1), 
        \end{align*}
        then $\Sing_1(X)$ consists of crepant curves $C_1,\dots, C_k$ where $C_i$ is of type $\mathsf{A}_{p_i^{a_i}-1}$ with $(-r_XK_X\cdot C_i)=\LB(p_i^{a_i})$ for $1\leq i\leq k$, and possibly several crepant curves of type $\mathsf{A}_1$; moreover, there is no crepant curve of type $\mathsf{A}_1$ if $a_0=0$.

        \item If $a_0\geq 2$ and
        \[
            \nabla_X<\sum_{i=0}^k \left(p_i^{a_i}-\frac{1}{p_i^{a_i}}\right)\LB(p_i^{a_i})+ \left(p'-\frac{1}{p'}\right)\LB(p'),
        \]
        where $p'=\min\{4, p_1\}$, then $\Sing_1(X)$ consists of crepant curves $C_0,\dots, C_k$ where $C_i$ is of type $\mathsf{A}_{p_i^{a_i}-1}$ with $(-r_XK_X\cdot C_i)=\LB(p_i^{a_i})$ for $0\leq i\leq k$, and possibly several crepant curves of type $\mathsf{A}_1$.
    \end{enumerate}
\end{prop}

\begin{proof}
    All crepant curves are of type $\mathsf{A}$ by Corollary~\ref{cor A primitive}; also there is no crepant curve of type $\mathsf{A}_1$ if $a_0=0$ as $2\nmid J_A$. Recall that each $p_i^{a_i}$ divides at least one $j_C$ for some crepant curve $C\subset\Sing(X)$ by the definition of $J_A$. Conversely, by Corollary~\ref{cor A primitive}, $j_C$ divides $J_A$ for any  crepant curve $C\subset\Sing(X)$. 

    \medskip

    (1) {\bf Step 1}. We show that there are distinct crepant curves $C_1, \dots, C_k\subset \Sing(X)$ such that $p_i^{a_i}\mid j_{C_i}$ for $1\leq i\leq k$; moreover, $j_{C_i}=p_i^{a_i}$ for $1\leq i\leq k$.

    Suppose that $p_{i_0}^{a_{i_0}}p_{i_1}^{a_{i_1}}\mid j_{C'}$ for some  crepant curve $C'\subset \Sing(X)$ and $1\leq i_0<i_1\leq k$, then by Theorem~\ref{thm Delta>jC},
    \begin{align*}
        \nabla_X \geq {}&\sum_{C\subset \text{\rm Sing}(X)}\left(j_C-\frac{1}{j_C}\right)\LB(j_C)\\
        \geq {}&\sum_{\substack{1\leq i\leq k\\i\neq i_0, i_1}} \left(p_i^{a_i}-\frac{1}{p_i^{a_i}}\right)\LB(p_i^{a_i})+\left(p_{i_0}^{a_{i_0}}p_{i_1}^{a_{i_1}}-\frac{1}{p_{i_0}^{a_{i_0}}p_{i_1}^{a_{i_1}}}\right)\LB(p_{i_0}^{a_{i_0}}p_{i_1}^{a_{i_1}})\\
        \geq{}& \sum_{\substack{1\leq i\leq k\\i\neq i_0, i_1}} \left(p_i^{a_i}-\frac{1}{p_i^{a_i}}\right)\LB(p_i^{a_i})+2\left(p_{i_0}^{a_{i_0}}-\frac{1}{p_{i_0}^{a_{i_0}}}\right)\LB(p_{i_0}^{a_{i_0}})+\left(p_{i_1}^{a_{i_1}}-\frac{1}{p_{i_1}^{a_{i_1}}}\right)\LB(p_{i_1}^{a_{i_1}})\\
        \geq{}&\sum_{i=1}^k \left(p_i^{a_i}-\frac{1}{p_i^{a_i}}\right)\LB(p_i^{a_i})+\left(p_1-\frac{1}{p_1}\right)\LB(p_1).
    \end{align*} 
    This contradicts the assumption. Here for the second and third inequalities we used Lemma~\ref{lem ab>a+b} and the fact that $p_{i_1}^{a_{i_1}}\geq p_{i_1}\geq 5$. 

    So there are distinct crepant curves $C_1, \dots, C_k\subset \Sing(X)$ such that $p_i^{a_i}\mid j_{C_i}$ for $1\leq i\leq k$.

    If for some $1\leq i_0\leq k$, $ j_{C_{i_0}}\neq p_{i_0}^{a_{i_0}}$, then $j_{C_{i_0}}\geq 2p_{i_0}^{a_{i_0}}$.  Then by Theorem~\ref{thm Delta>jC}, 
    \begin{align*}
        \nabla_X
        \geq {}&\sum_{i=1}^k\left(j_{C_i}-\frac{1}{j_{C_i}}\right)\LB(j_{C_i})\\
        \geq {}&\sum_{\substack{1\leq i\leq k\\i\neq i_0}} \left(p_i^{a_i}-\frac{1}{p_i^{a_i}}\right)\LB(p_i^{a_i})+\left(2p_{i_0}^{a_{i_0}}-\frac{1}{2p_{i_0}^{a_{i_0}}}\right)\LB(2p_{i_0}^{a_{i_0}})\\
        \geq{}&\sum_{i=1}^k \left(p_i^{a_{i}}-\frac{1}{p_i^{a_{i}}}\right)\LB(p_i^{a_{i}})+\left(p_1-\frac{1}{p_1}\right)\LB(p_1).
    \end{align*}
    This contradicts the assumption. 

    \medskip 

    {\bf Step 2}. We show that $(-r_XK_X\cdot C_i)=\LB(p_i^{a_i})$ for $1\leq i\leq k$.
 
    If for some $1\leq i_0\leq k$, $(-r_XK_X\cdot C_{i_0})\neq \LB(p_{i_0}^{a_{i_0}})$, then $(-r_XK_X\cdot C_{i_0})\geq 2 \LB(p_{i_0}^{a_{i_0}})$ by Theorem~\ref{thm c>=LB}. Then by Theorem~\ref{thm Delta>jC},  
    \begin{align*}\nabla_X 
        \geq  {}&\sum_{i=1}^k\left(j_{C_i}-\frac{1}{j_{C_i}}\right)(-r_XK_X\cdot C_i)\\\geq {}&\sum_{\substack{1\leq i\leq k\\i\neq i_0}} \left(p_i^{a_i}-\frac{1}{p_i^{a_i}}\right)\LB(p_i^{a_i})+\left(p_{i_0}^{a_{i_0}}-\frac{1}{p_{i_0}^{a_{i_0}}}\right)2\LB(p_{i_0}^{a_{i_0}})\\
        \geq{}&\sum_{i=1}^k \left(p_i^{a_i}-\frac{1}{p_i^{a_i}}\right)\LB(p_i^{a_i})+\left(p_{1}-\frac{1}{p_{1}}\right)\LB(p_{1}).
    \end{align*}
    This contradicts the assumption. 
 
    \medskip

    {\bf Step 3}. We show that there is no other crepant curve $C'\subset \Sing(X)$ with $j_{C'}\neq 2$. 

    If there is another crepant curve $C'$ with $j_{C'}\neq 2$, then $p_{i_0}\mid j_{C'}$ for some $1\leq i_0\leq k$ by Corollary~\ref{cor A primitive} as $j_{C'}\mid J_A$ and $a_0\leq 1$. In particular, $j_{C'}\geq p_1$.  Then by Theorem~\ref{thm Delta>jC}, 
    \begin{align*}
        \nabla_X
        \geq {}& \sum_{i=1}^k\left(j_{C_i}-\frac{1}{j_{C_i}}\right)\LB(j_{C_i})+\left(j_{C'}-\frac{1}{j_{C'}}\right)\LB(j_{C'})\\
        \geq{}&\sum_{i=1}^k \left(p_i^{a_i}-\frac{1}{p_i^{a_i}}\right)\LB(p_i^{a_i})+\left(p_{1}-\frac{1}{p_{1}}\right)\LB(p_{1}).
    \end{align*}
    This contradicts the assumption. 

    \medskip

    (2) The proof is similar to the proof of (1) under suitable modifications.
  
   \medskip
   
    {\bf Step 1}. We show that there are distinct crepant curves $C_0, \dots, C_k\subset \Sing(X)$ such that $p_i^{a_i}\mid j_{C_i}$ for $0\leq i\leq k$; moreover, $j_{C_i}=p_i^{a_i}$ for $0\leq i\leq k$. 

    Suppose that $p_{i_0}^{a_{i_0}}p_{i_1}^{a_{i_1}}\mid j_{C'}$ for some crepant curve  $C'\subset \Sing(X)$ and $0\leq i_0, i_1\leq k$ with $p_{i_0}^{a_{i_0}}<p_{i_1}^{a_{i_1}}$, then by Theorem~\ref{thm Delta>jC}, 
    \begin{align*}
        \nabla_X\geq {}&\sum_{C\subset \text{\rm Sing}(X)}\left(j_C-\frac{1}{j_C}\right)\LB(j_C)\\
        \geq {}&\sum_{\substack{0\leq i\leq k\\i\neq i_0, i_1}} \left(p_i^{a_i}-\frac{1}{p_i^{a_i}}\right)\LB(p_i^{a_i})+\left(p_{i_0}^{a_{i_0}}p_{i_1}^{a_{i_1}}-\frac{1}{p_{i_0}^{a_{i_0}}p_{i_1}^{a_{i_1}}}\right)\LB(p_{i_0}^{a_{i_0}}p_{i_1}^{a_{i_1}})\\
        \geq{}& \sum_{\substack{0\leq i\leq k\\i\neq i_0, i_1}} \left(p_i^{a_i}-\frac{1}{p_i^{a_i}}\right)\LB(p_i^{a_i})+2\left(p_{i_0}^{a_{i_0}}-\frac{1}{p_{i_0}^{a_{i_0}}}\right)\LB(p_{i_0}^{a_{i_0}})+\left(p_{i_1}^{a_{i_1}}-\frac{1}{p_{i_1}^{a_{i_1}}}\right)\LB(p_{i_1}^{a_{i_1}})\\
        \geq{}&\sum_{i=0}^k \left(p_i^{a_i}-\frac{1}{p_i^{a_i}}\right)\LB(p_i^{a_i})+\left(p'-\frac{1}{p'}\right)\LB(p').
    \end{align*} 
    This contradicts the assumption. Here for the second and third inequalities we used Lemma~\ref{lem ab>a+b} and the fact that $p_{i_1}^{a_{i_1}}\geq 4$. 

    So there are distinct crepant curves $C_0, \dots, C_k\subset \Sing(X)$ such that $p_i^{a_i}\mid j_{C_i}$ for $0\leq i\leq k$.

    If for some $0\leq i_0\leq k$, $ j_{C_{i_0}}\neq p_{i_0}^{a_{i_0}}$, then $j_{C_{i_0}}\geq 2p_{i_0}^{a_{i_0}}$.  Then by Theorem~\ref{thm Delta>jC},   
    \begin{align*}
        \nabla_X\geq {}&\sum_{i=0}^k\left(j_{C_i}-\frac{1}{j_{C_i}}\right)\LB(j_{C_i})\\
        \geq {}&\sum_{\substack{0\leq i\leq k\\i\neq i_0}} \left(p_i^{a_i}-\frac{1}{p_i^{a_i}}\right)\LB(p_i^{a_i})+\left(2p_{i_0}^{a_{i_0}}-\frac{1}{2p_{i_0}^{a_{i_0}}}\right)\LB(2p_{i_0}^{a_{i_0}})\\
        \geq{}&\sum_{i=0}^k \left(p_i^{a_{i}}-\frac{1}{p_i^{a_{i}}}\right)\LB(p_i^{a_{i}})+\left(p'-\frac{1}{p'}\right)\LB(p').
    \end{align*}
    This contradicts the assumption. 

    \medskip 

    {\bf Step 2}. We show that $(-r_XK_X\cdot C_i)=\LB(p_i^{a_i})$ for $0\leq i\leq k$.
 
    If for some $0\leq i_0\leq k$, $(-r_XK_X\cdot C_{i_0})\neq \LB(p_{i_0}^{a_{i_0}})$, then $(-r_XK_X\cdot C_{i_0})\geq 2 \LB(p_{i_0}^{a_{i_0}})$ by Theorem~\ref{thm c>=LB}. Then by Theorem~\ref{thm Delta>jC}, 
    \begin{align*}
        \nabla_X \geq  {}&\sum_{i=0}^k\left(j_{C_i}-\frac{1}{j_{C_i}}\right)(-r_XK_X\cdot C_i)\\\geq {}&\sum_{\substack{0\leq i\leq k\\i\neq i_0}} \left(p_i^{a_i}-\frac{1}{p_i^{a_i}}\right)\LB(p_i^{a_i})+\left(p_{i_0}^{a_{i_0}}-\frac{1}{p_{i_0}^{a_{i_0}}}\right)2\LB(p_{i_0}^{a_{i_0}})\\
        \geq{}&\sum_{i=0}^k \left(p_i^{a_i}-\frac{1}{p_i^{a_i}}\right)\LB(p_i^{a_i})+\left(p'-\frac{1}{p'}\right)\LB(p').
    \end{align*}
    This contradicts the assumption. 
 
    \medskip

    {\bf Step 3}. We show that there is no other crepant curve $C'\subset \Sing(X)$ with $j_{C'}\neq 2$. 

    If there is another crepant curve $C'$ with $j_{C'}\neq 2$, then either $4\mid j_{C'}$ or $p_{i_0}\mid j_{C'}$ for some $1\leq i_0\leq k$ by Corollary~\ref{cor A primitive}. In particular, $j_{C'}\geq p'$. Then by Theorem~\ref{thm Delta>jC},  
    \begin{align*}
        \nabla_X\geq {}& \sum_{i=0}^k\left(j_{C_i}-\frac{1}{j_{C_i}}\right)\LB(j_{C_i})+\left(j_{C'}-\frac{1}{j_{C'}}\right)\LB(j_{C'})\\
        \geq{}&\sum_{i=0}^k \left(p_i^{a_i}-\frac{1}{p_i^{a_i}}\right)\LB(p_i^{a_i})+\left(p'-\frac{1}{p'}\right)\LB(p').
    \end{align*}
    This contradicts the assumption. 
\end{proof}

For a crepant curve of type $\mathsf{A}_1$, it is hard to determine its intersection number with $-K_X$ by Theorem~\ref{thm Delta>jC} as its contribution could be very small (see Remark~\ref{label rem LB}(3)). So the trick is to treat all crepant curves of type $\mathsf{A}_1$ as a whole.

\begin{defn}\label{def xa1}
    Let $(X,A)$ be an unmodifiable canonical Fano $3$-fold. 
    Denote 
    \[
        x_{\mathsf{A}_1}\coloneq \sum_{\substack{C\subset\Sing(X)\\\text{type }\mathsf{A}_1}} (-r_XK_X\cdot C)\in \mathbb{Z}_{\geq 0}.
    \]
    We will frequently use the following formulas (see Example~\ref{ex cc A123}): 
\begin{align}
\sum_{\substack{C\subset\Sing(X)\\\text{type }\mathsf{A}_1}} (-K_X\cdot C)c_C(sA)=-\frac{x_{\mathsf{A}_1}(\overline{s})_2}{4r_X},\label{eq xa1 formula 1}\\
\sum_{\substack{C\subset\Sing(X)\\\text{type }\mathsf{A}_1}} \left(j_C-\frac{1}{j_C}\right) (-r_XK_X\cdot C)=\frac{3x_{\mathsf{A}_1}}{2}.\label{eq xa1 formula 2}
\end{align}   
\end{defn}
We will see in later discussion that usually this integer $x_{\mathsf{A}_1}$ satisfies certain congruence equations, which means that it is either $0$ or sufficiently large to violate Theorem~\ref{thm Delta>jC}.

\subsection{Group A}\label{section group A}
 
In this subsection, we treat Group A which contains  \textnumero 1, \textnumero2, \textnumero5, \textnumero9, \textnumero16--19, \textnumero25, \textnumero26, \textnumero28--31, \textnumero34 in Table~\ref{main tab}.

\begin{prop}\label{prop rule out A}
    All candidates in Group A do not exist. 
\end{prop}

\begin{proof}
    In all these cases, we can verify by hand that the conditions of Proposition~\ref{prop determine C and KC} are satisfied. Here all data of $\{p^a\}, \{\LB(p^a)\}$, and $\nabla_X$ in Table~\ref{main tab} are sufficient for the verification, except for \textnumero 5 where we should further compute $\LB(3)$. By Remark~\ref{label rem LB}(4), $\LB(3)=r_X=35$ in \textnumero 5.
    
    Then we can describe the types of all crepant curves $C$ in $\Sing(X)$ which are not of type $\mathsf{A}_1$ and the corresponding $(-r_XK_X\cdot C)$ in Table~\ref{tab group A}, where the symbol $(\mathsf{A}_1)$ means that there might be several crepant curves of type $\mathsf{A}_1$ in $\Sing(X)$. 

    {
    \begin{longtable}{LLLLL}
        \caption{Crepant curves in Group A}\label{tab group A}\\
        \hline
        \text{\textnumero} & \qQ& r_Xc_1^3 & C\subset\Sing(X) & (-r_XK_X\cdot C)  \\
        \hline
        \endfirsthead
        \multicolumn{4}{l}{{ {\bf \tablename\ \thetable{}} \textrm{-- continued}}}
        \\
        \hline 
        \text{\textnumero} &\qQ& r_Xc_1^3 & C\subset\Sing(X)  & (-r_XK_X\cdot C)    \\
        \hline 
        \endhead
        \hline
        \hline \multicolumn{4}{c}{{\textrm{Continued on next page}}} \\ \hline
        \endfoot
        
        \hline \hline
        \endlastfoot
        1& 84& 84& \mathsf{A}_2,\mathsf{A}_3,\mathsf{A}_6, (\mathsf{A}_1) &5,5,5 \\
        2& 70 & 70&\mathsf{A}_4,\mathsf{A}_6,(\mathsf{A}_1)&3,3 \\
        5& 72 & 288& \mathsf{A}_8, (\mathsf{A}_1) &35\\
        9&   70& 70&\mathsf{A}_4,\mathsf{A}_6, (\mathsf{A}_1) &6,6 \\
        16& 75& 375& \mathsf{A}_2,\mathsf{A}_4 &14,42 \\
        17& 75& 375& \mathsf{A}_2,\mathsf{A}_4 &4,20\\
        18 & 90& 270& \mathsf{A}_2,\mathsf{A}_4,(\mathsf{A}_1) &7,35\\
        19&98&686& \mathsf{A}_6, (\mathsf{A}_1) &30\\
        25& 84& 168&\mathsf{A}_2,\mathsf{A}_6,(\mathsf{A}_1) &5,15\\
        26& 90& 270& \mathsf{A}_2,\mathsf{A}_4,(\mathsf{A}_1) &7,21\\
        28& 81& 2187& \mathsf{A}_2  &70\\   
        29& 84& 168& \mathsf{A}_2,\mathsf{A}_6, (\mathsf{A}_1) &5,10\\
        30&  72& 432&\mathsf{A}_2,\mathsf{A}_3,(\mathsf{A}_1) &7,7 \\
        31&  80& 320&\mathsf{A}_3,\mathsf{A}_4,(\mathsf{A}_1) &7,7\\
        34&  72 & 864&\mathsf{A}_2,\mathsf{A}_3,(\mathsf{A}_1) &35,35 
    \end{longtable}
    } 
    To unify the notation, we suppose that $\Sing_1(X)$ consists of crepant curves $C_1, \dots, C_k$ of types $\mathsf{A}_{r_1-1}, \dots, \mathsf{A}_{r_k-1}$ with $r_i>2$ for $1\leq i\leq k$, and possibly several crepant curves of type $\mathsf{A}_1$. 

    By applying Corollary~\ref{cor.canpart} to $D=2A$, we have
    \begin{align}
        \mathbb{Z}\ni{}&   \frac{4r_Xc_1(X)^3}{q^2}+ \sum_{C\subset \Sing(X)}2(-r_XK_X\cdot C)c_C(2A)\notag\\
        ={}& \frac{4r_Xc_1(X)^3}{q^2}-\sum_{i=1}^k(-r_XK_X\cdot C_i)\frac{a_i(r_i-a_i)}{r_i}\label{eq group A1}
    \end{align}
    for some integers $0\leq a_i<r_i$ $(1\leq i\leq k)$. Here we used the fact that any crepant curve of type $\mathsf{A}_1$ in $\Sing(X)$ does not contribute to the sum by Example~\ref{ex cc A123} (see also Definition \ref{def xa1}). But we can check by hand that \eqref{eq group A1} is not an integer for any choice of integers $a_i$ $(1\leq i\leq k)$ in each case.
\end{proof}

\subsection{Group B}\label{sec group B}

In this subsection, we treat Group B which contains \textnumero10, \textnumero20, \textnumero23, \textnumero24, \textnumero27, \textnumero32, \textnumero33, \textnumero35, \textnumero36 in Table~\ref{main tab}. Group B is more complicated than Group A. Often we need to further consider $x_{\mathsf{A}_1}$ and get contradictions by congruence equations on square residues. For \textnumero24, \textnumero27,  \textnumero35 we need to consider further the geometric structure of $X$ and local indices of Weil pullbacks. 

\begin{prop}\label{prop rule out B1}
    \textnumero 20 in Table~\ref{main tab} does not exist. 
\end{prop}

\begin{proof}
    Applying Corollary~\ref{cor JA integer} to $D=J_AA=6A$, we have
    \begin{align*}
        -\frac{36}{2}A^2\cdot K_X- \frac{a_1({2-a_1})}{4}-\frac{a_2({3-a_2})}{6}-\frac{a_3({5-a_3})}{10}-\frac{a_4({6-a_4})}{12} \in \mathbb{Z}
    \end{align*}
    for some integers $a_i$ $(1\leq i\leq 4)$. But we can check that this does not hold for any choice of integers $a_i$ $(1\leq i\leq 4)$. 
\end{proof}

\begin{prop}\label{prop rule out B2}
    \textnumero 10, \textnumero23, \textnumero32, \textnumero33, \textnumero36 in Table~\ref{main tab} do not exist. 
\end{prop}

\begin{proof}
    Consider \textnumero 23. By Proposition~\ref{prop determine C and KC}(2), $\Sing_1(X)$ consists of crepant curves $C_1, C_2$ of types $\mathsf{A}_2, \mathsf{A}_3$ with $(-r_XK_X\cdot C_i)=14$ for $i=1,2$, and possibly several crepant curves  of type $\mathsf{A}_1$.

    Take $r'=336=r_XJ_A$. Then $2r_X\mid r'$. For $i=1,2$, $ (-r'K_X\cdot C_i)=14J_A$ is divided by $2j_{C_i}$; for any crepant curve $C'$ of type $\mathsf{A}_1$, $(-r'K_X\cdot C')=J_A(-r_XK_X\cdot C')\in J_A\mathbb{Z}$ is divided by $4$. So by Corollary~\ref{cor RR integer general} for $D=A$ and Remark~\ref{rem ignore part}, we have $-\frac{r'}{2}A^2\cdot K_X\in \mathbb{Z}$, but this contradicts the fact that $-\frac{r'}{2}A^2\cdot K_X=\frac{1}{2}$.
    
    \medskip

    Consider \textnumero 36. Recall that by Corollary~\ref{cor A primitive}, $j_C$ divides $J_A=70$ for any crepant curve $C\subset\Sing(X)$. By Algorithm~\ref{algo LB}, $\LB(10)=6$, so by Theorem~\ref{thm Delta>jC} and Step 1 of the proof of Proposition~\ref{prop determine C and KC}, $X$ only contains crepant curves of types $\mathsf{A}_1, \mathsf{A}_4, \mathsf{A}_6$. By Theorem~\ref{thm Delta>jC} again, $\Sing_1(X)$ consists of a crepant curve $C_1$ of type $\mathsf{A}_6$ with $(-r_XK_X\cdot C_1)=\LB(7)=3$ and possibly several crepant curves of types $\mathsf{A}_4, \mathsf{A}_1$. 

    Take $r'=120=20r_X$. Then $2r_X\mid r'$.  For any crepant curve $C'$ of type $\mathsf{A}_4$ or $\mathsf{A}_1$, $(-r'K_X\cdot C')=20(-r_XK_X\cdot C')\in 20\mathbb{Z}$ is divided by $5$ and $4$. So by Corollary~\ref{cor RR integer general} for $D=A$ and Remark~\ref{rem ignore part}, we have
    \[  
        -\frac{r'}{2}A^2\cdot K_X+(-r'K_X\cdot  C_1)c_{C_1}(A)=\frac{1}{7}-\frac{30a(7-a)}{7}\in \mathbb{Z}
    \]
    for some integer $a$. It implies that $7\mid 1+2a^2$, which is absurd.

    \medskip

    Consider \textnumero 10. By Proposition~\ref{prop determine C and KC}(1), $\Sing_1(X)$ consists of a crepant curve $C_1$ of type $\mathsf{A}_4$ with $(-r_XK_X\cdot C_1)=18$ and several crepant curves of type $\mathsf{A}_1$. Note that there is at least one crepant curve of type $\mathsf{A}_1$ by Corollary~\ref{cor A primitive} as $2\mid J_A$. In particular, $x_{\mathsf{A}_1}>0$.

    Take $r'=40$. Then $(-r'K_X\cdot C_1)=40$ is divided by $5$. So by Corollary~\ref{cor RR integer general} for $D=sA$ and Remark~\ref{rem ignore part}, we have 
    \begin{align*}
        \mathbb{Z}\ni {}&  -\frac{r's^2}{2}A^2\cdot K_X +\sum_{\substack{C\subset \Sing(X)\\
        \text{type }\mathsf{A}_1}}(-r'K_X\cdot C)c_C(sA)-\frac{r'a_s(9-a_s)}{18}\\
        ={}&\frac{s^2}{9}-\frac{5x_{\mathsf{A}_1}(\overline{s})_2}{9}-\frac{20a_s(9-a_s)}{9}  
    \end{align*}
    for some integer $a_s$. Here we used \eqref{eq xa1 formula 1}. In particular, for $s=1,3$, we have
    \begin{align*}    
        5x_{\mathsf{A}_1}\equiv 1+2a_1^2\equiv 2a_3^2\bmod 9.
    \end{align*}
    By considering square residues, we have 
    \begin{align*}
        (\overline{1+2a_1^2})_9\in {}& \{0,1,3,6\};\\
        (\overline{2a_3^2})_9\in {}& \{0,2,5,8\}.  
    \end{align*}
    So this congruence equation implies that $9\mid x_{\mathsf{A}_1}$. In particular, $x_{\mathsf{A}_1}\geq 9$ as $x_{\mathsf{A}_1}>0$. But then by \eqref{eq xa1 formula 2},
    \begin{align*} 
        {}&\sum_{C\subset \text{\rm Sing}(X)}\left(j_C-\frac{1}{j_C}\right)(-r_XK_X\cdot C) 
        =\frac{3x_{\mathsf{A}_1}}{2}+\left(5-\frac{1}{5}\right)18\geq 99.9>\nabla_X,
    \end{align*}
    which contradicts Theorem~\ref{thm Delta>jC}.

    \medskip

    Consider \textnumero 32 and \textnumero33. By Corollary~\ref{cor A primitive}, $j_C$ divides $J_A=6$ for any crepant curve $C\subset\Sing(X)$. On the other hand,  $\LB(6)=105$ by Algorithm~\ref{algo LB}, so by Theorem~\ref{thm Delta>jC}, $X$ only contains crepant curves of types $\mathsf{A}_1, \mathsf{A}_2$ (at least one of each). 
    
    As $\LB(3)=35$, by Theorem~\ref{thm c>=LB} we may write 
    \[
        \sum_{\substack{C\subset\Sing(X)\\\text{type }\mathsf{A}_2}}(-r_XK_X\cdot C)=35y
    \]
    for some positive integer $y$. Applying Corollary~\ref{cor.canpart} for $D=2A$, we have
    \[
        -4r_XA^2\cdot K_X+ \sum_{\substack{C\subset\Sing(X)\\\text{type }\mathsf{A}_2}}2(-r_XK_X\cdot C)c_C(2A) =\frac{2}{3}-\frac{70y}{3}\in \mathbb Z,
    \]
  Here we used \eqref{eq xa1 formula 1} and Example~\ref{ex cc A123}. So $y\geq 2.$

    Take $r'=18$, then for any crepant curve $C'$ of type $\mathsf{A}_2$, as $35\mid (-r_XK_X\cdot C')$, we have that $(-r'K_X\cdot C')\in 6\mathbb{Z}$ is divided by $3$. So by Corollary~\ref{cor RR integer general} for $D=sA$ and Remark~\ref{rem ignore part}, we have 
    \begin{align}
        \mathbb{Z}\ni {}&  -\frac{r's^2}{2}A^2\cdot K_X +\sum_{\substack{C\subset \Sing(X)\\
        \text{type }\mathsf{A}_1}}(-r'K_X\cdot C)c_C(sA)-\frac{r'a_s(5-a_s)}{10}-\frac{r'b_s(7-b_s)}{14}\notag\\
        ={}&\frac{s^2}{70}-\frac{3x_{\mathsf{A}_1}(\overline{s})_2}{70}-\frac{9a_s(5-a_s)}{5}-\frac{9b_s(7-b_s)}{7}\label{eq no 32,33}
    \end{align}
    for some integers $a_s, b_s$. Here we used \eqref{eq xa1 formula 1}.
    In particular, multiplying \eqref{eq no 32,33} by $14$ and $10$ respectively, we have
    \begin{align*}
        \frac{s^2}{5}-\frac{3x_{\mathsf{A}_1}(\overline{s})_2}{5}+\frac{a_s^2}{5}\in \mathbb{Z},\\
        \frac{s^2}{7}-\frac{3x_{\mathsf{A}_1}(\overline{s})_2}{7}+\frac{6b_s^2}{7}\in \mathbb{Z}. 
    \end{align*}
    For $s=1,3,5$, we have
    \begin{align*}
        &3x_{\mathsf{A}_1}\equiv 1+a_1^2\equiv 4+a_3^2\bmod 5;\\
        &3x_{\mathsf{A}_1}\equiv 1+6b_1^2\equiv 2+6b_3^2\equiv 4+6b_5^2\bmod 7.
    \end{align*}
    These congruence equations imply that $35\mid x_{\mathsf{A}_1}.$ In particular, $x_{\mathsf{A}_1}\geq 35$ as $x_{\mathsf{A}_1}>0$. But then by \eqref{eq xa1 formula 2},
    \begin{align*}
        {}&\sum_{C\subset \text{\rm Sing}(X)}\left(j_C-\frac{1}{j_C}\right)(-r_XK_X\cdot C) =\frac{3x_{\mathsf{A}_1}}{2}+\left(3-\frac{1}{3}\right)35y>239>\nabla_X,
    \end{align*}
    which contradicts Theorem~\ref{thm Delta>jC}.
\end{proof}

\begin{prop}\label{prop rule out B3}
    \textnumero 24 in Table~\ref{main tab} does not exist. 
\end{prop}

\begin{proof}
    By Corollary~\ref{cor A primitive}, $j_C$ divides $J_A=12$ for any crepant curve $C\subset\Sing(X)$. On the other hand,  $\LB(6)=15$ by Algorithm~\ref{algo LB}, so by Theorem~\ref{thm Delta>jC}, $X$ only contains crepant curves of types $\mathsf{A}_1, \mathsf{A}_2, \mathsf{A}_3$. By Corollary~\ref{cor A primitive}, there is at least one crepant curve of type $\mathsf{A}_2$ (resp., $\mathsf{A}_3$). 
 
    As $\LB(4)=\LB(3)=5$, by Theorem~\ref{thm c>=LB} we may write
    \begin{align*}
        \sum_{\substack{C\subset \Sing(X)\\\text{type }\mathsf{A}_2}}(-r_XK_X\cdot C)=5y_3;\\
        \sum_{\substack{C\subset \Sing(X)\\\text{type }\mathsf{A}_3}}(-r_XK_X\cdot C)=5y_4 
    \end{align*}
    for some positive integers $y_3, y_4$. Then by Theorem~\ref{thm Delta>jC} and \eqref{eq xa1 formula 2},
    \begin{align}\label{eq no 24 ineq}
        \nabla_X\geq \frac{3x_{\mathsf{A}_1}}{2}+\left(3-\frac{1}{3}\right)5y_3+\left(4-\frac{1}{4}\right)5y_4.
    \end{align}
    In particular, we have $x_{\mathsf{A}_1}<20$ and $y_4\leq 2$. We will show that $x_{\mathsf{A}_1}=10$ and $y_3=y_4=1$. 
    
    Take $r'=9$, then for any crepant curve $C'$ of type $\mathsf{A}_2$, as $5\mid (-r_XK_X\cdot C')$, we have that $(-r'K_X\cdot C')\in 3\mathbb{Z}$ is divided by $3$. So by Corollary~\ref{cor RR integer general} for $D=sA$ and Remark~\ref{rem ignore part}, we have for $s$ odd,
    \begin{align}
        \mathbb{Z}\ni {}&  -\frac{r's^2}{2}A^2\cdot K_X +\sum_{\substack{C\subset \Sing(X)\\
       \text{type }\mathsf{A}_1 \text{ or }\mathsf{A}_3}}(-r'K_X\cdot C)c_C(sA)-\frac{r'a_s(5-a_s)}{10}\notag\\
       ={}&\frac{s^2}{40}-\frac{3x_{\mathsf{A}_1}}{20}-\frac{9y_4}{8}-\frac{9a_s(5-a_s)}{10} \label{eq no 24}
    \end{align}
    for some integer $a_s$. Here we used \eqref{eq xa1 formula 1} and Example~\ref{ex cc A123}. Multiplying \eqref{eq no 24} by $5$ and taking $s=1$, we have
    \[
        \frac{1}{8}-\frac{3x_{\mathsf{A}_1}}{4}-\frac{5y_4}{8}\in \mathbb{Z}. 
    \]
    Combining with the fact that $0<y_4\leq 2$, we have $y_4=1$,  $x_{\mathsf{A}_1}\neq 0$, and $2\mid x_{\mathsf{A}_1}$. Multiplying \eqref{eq no 24} by $8$ and taking $s=1,3$, we have
    \[
        x_{\mathsf{A}_1}\equiv 1+a_1^2\equiv 4+a_3^2\bmod 5.
    \] 
    This congruence equation implies that $5\mid x_{\mathsf{A}_1}$. Combining with the facts that $2\mid x_{\mathsf{A}_1}$,  $x_{\mathsf{A}_1}\neq 0$, and $x_{\mathsf{A}_1}<20$, we have $x_{\mathsf{A}_1}=10$. Then by \eqref{eq no 24 ineq}, $y_3=1$. 

    Then $\Sing_1(X)$ consists of crepant curves $C_1, C_2$ of type $\mathsf{A}_2$, $\mathsf{A}_3$ with $(-r_XK_X\cdot C_i)=5$ for $i=1, 2$ and several crepant curves of type $\mathsf{A}_1$. By Theorem~\ref{thm.rr.sA}, for $0<s<72$, we have
    \begin{align*}
        h^0(X, \mathcal{O}_X(sA))={}&\frac{s^2}{360}+2-\frac{(\overline{s})_2}{6}-\frac{1}{3}\frac{(\overline{s})_3(\overline{-s})_3}{6} -\frac{1}{3}\frac{(\overline{s})_4(\overline{-s})_4}{8}\\{}&-\frac{(\overline{a_s})_5(\overline{-a_s})_5}{10}-\frac{(\overline{b_s})_3(\overline{-b_s})_3}{6}-\frac{(\overline{c_s})_3(\overline{-c_s})_3}{6}-\frac{(\overline{d_s})_3(\overline{-d_s})_3}{6}
    \end{align*}
    for some integers $a_s,b_s, c_s,d_s$. Here we used \eqref{eq xa1 formula 1} and Example~\ref{ex cc A123}. 
    
    For $s\in \{2,3,6,30,31\}$, we can find all integers $a_s, b_s, c_s, d_s$ satisfying $h^0(X, \mathcal{O}_X(sA))\in \mathbb{Z}$, which give the following results: 
    \begin{align}
        {}&h^0(X, \mathcal{O}_X(2A))=h^0(X, \mathcal{O}_X(3A))=h^0(X, \mathcal{O}_X(6A))=1, \label{eq no 24 2}\\ {}& h^0(X, \mathcal{O}_X(30A))=4, h^0(X, \mathcal{O}_X(31A))=3. \label{eq no 24 3}
    \end{align}
    Here \eqref{eq no 24 2} implies that $h^0(X, \mathcal{O}_X(A))>0$: indeed, take the unique elements $A_2\in |2A|$ and $A_3\in |3A|$, then $3A_2=2A_3\in |6A|$, which implies that $A_3-A_2\geq 0$ is an element in $|A|$. On the other hand, \eqref{eq no 24 3} implies that $h^0(X, \mathcal{O}_X(A))=0$, which is a contradiction. 
\end{proof}
    
\begin{prop}\label{prop rule out B4}
    \textnumero 27 in Table~\ref{main tab} does not exist. 
\end{prop}

\begin{proof}
    By Corollary~\ref{cor A primitive}, all crepant curves are of type $\mathsf{A}_2$.  As $\LB(3)=70$, by Theorem~\ref{thm c>=LB} we may write 
    \[
        \sum_{ C\subset\Sing(X) }(-r_XK_X\cdot C)=70y
    \]
    for some positive integer $y$. Then by Theorem~\ref{thm Delta>jC}, we have $\nabla_X\geq (3-\frac{1}{3})70y$, which implies that $y \leq 2$. 

    By applying Corollary~\ref{cor.canpart} to $D=A$, we have
    \[
       -r_XA^2\cdot K_X+ \sum_{ C\subset\Sing(X)}2(-r_XK_X\cdot C)c_C(A) =  \frac{1}{3}-\frac{140y}{3}  \in \mathbb Z.
    \]  
  Here we used Example~\ref{ex cc A123}.  So we conclude that $y=2$. We are in one of the following cases: 
    \begin{enumerate}
        \item $X$ has a unique crepant curve $C_1$ of type $\mathsf{A}_2$ with $(-K_X\cdot C_1)=\frac{70y}{r_X}=\frac{2}{3}$; or
        
        \item $X$ has exactly two crepant curves $C_1, C_2$ of type $\mathsf{A}_2$ with $(-K_X\cdot C_i)=\frac13$ for $i=1,2$. 
    \end{enumerate}
  
Here we need to further consider local indices of Weil pullbacks to derive a contradiction.  Take $f\colon Y\to X$ to be a sequential terminalization. We denote by $\{Q_3, Q_5, Q_6, Q_7\}$ the orbifold points in $B_X$ labeled by their Gorenstein indices. 

    \begin{claim}\label{claim no 27}
        For any prime $f$-exceptional divisor $E$,  we have 
        \[\begin{cases}  \mathsf{i}_{E,Q_6}\equiv 0\bmod 6& \text{if in Case (1)};\\
            \mathsf{i}_{E,Q_3}\equiv 0\bmod 3& \text{if in Case (2)}.
        \end{cases}\]
    \end{claim}
    
    \begin{proof}
        By \cite{kawakita}*{Theorem~1.1, Table~2} (see also Remark~\ref{rem kawakita}), $X$ has no non-Gorenstein crepant point. If $f(E)$ is a point, then $f(E)$ is a Gorenstein crepant point on $X$, which means that $E$ contains no non-Gorenstein point. Hence $\mathsf{i}_{E,Q_3}\equiv 0\bmod 3$ and $\mathsf{i}_{E,Q_6}\equiv 0\bmod 6$ both hold.

        So without loss of generality, we may assume that $E$ is centered at $C_1$.

        In Case (1), by Proposition~\ref{prop.KC in Z}, we have
        \begin{align*}
            \mathbb{Z}\ni{}& -\frac{2}{3}-\frac{{\mathsf{i}_{E, Q_3}}(3-{\mathsf{i}_{E, Q_3}})}{6}-\frac{{\mathsf{i}_{E, Q_6}}(6-{\mathsf{i}_{E, Q_6}})}{12}\\{}&-\frac{{2\mathsf{i}_{E, Q_5}}(5-2\mathsf{i}_{E, Q_5})}{10}-\frac{{\mathsf{i}_{E, Q_7}}(7-{\mathsf{i}_{E, Q_7}})}{14}.
        \end{align*}
        By considering $5$ and $7$ in the denominators, we have $\mathsf{i}_{E, Q_5}\equiv 0\bmod 5$ and $\mathsf{i}_{E, Q_7}\equiv 0\bmod 7$, which implies that 
        \begin{align*}
            \mathbb{Z}\ni{}& -\frac{2}{3}-\frac{{\mathsf{i}_{E, Q_3}}(3-{\mathsf{i}_{E, Q_3}})}{6}-\frac{{\mathsf{i}_{E, Q_6}}(6-{\mathsf{i}_{E, Q_6}})}{12}. 
        \end{align*}
        This congruence equation implies that $\mathsf{i}_{E, Q_3}\equiv \pm 1\bmod 3$ and hence $\mathsf{i}_{E, Q_6}\equiv 0\bmod 6$. Indeed, if $\mathsf{i}_{E, Q_3}\equiv 0\bmod 3$, then we have 
        \[
            12\mid 8+{{\mathsf{i}_{E, Q_6}}(6-{\mathsf{i}_{E, Q_6}})},
        \]
        which is impossible. 

        In Case (2), by the same argument we have 
        \begin{align*}
            \mathbb{Z}\ni{}& -\frac{1}{3}-\frac{{\mathsf{i}_{E, Q_3}}(3-{\mathsf{i}_{E, Q_3}})}{6}-\frac{{\mathsf{i}_{E, Q_6}}(6-{\mathsf{i}_{E, Q_6}})}{12}. 
        \end{align*}
        This congruence equation implies that $\mathsf{i}_{E, Q_3}\equiv 0\bmod 3$. Indeed, if $\mathsf{i}_{E, Q_3}\equiv \pm1\bmod 3$, then again we have 
        \[
            12\mid 8+{{\mathsf{i}_{E, Q_6}}(6-{\mathsf{i}_{E, Q_6}})},
        \]
        which is impossible. 
    \end{proof}

    Take $r'=70$, then by Corollary~\ref{cor RR integer general} for $D=sA$ and Remark~\ref{rem ignore part}, we have
    \begin{align*}
       \frac{s^2}{18}-\frac{70(\overline{s})_3(\overline{-s})_3}{9}-\frac{35{\mathsf{i}_{f^{\lfloor*\rfloor}(sA), Q_3}}(3-{\mathsf{i}_{f^{\lfloor*\rfloor}(sA), Q_3}})}{3}-\frac{35{\mathsf{i}_{f^{\lfloor*\rfloor}(sA), Q_6}}(6-{\mathsf{i}_{f^{\lfloor*\rfloor}(sA), Q_6}})}{6}\in \mathbb{Z}. 
    \end{align*}
    For $s=2,4$, we have
    \begin{align*}
        4+4\mathsf{i}_{f^{\lfloor*\rfloor}(2A), Q_3}^2+5\mathsf{i}_{f^{\lfloor*\rfloor}(2A), Q_6}^2\equiv 0\bmod 6;\\
         2+4\mathsf{i}_{f^{\lfloor*\rfloor}(4A), Q_3}^2+5\mathsf{i}_{f^{\lfloor*\rfloor}(4A), Q_6}^2\equiv 0\bmod 6.
    \end{align*}
    These congruence equations imply that
    \begin{align}
        \mathsf{i}_{f^{\lfloor*\rfloor}(2A), Q_3}\equiv 0, \quad \mathsf{i}_{f^{\lfloor*\rfloor}(4A), Q_3}\equiv \pm1  \bmod 3;\label{eq 27 Q3}\\
        \mathsf{i}_{f^{\lfloor*\rfloor}(2A), Q_6}\equiv \pm 2, \quad \mathsf{i}_{f^{\lfloor*\rfloor}(4A), Q_6}\equiv0 \bmod 6.\label{eq 27 Q6}
    \end{align}

    Now consider $G\coloneq f^{\lfloor*\rfloor}(4A)-2f^{\lfloor*\rfloor}(2A)$ which is an $f$-exceptional integral Weil divisor by construction. By \eqref{eq 27 Q3}, \eqref{eq 27 Q6}, and the definition of local index,
    \begin{align*}
        \mathsf{i}_{G,Q_3}\equiv \mathsf{i}_{f^{\lfloor*\rfloor}(4A), Q_3}-2\mathsf{i}_{f^{\lfloor*\rfloor}(2A), Q_3}\equiv \pm 1 \bmod 3;\\
         \mathsf{i}_{G,Q_6}\equiv \mathsf{i}_{f^{\lfloor*\rfloor}(4A), Q_6}-2\mathsf{i}_{f^{\lfloor*\rfloor}(2A), Q_6} \equiv \pm 2\bmod 6.
    \end{align*}
    On the other hand, by Claim~\ref{claim no 27}, either $\mathsf{i}_{G, Q_3}\equiv 0\bmod 3$ or  $\mathsf{i}_{G, Q_6}\equiv 0\bmod 6$, which is a contradiction.
\end{proof}

\begin{prop}\label{prop rule out B5}
    \textnumero 35 in Table~\ref{main tab} does not exist. 
\end{prop}

\begin{proof}
    By Proposition~\ref{prop determine C and KC}(2),  {$\Sing_1(X)$} consists of a crepant curve $C_1$ of type $\mathsf{A}_3$ with $(-r_XK_X\cdot C_1)=35$ and possibly several crepant curves of type $\mathsf{A}_1$. 

    Let $f\colon Y\to X$ be a sequential terminalization. We denote by $\{Q_1, Q_2, Q_3, Q_4\}$ the four orbifold points $(2,1)$ in $B_X$. Applying Corollary~\ref{cor RR integer general} to $D=sA$ and $r'=1$, we have
    \begin{align}\label{eq.222257}
        \begin{split}    \mathbb{Z}\ni {}&\frac{s^2}{560}-\frac{x_{\mathsf{A}_1}(\overline{s})_2}{280}-\frac{(\overline{s})_4(\overline{-s})_4}{16}-
        \sum_{i=1}^4\frac{\mathsf{i}_{f^{\lfloor*\rfloor}(sA), Q_i}(2-\mathsf{i}_{f^{\lfloor*\rfloor}(sA), Q_i})}{4}\\{}&-\frac{a_s(5-a_s)}{10}-\frac{b_s(7-b_s)}{14} 
        \end{split} 
    \end{align}
    for some integers $a_s, b_s$.  Here we used \eqref{eq xa1 formula 1} and Example~\ref{ex cc A123}. Multiplying \eqref{eq.222257} by $80$ and $112$ respectively, we have
    \begin{align*}
        \frac{s^2}{7}-\frac{2x_{\mathsf{A}_1}(\overline{s})_2}{7}+\frac{5b_s^2}{7}\in \mathbb{Z};\\
         \frac{s^2}{5}-\frac{2x_{\mathsf{A}_1}(\overline{s})_2}{5}+\frac{a_s^2}{5}\in \mathbb{Z}.
    \end{align*}
    For $s=1,3,5$, we have 
    \begin{align*}
        {}& 2x_{\mathsf{A}_1}\equiv 1+a_1^2\equiv 4+a_3^2\bmod 5;\\
        {}& 2x_{\mathsf{A}_1}\equiv 1+5b_1^2\equiv 2+5b_3^2  \equiv 4+5b_5^2\bmod 7.
    \end{align*}
    These congruence equations imply that $35\mid x_{\mathsf{A}_1}$. Hence $x_{\mathsf{A}_1}=0$ as $\nabla_X<\frac{3}{2}\cdot 35+(4-\frac{1}{4})\cdot 35$. 

    Multiplying \eqref{eq.222257} by $35$, we have
    \begin{align*}
        \frac{s^2}{16} -\frac{35(\overline{s})_4(\overline{-s})_4}{16}-
        \sum_{i=1}^4\frac{35\mathsf{i}_{f^{\lfloor*\rfloor}(sA), Q_i}(2-\mathsf{i}_{f^{\lfloor*\rfloor}(sA), Q_i})}{4} 
       \in \mathbb Z.
    \end{align*}
    For $s=1$, we have
    \begin{align*}
        \frac{1}{2} +\sum_{i=1}^4\frac{\mathsf{i}_{f^{\lfloor*\rfloor}(A), Q_i}(2-\mathsf{i}_{f^{\lfloor*\rfloor}(A), Q_i})}{4} 
       \in \mathbb Z,
    \end{align*}
    which means that, up to a permutation of $Q_i$, we have 
    \begin{align}\label{eq 35 1100}
        (\mathsf{i}_{f^{\lfloor*\rfloor}(A), Q_1}, \mathsf{i}_{f^{\lfloor*\rfloor}(A), Q_2}, \mathsf{i}_{f^{\lfloor*\rfloor}(A), Q_3}, \mathsf{i}_{f^{\lfloor*\rfloor}(A), Q_4})\equiv (1,1,0,0) \bmod 2.
    \end{align}
    For $s=4,5$, we have 
    \begin{align}\label{eq 35 0000}
        \begin{cases}
            \mathsf{i}_{f^{\lfloor*\rfloor}(4A), Q_1}\equiv \mathsf{i}_{f^{\lfloor*\rfloor}(4A), Q_2}\equiv \mathsf{i}_{f^{\lfloor*\rfloor}(4A), Q_3}\equiv \mathsf{i}_{f^{\lfloor*\rfloor}(4A), Q_4}\bmod 2;\\
            \mathsf{i}_{f^{\lfloor*\rfloor}(5A), Q_1}\equiv \mathsf{i}_{f^{\lfloor*\rfloor}(5A), Q_2}\equiv \mathsf{i}_{f^{\lfloor*\rfloor}(5A), Q_3}\equiv \mathsf{i}_{f^{\lfloor*\rfloor}(5A), Q_4}\bmod 2.
        \end{cases}   
    \end{align}
  
    \begin{claim}\label{claim 35 F}
        Take $G\coloneq f^{\lfloor*\rfloor}(5A)-f^{\lfloor*\rfloor}(A)-f^{\lfloor*\rfloor}(4A). $
        Then 
        \[
            \mathsf{i}_{G,Q_1}\equiv \mathsf{i}_{G,Q_2}\equiv \mathsf{i}_{G,Q_3}\equiv \mathsf{i}_{G,Q_4}\bmod 2.
        \]
    \end{claim} 
    \begin{proof}
       As $J_A=4$, take $U\subset X$ to be an open subset such that $4A$ is Cartier on $U$ and $X\setminus U$ has dimension $0$. Then $f$ induces a sequential terminalization $f^{-1}(U)\to U$. Applying Lemma~\ref{lem pullback of D+G} to $f^{-1}(U)\to U$, we know that over $U$ the equality $f^{\lfloor*\rfloor}(5A)=f^{\lfloor*\rfloor}(A)+f^{\lfloor*\rfloor}(4A)$ holds. 

        Fix an irreducible component $E$ of $G$, then $f(E)=P\in X\setminus U$ is a (crepant) point on $X$. It suffices to show that
        \begin{align}\label{eq E0000}
            \mathsf{i}_{E, Q_1}\equiv \mathsf{i}_{E, Q_2}\equiv \mathsf{i}_{E, Q_3}\equiv \mathsf{i}_{E, Q_4}\bmod 2. 
        \end{align}
        We may assume that $\mathsf{i}_{E,Q_{i_0}}\not\equiv 0\bmod 2$ for some $1\leq i_0\leq 4$. Then $P$ is a non-Gorenstein crepant point on $X$, whose Gorenstein index is $2$ by \cite{kawakita}*{Theorem~1.1, Table~2} (see also Remark~\ref{rem kawakita}). Then we get \eqref{eq E0000} by Example~\ref{ex.2222}.
    \end{proof}
    Now \eqref{eq 35 0000} and Claim~\ref{claim 35 F} imply that 
    \[
        \mathsf{i}_{f^{\lfloor*\rfloor}(A), Q_1}\equiv \mathsf{i}_{f^{\lfloor*\rfloor}(A), Q_2}\equiv \mathsf{i}_{f^{\lfloor*\rfloor}(A), Q_3}\equiv \mathsf{i}_{f^{\lfloor*\rfloor}(A), Q_4}\bmod 2,
    \]
    which contradicts \eqref{eq 35 1100}. 
\end{proof}

\section{Ruling out remaining cases: Group C}\label{sec group C1}

In this section, we treat Group C which contains \textnumero3, \textnumero4, \textnumero6--8, \textnumero11--15, \textnumero21, \textnumero22 in Table~\ref{main tab}. 
The main theorem of this section (Theorem~\ref{thm RR for Group C}) is to give the explicit form of the formula for $h^0(X, \mathcal{O}_X(sA))$ in Theorem~\ref{thm.rr.sA}. 
It turns out that all candidates in Group C share the same formula, although they have different Reid's baskets and crepant curves.

For \textnumero3, \textnumero4, \textnumero6--8, \textnumero11, we can check that the conditions of Proposition~\ref{prop determine C and KC} are satisfied. Then we can describe the types of all crepant curves $C$ in $\Sing(X)$ which are not of type $\mathsf{A}_1$ and the corresponding $(-r_XK_X\cdot C)$ in Table~\ref{tab group C}, where the symbol $(\mathsf{A}_1)$ means that there might be several crepant curves of type $\mathsf{A}_1$ in $\Sing(X)$. 
For \textnumero12--15, there are only  crepant curves of type $\mathsf{A}_1$ by Corollary~\ref{cor A primitive} as $J_A=2$.
For \textnumero 21 and \textnumero 22, there is no crepant curve by Corollary~\ref{cor A primitive} as $J_A=1$.

{\footnotesize
\begin{longtable}{LLLLLLLLL}
    \caption{Numerical data in Group C}\label{tab group C}\\
    \hline
    \text{\textnumero}  & B_X & q & r_X & r_Xc_1^3 & r_Xc_2c_1 &  C\subset\Sing(X)& (-r_XK_X\cdot C) &\nabla_X  \\
    \hline
    \endfirsthead
    \multicolumn{4}{l}{{ {\bf \tablename\ \thetable{}} \textrm{-- continued}}}
    \\
    \hline 
    $\textnumero$ & B_X & q & r_X & r_Xc_1^3 & r_Xc_2c_1 &  C\subset\Sing(X)& (-r_XK_X\cdot C) &\nabla_X \\
    \hline 
    \endhead
    \hline
    \hline \multicolumn{4}{c}{{\textrm{Continued on next page}}} \\ \hline
    \endfoot
    
    \hline \hline
    \endlastfoot
    3& \{(3,1),(11,3)\}& 70 & 33 & 490 & 344 &\mathsf{A}_4, (\mathsf{A}_1)&r_X &218.1 \\
    4& \{(3,1),(11,5)\}& 80 & 33 & 640 & 344 &\mathsf{A}_4, (\mathsf{A}_1)&r_X &180.1 \\
    6&  \{(5,1),(11,1)\}& 72 & 55 & 864 & 456 &\mathsf{A}_2, (\mathsf{A}_1) &r_X & 234.17 \\
    7&  \{(5,1),(11,2)\}& 78 & 55 & 1014 & 456 &\mathsf{A}_2, (\mathsf{A}_1) &r_X & 196.17 \\
    8&  \{(5,2),(11,3)\}& 84 & 55 & 1176 & 456 &\mathsf{A}_2, (\mathsf{A}_1) &r_X & 155.17  \\
    11&  \{(2,1),(5,2),(11,5)\}& 69 & 110 & 1587 & 747 &\mathsf{A}_2 &r_X &339.09 \\
    12&  \{(3,1),(5,1),(11,1)\}& 82 & 165 & 3362 & 928 &(\mathsf{A}_1) & & 67.5 \\
    13&  \{(3,1),(5,1),(11,4)\}& 68 & 165 & 2312 & 928 &(\mathsf{A}_1) & & 333.5\\
    14&  \{(3,1),(5,2),(11,2)\}& 76 & 165 & 2888 & 928 &(\mathsf{A}_1) & &187.5 \\
    15&  \{(3,1),(5,2),(11,5)\}& 74 & 165 & 2738 & 928 &(\mathsf{A}_1) & & 225.5 \\
    21&  \{(2,1),(3,1),(5,1),(11,2)\}& 67 & 330 & 4489 & 1361 &\emptyset & &206.25 \\
    22&  \{(2,1),(3,1),(5,2),(11,1)\}& 71 & 330 & 5041 & 1361 &\emptyset & &66.25 \\
\end{longtable}
} 
  
To unify the notation, we make the following settings:
\begin{set}\label{set group C}
    \begin{enumerate}
        \item Let $(X,A)$ be an unmodifiable canonical Fano $3$-fold in Group C and let $f\colon Y\to X$ be a sequential terminalization.
    
        \item Write $B_X=\{(r_1, b_1), \dots, (r_k, b_k)\}$ for some $2\leq k\leq 4$. 
         
        \item From Table~\ref{tab group C}, $X$ contains at most one crepant curve which is not of type $\mathsf{A}_1$; if such a crepant curve exists, then we denote it by $C_0$ which is of type $\mathsf{A}_{r_0-1}$ with $r_0>2$ and $(-K_X\cdot C_0)=1$.
    \end{enumerate}
\end{set}

In order to determine the formula in Theorem~\ref{thm.rr.sA}, we need to understand the local indices of $f^{\lfloor *\rfloor}(sA)$ for each $s$. In general, this is very difficult as the Weil pullback is not linear, but as we already have enough information about the crepant curves of $X$, we can get some useful information about the local indices of $f^{\lfloor *\rfloor}(sA)$ in the following two lemmas.

\begin{lem}\label{lem E=0}
    Keep Setting~\ref{set group C}. Suppose that $C_0$ exists. Then for any prime $f$-exceptional divisor $E$ on $Y$ centered at $C_0$, we have $\mathsf{i}_{E,(r_i,b_i)}\equiv 0\bmod r_i$ for any $(r_i,b_i)\in B_X$.
\end{lem}

\begin{proof}
    Note that $(-K_X\cdot C_0)=1$, so by \eqref{eq.KCE},
    \[
        \sum^k_{i=1}\frac{\mathsf{i}_{E,(r_i,b_i)}b_i(r_i-\mathsf{i}_{E,(r_i,b_i)}b_i)}{r_i}\in \mathbb Z.
    \]
    As $r_i\in\{2,3,5,11\}$ in the denominators are distinct prime numbers and $b_i$ is coprime to $r_i$, we have $\mathsf{i}_{E,(r_i,b_i)}\equiv 0\bmod r_i$ for every $(r_i,b_i)\in B_X$.
\end{proof}

The following lemma shows that, although the local indices of $f^{\lfloor *\rfloor}(sA)$ might not be linear in $s$, they are linear after adding a correction term which is a periodic function in $s$ of period $2$. 

\begin{lem}\label{lem asa2a1}
    Keep Setting~\ref{set group C}. Then for any integer $s$ and any $(r_i,b_i)\in B_X$, we have
    \[
        \mathsf{i}_{f^{\lfloor *\rfloor}(sA), (r_i,b_i)}\equiv \left\lfloor\frac{s}{2}\right\rfloor  \cdot \mathsf{i}_{f^{\lfloor *\rfloor}(2A), (r_i,b_i)} +(\overline{s})_2 \cdot\mathsf{i}_{f^{\lfloor *\rfloor}(A), (r_i,b_i)} \bmod r_i.
    \]
\end{lem}

\begin{proof}
    Fix an integer $s$. Denote 
    \[
        G\coloneq f^{\lfloor *\rfloor}(sA)-\left\lfloor\frac{s}{2}\right\rfloor f^{\lfloor *\rfloor}(2A)-(\overline{s})_2f^{\lfloor *\rfloor}(A).
    \]
    The assertion is equivalent to showing that $\mathsf{i}_{G, (r_i,b_i)}\equiv 0 \bmod r_i$ for any $(r_i,b_i)\in B_X$. Take $E$ to be an irreducible component of $G$, then it suffices to show that $\mathsf{i}_{E, (r_i,b_i)}\equiv 0 \bmod r_i$ for any $(r_i,b_i)\in B_X$. By construction, $E$ is $f$-exceptional.

    By \cite{kawakita}*{Theorem~1.1, Table~2} (see also Remark~\ref{rem kawakita}), $X$ has no non-Gorenstein crepant point. We may take $Z\subset X$ to be the union of all Gorenstein crepant points and $C_0$ (if exists). Take $U=X\setminus Z$, then all crepant centers of $U$ are crepant curves of type $\mathsf{A}_1$ by construction. In particular, $f^*(2A)$ is an integral Weil divisor on $f^{-1}(U)$. Repeatedly applying Lemma~\ref{lem pullback of D+G} to $D'=2A$  for the sequential terminalization $f^{-1}(U)\to U$, we know that over $U$ the equality $G=0$ holds. This implies that $f(E)\subset Z$. Namely, $f(E)$ is either a Gorenstein crepant point or $C_0$.

    If $f(E)$ is a Gorenstein crepant point, then it is clear that $\mathsf{i}_{E, (r_i,b_i)}\equiv 0 \bmod r_i$ for any $(r_i,b_i)\in B_X$ as $E$ contains no non-Gorenstein point.

    If $f(E)=C_0$, then  $\mathsf{i}_{E, (r_i,b_i)}\equiv 0 \bmod r_i$ for any $(r_i,b_i)\in B_X$ by Lemma~\ref{lem E=0}.
\end{proof}

After all the preparations, we have got most of the information on $(X, A)$ in Group C, especially on singularities of $X$. Combining with certain integrality constraints, we can finally give the explicit form of the formula for $h^0(X, \mathcal{O}_X(sA))$.

\begin{thm}\label{thm RR for Group C}
    Let $(X,A)$ be an unmodifiable canonical Fano $3$-fold in Group C. Then for $0<s<\qQ(X)$, we have 
    \begin{align*}
     h^0(X,\mathcal O_X(sA))= 
        \frac{s^2}{660}+2-\frac{(\overline{s})_2(\overline{-s})_2}{4}-\frac{(\overline{s})_3(\overline{-s})_3}{6}-\frac{(\overline{2s})_5(\overline{-2s})_5}{10}-\frac{(\overline{2s})_{11}(\overline{-2s})_{11}}{22}. 
    \end{align*}
\end{thm}

\begin{proof}
    For a positive integer $r$ and an integer $x$, denote $F_r(x)\coloneq \frac{(\overline{x})_r(\overline{-x})_r}{2r}$. Note that $F_r\colon  \mathbb{Z}\to \mathbb{Q}$ is a periodic even function with period $r$ and $F_r(0)=0$. 

    In all cases, we have $-A^2\cdot K_X=\frac{c_1(X)^3}{\qQ(X)^2}=\frac{1}{330}$. So by Theorem~\ref{thm.rr.sA}, for $0<s<\qQ(X)$, we may write 
    \begin{align}\label{eq group C RR1}
        h^0(X,\mathcal O_X(sA))={}& 
        \frac{s^2}{660}+2+R_1(s)+R_2(s),
    \end{align}
    where we write the contributions of singularities into two parts:
    \begin{align*}
        R_1(s)\coloneq{}&\sum_{\substack{C\subset \Sing(X)\\\text{type }\mathsf{A}_1}}(-K_X\cdot C)c_C(sA)-\sum_{\substack{(r_i, b_i)\in B_X\\r_i=2}}F_{r_i}({\mathsf{i}_{f^{\lfloor *\rfloor}(sA), (r_i,b_i)}b_i});\\ 
        R_2(s)\coloneq{}&\sum_{\substack{C\subset \Sing(X)\\\text{not type }\mathsf{A}_1}}(-K_X\cdot C)c_C(sA)-\sum_{\substack{(r_i, b_i)\in B_X\\r_i\neq 2}}F_{r_i}({\mathsf{i}_{f^{\lfloor *\rfloor}(sA), (r_i,b_i)}b_i}).
    \end{align*}
    
    \medskip 

    {\bf Step 1}. We determine general forms of $R_1(s)$ and $R_2(s)$.  

    By Lemma~\ref{lem asa2a1}, for any $(r_i,b_i)\in B_X$, we may write
    \begin{align}
        \mathsf{i}_{f^{\lfloor *\rfloor}(sA), (r_i,b_i)}b_i\equiv  \left\lfloor\frac{s}{2}\right\rfloor a_i +(\overline{s})_2a'_i \bmod r_i \label{eq ibaa'}
    \end{align}
    for some integers $a_i, a'_i$ independent of $s$. 
    Combining \eqref{eq xa1 formula 1} with \eqref{eq ibaa'} for $r_i=2$, we conclude that $R_1(s)$ is of the form
    \begin{align}\label{eq R1}
        R_1(s)=-\frac{x_{\mathsf{A}_1}(\overline{s})_2}{4r_X} -F_2\left(\left\lfloor\frac{s}{2}\right\rfloor x_2+(\overline{s})_2y_2\right)
    \end{align}
    for some integers $x_2, y_2$. Here the case that $(2,1)\not \in B_X$ is also included by taking $x_2=y_2=0$. 

    Note that 
    \[
        \sum_{\substack{C\subset \Sing(X)\\\text{not type }\mathsf{A}_1}}(-K_X\cdot C)c_C(sA)=\begin{dcases} -\frac{(\overline{sx})_{r_0}(\overline{-sx})_{r_0}}{2r_0}=-F_{r_0}(sx)&\text{if } C_0 \text{ exists};\\0&\text{otherwise,}
        \end{dcases}
    \]
    for some integer $x$. Note also that we have
    \[
        \{3,5,11\}=\begin{cases} \{r_i\neq 2\mid (r_i,b_i)\in B_X\}\sqcup \{r_0\} &\text{if } C_0 \text{ exists};\\\{r_i\neq 2\mid (r_i,b_i)\in B_X\} &\text{otherwise.}
        \end{cases}
    \]
    So combining with \eqref{eq ibaa'}, we conclude that $R_2(s)$ is of the form
    \begin{align}\label{eq R2}
        R_2(s)=- F_3\left(\left\lfloor\frac{s}{2}\right\rfloor x_3+(\overline{s})_2y_3\right)- F_5\left(\left\lfloor\frac{s}{2}\right\rfloor x_5+(\overline{s})_2y_5\right)- F_{11}\left(\left\lfloor\frac{s}{2}\right\rfloor x_{11}+(\overline{s})_2y_{11}\right)
    \end{align}
    for some integers $x_3,y_3, x_5, y_5, x_{11}, y_{11}$.

    \medskip

    {\bf Step 2}. We determine the values of $x_i, y_i$ for $i=2,3,5,11$.  

    Applying \eqref{eq group C RR1}  for $s=2$ together with \eqref{eq R1} and \eqref{eq R2}, we have 
    \begin{align*}
        h^0(X,\mathcal O_X(2A))=\frac{4}{660}+2-F_2(x_2)-F_3(x_3)-F_5(x_5)-F_{11}(x_{11}).
    \end{align*}
    As $h^0(X,\mathcal O_X(2A))\in \mathbb{Z}$, all possible values of $x_i$ for $i=2,3,5,11$ are
    \begin{align*}
        x_2{}&\equiv0\bmod 2;\\
        x_3{}&\equiv\pm 2\bmod 3;\\
        x_5{}&\equiv\pm 4\bmod 5;\\
        x_{11}{}&\equiv\pm 4 \bmod 11;
    \end{align*}
    and we always have $h^0(X,\mathcal O_X(2A))=0$, which implies that $h^0(X,\mathcal O_X(A))=0$.

    After possibly changing $x_i, y_i$ to $-x_i, -y_i$ and modulo $r_i$, we may just assume that $x_2=0$, $x_3= 2$, $x_5= 4$, $x_{11}= 4$. So \eqref{eq group C RR1} becomes  
    \begin{align}\label{eq group C RR2}
        \begin{split}
            h^0(X,\mathcal O_X(sA))={}& 
            \frac{s^2}{660}+2-\frac{x_{\mathsf{A}_1}(\overline{s})_2}{4r_X} -F_2((\overline{s})_2y_2) -F_3\left(2\left\lfloor\frac{s}{2}\right\rfloor+(\overline{s})_2y_3\right)\\{}&-F_5\left(4\left\lfloor\frac{s}{2}\right\rfloor+(\overline{s})_2y_5\right)-F_{11}\left(4\left\lfloor\frac{s}{2}\right\rfloor+(\overline{s})_2y_{11}\right).
        \end{split}
    \end{align}
    By $h^0(X, \mathcal{O}_X(A))-h^0(X, \mathcal{O}_X(3A))\in \mathbb{Z}$,
    we have
\begin{align*}
    -\frac{2}{165}-F_3(y_3)+F_3(2+y_3)-F_5(y_5)+F_5(4+y_5)-F_{11}(y_{11})+F_{11}(4+y_{11})\in \mathbb{Z}.
\end{align*}
  The only possible values of $y_i$ for $i=3,5,11$ are
    \begin{align}
        \begin{cases}\label{eq group C yi}
        y_3 \equiv 1\bmod 3;\\
        y_5 \equiv 2\bmod 5;\\
        y_{11} \equiv 2 \bmod 11.
        \end{cases}
    \end{align}
    Putting these values into \eqref{eq group C RR2}, by  $h^0(X, \mathcal{O}_X(A))=0$, we get 
     \begin{align}\label{eq group C xa1}
         \frac{x_{\mathsf{A}_1} }{4r_X}+F_2(y_2)=\frac{1}{4},
    \end{align}
    which implies that 
    \begin{align}\label{eq group C xa1s}  
        \frac{x_{\mathsf{A}_1}(\overline{s})_2}{4r_X} +F_2((\overline{s})_2y_2)=\frac{1}{4}(\overline{s})_2=F_{2}(s).
    \end{align}
    Combining \eqref{eq group C RR2}, \eqref{eq group C yi}, and \eqref{eq group C xa1s}, we get
    \begin{align*} 
        h^0(X,\mathcal O_X(sA))={}& 
        \frac{s^2}{660}+2- F_2  (s) -F_3(s)-F_5(2s)-F_{11}(2s). 
    \end{align*}
    This concludes the proof.
\end{proof}

As a corollary, we can determine $x_{\mathsf{A}_1}$ for candidates in Group C with $2\nmid r_X$, which allows us to rule out $6$ cases that violate Theorem~\ref{thm Delta>jC}. 

\begin{cor}\label{cor xa1=rx} 
    Let $(X,A)$ be an unmodifiable canonical Fano $3$-fold in Group C. If $2\nmid r_X$, then $x_{\mathsf{A}_1}=r_X$. 
\end{cor}

\begin{proof}
    In the proof of Theorem~\ref{thm RR for Group C}, if $2\nmid r_X$, then $x_2=y_2=0$. Hence $x_{\mathsf{A}_1}=r_X$ by \eqref{eq group C xa1}.
\end{proof}

\begin{prop}\label{prop rule out C-}
    \textnumero4, \textnumero7, \textnumero8, \textnumero12, \textnumero14, \textnumero15 in Table~\ref{main tab} do not exist. 
\end{prop}

\begin{proof}
    In these cases, $2\nmid r_X$. Take $r_0=1$ if $C_0$ does not exist, then in all these cases, by Corollary~\ref{cor xa1=rx} and \eqref{eq xa1 formula 2}, we have
    \begin{align*} 
        {}&\sum_{C\subset \text{\rm Sing}(X)}\left(j_C-\frac{1}{j_C}\right)(-r_XK_X\cdot C) =\frac{3r_X}{2}+ \left(r_0-\frac{1}{r_0}\right)r_X>\nabla_X,
    \end{align*}
    which contradicts Theorem~\ref{thm Delta>jC}. 
\end{proof}
 
We can also compute explicit values of $h^0(X, \mathcal O_X(sA))$ from Theorem~\ref{thm RR for Group C}. This will be used later in \S\,\ref{sec group C2} to study the geometry of linear systems $|sA|$.

\begin{lem}\label{lem.rrfor<=33}
    Let $(X,A)$ be an unmodifiable canonical Fano $3$-fold in Group C. Then for $0<s\leq 34$, we have
    \begin{align*}
        h^0(X, \mathcal O_X(sA))=
        \begin{dcases}
        0 & \text{if } s=1,2,3,4,7,8,9,13,14,19 ;\\
        1 & \text{if } s=5,6,10,11,12,15,16,17,18,20,21,23,24,25,26,29,31;\\
        2& \text{if } s=22,27,28,30,32, 34;\\ 
        3& \text{if } s=33. 
        \end{dcases}
    \end{align*}
\end{lem}

\begin{proof}
    The calculation of $h^0(X, \mathcal O_X(sA))$ follows directly from Theorem~\ref{thm RR for Group C}. 
\end{proof}

\begin{rem}\label{rem RR=RR}
    Indeed, by Theorem~\ref{thm RR for Group C}, we can show that for an unmodifiable canonical Fano $3$-fold $(X,A)$ in Group C and for any integer $0<s<66$, we have
    \[
        h^0(X, \mathcal O_X(sA))=h^0(\mathbb{P}(5,6,22,33), \mathcal{O}_{\mathbb{P}(5,6,22,33)}(s)).
    \]
    This can be proved by direct computation or using Theorem~\ref{thm.rr.sA} and Example~\ref{ex compute -rKC}, and we leave the details to the interested reader. This phenomenon suggests that Group C is difficult to rule out numerically. 
\end{rem}

For the rest of the paper, we will deal with \textnumero3,  \textnumero6, \textnumero11, \textnumero13, \textnumero21, \textnumero22 in Table~\ref{main tab}. All the methods we have been used in \S\,\ref{sec 7} and \S\,\ref{sec group C1} do not give any contradiction to these cases, so we need to look at the geometric structure of $X$ in more detail. The main tool we will use is a foliation of rank $2$ on $X$ naturally induced by Theorem~\ref{thm.kmineqrefined}.

\section{Geometry of foliations of rank two}\label{sec.foliation}

We gather some basic notions and facts regarding foliations on varieties. We refer the reader to \cite{druel}*{\S\,3} and the references therein for a more detailed explanation. 

\begin{defn}
    A {\it foliation} on a normal variety $X$ is a non-zero coherent subsheaf $\mathcal{F}$ of the tangent sheaf $\mathcal{T}_X$ such that 
    \begin{enumerate}
	\item $\mathcal{T}_X/\mathcal{F}$ is torsion-free, and	
	\item $\mathcal{F}$ is closed under the Lie bracket.
    \end{enumerate} 
    The {\it canonical divisor} of a foliation $\mathcal{F}$ is any integral Weil divisor $K_{\mathcal{F}}$ on $X$ such that $\det(\mathcal{F})\cong \mathcal{O}_X(-K_{\mathcal{F}})$. 
 
    Let $X_{\circ}\subset X_{\reg}$ be the largest open subset over which $\mathcal{T}_X/\mathcal{F}$ is locally free. A \emph{leaf} of $\mathcal{F}$ is a maximal connected and immersed holomorphic submanifold $L\subset X_{\circ}$ such that $\mathcal{T}_L=\mathcal{F}|_L$. A leaf is called \emph{algebraic} if it is open in its Zariski closure and a foliation $\mathcal{F}$ is said to be \emph{algebraically integrable} if its leaves are algebraic. 
\end{defn}

Let $\mathcal{F}$ be an algebraically integrable foliation on a normal projective variety $X$. Then there exists a diagram, called the \emph{family of leaves}, as follows (\cite{druel}*{\S\,3.6})

\begin{equation}\label{eq.familyofleaves}
    \begin{tikzcd}[row sep=large, column sep=large]
	U \arrow[r,"e"] \arrow[d,"g"]
	& X \\
	T
    \end{tikzcd}
\end{equation}
with the following properties: 

\begin{enumerate}
    \item $U$ and $T$ are normal projective varieties;
    \item $e$ is birational and $g$ is an equidimensional fibration;
    \item For a general $t\in T$, $e$ is finite on $g^{-1}(t)$ and the image $e(g^{-1}(t))$ is the closure of a leaf of $\mathcal{F}$. 
\end{enumerate} 

Assume in addition that $K_{\mathcal{F}}$ is $\mathbb{Q}$-Cartier, then there exists a canonically defined effective $e$-exceptional $\mathbb{Q}$-divisor $\Delta_U$ such that 
\begin{align}\label{eq.PullBackKF}
    K_U-g^*K_T-R(g) + \Delta_U\sim K_{e^{-1}\mathcal{F}} + \Delta_U \sim_{\mathbb{Q}} e^*K_{\mathcal{F}},
\end{align}
where 
\begin{align}\label{eq.def Rg}
    R(g)\coloneq\sum_P(g^*P-(g^*P)_{\red})
\end{align}
is the \emph{ramification divisor} of $g$, where $P$ runs through all prime divisors on $T$, see \cite{druel}*{\S\,3.6}. 

We are mainly interested in algebraically integrable foliations of rank $2$ on a $\mathbb Q$-factorial Fano variety of Picard number $1$. Unlike the rank $1$ case (cf. \cite{liu-liu}*{Proposition~3.8} and \cite{jiang-liu-liu}*{Proposition~3.6}) where a general fiber of $g$ in \eqref{eq.familyofleaves} is a smooth rational curve, the rank $2$ case is more challenging, as a general fiber of $g$ is a normal surface not necessarily smooth, which makes the situation much more complicated. We will use the following lemma to pass to a modification of $U$ whose general fibers are smooth.

\begin{lem}\label{lem.termoffoliation}
    Let $\mathcal{F}$ be an algebraically integrable foliation on a normal projective variety $X$ such that $K_{\mathcal{F}}$ is $\mathbb{Q}$-Cartier. Then there is a commutative diagram
    \begin{equation}\label{eq.termoffolation}
        \begin{tikzcd}[row sep=large, column sep=large]
		V \arrow[rr,"\mu", bend left=20]\arrow[r,"\pi"] \arrow[dr,"f"] & U \arrow[r,"e"] \arrow[d,"g"] & X \\
		  & T
	\end{tikzcd}
    \end{equation} 
    with the following properties:
    \begin{enumerate}
        \item $U, T, e, g$ are as in \eqref{eq.familyofleaves};
        \item $\pi$ is projective birational;
        \item $V$ is $\mathbb{Q}$-factorial terminal and $K_V$ is nef over $U$;
        \item There exists a canonically defined $\mathbb{Q}$-divisor $\Delta_V$ on $V$ such that 
            \begin{align}\label{eq.PullBackKFtoV}
	       K_{\mu^{-1}\mathcal{F}} + \Delta_V \sim_{\mathbb{Q}} \pi^*(K_{e^{-1}\mathcal{F}} + \Delta_U)\sim_{\mathbb{Q}}\mu^*K_{\mathcal{F}},
            \end{align} 
            where the $f$-horizontal part of $\Delta_V$ is effective. 
    \end{enumerate}
\end{lem}

\begin{proof}
    The existence of $V$ is by \cite{BCHM}*{Theorem~1.2}, namely, $V$ is a minimal model of a resolution of $U$ over $U$. Then $\Delta_V$ is canonically defined by \eqref{eq.PullBackKFtoV}. 

    In order to show that the $f$-horizontal part of $\Delta_V$ is effective, we may shrink $T$ so that $K_{e^{-1}\mathcal{F}}\sim K_U$ and $K_{\mu^{-1}\mathcal{F}}\sim K_V$ by \eqref{eq.PullBackKF}. So \eqref{eq.PullBackKFtoV} shows that $-\Delta_V$ is nef over $U$ and $\pi_*\Delta_V=\Delta_U$ is effective, which implies that $\Delta_V$ is effective by the negativity lemma (\cite{kollar-mori}*{Lemma~3.39}). 
\end{proof}

The following theorem allows us to understand the general fibers under good numerical conditions. 

\begin{thm}\label{thm.hirzebruch}
    Let $X$ be a $\mathbb Q$-factorial canonical Fano variety of dimension $d\geq 3$ and of Picard number $1$. Let $\mathcal{F}$ be a foliation of rank $2$ on $X$ such that $\mu_{c_1(X),\min}(\mathcal F)>0$.
    Then the following assertions hold.
    \begin{enumerate}
        \item $\mathcal{F}$ is algebraically integrable;

        \item   Let $F$ be a general fiber of $f$ in \eqref{eq.termoffolation}. If $-K_{\mathcal F}+\frac{5}{6}K_X$ is ample, then there exists a birational morphism $\phi\colon F\to \mathbb{F}_n$ for some $n\in\{1,2\}$.
    \end{enumerate}
\end{thm}

\begin{proof}

    (1) As $\mu_{c_1(X),\min}(\mathcal F)>0$, by \cite{cp}*{Theorem~1.1} or \cite{ouwenhao}*{Proposition~2.2}, $\mathcal F$ is algebraically integrable and a general fiber $G$ of $g$ in \eqref{eq.familyofleaves} is rationally connected.

    (2) As $X$ is of Picard number $1$, we may write $-K_{\mathcal{F}}\sim_{\mathbb{Q}} -uK_X$ where $u>\frac{5}{6}$ by assumption. 

    As $V$ has only terminal singularities, $F$ is a smooth rationally connected surface as the natural map $\pi|_F: F\to G$ between general fibers is birational. By Lemma~\ref{lem.termoffoliation}, there exists an effective $\mathbb{Q}$-divisor $\Delta_F\coloneq \Delta_V|_F$ such that 
    \[
        -K_F-\Delta_F\sim_{\mathbb{Q}} \mu^*(-K_{\mathcal F})|_F\sim_{\mathbb{Q}} \mu^*(-uK_X)|_F
    \]
    is nef and big. On the other hand, as $X$ has only canonical singularities, there exists an effective $\mathbb{Q}$-divisor $\Delta'$ on $V$ such that $ K_V=\mu^*(K_X)+\Delta'$. Hence there exists an effective $\mathbb{Q}$-divisor $\Delta'_F\coloneq \Delta'|_F$ such that 
    \[
        -K_F+\Delta'_F\sim_{\mathbb{Q}}\mu^*(-K_X)|_F.
    \]

    By a standard minimal model program, there exists a birational morphism $F\to S$ where $S$ is either $\mathbb P^2$ or a Hirzebruch surface $\mathbb F_n$ for some $n\geq 0$.

    First, we claim that $F$ is neither $\mathbb{P}^2$ nor $\mathbb{F}_0$. Suppose that $F$ is either $\mathbb{P}^2$ or $\mathbb{F}_0$, then $\pi|_F\colon F\to G$ is the identity morphism as $F$ admits no non-trivial birational contraction. Hence the coefficients of $\Delta_F$ are at least $1$ by \cite{jiang-liu-liu}*{Proposition~3.5}. If $F=\mathbb{P}^2$, then for a line $\ell$ on $F$, $(\Delta_F\cdot \ell)\geq 1$ and hence
    \[ 
        \frac{5}{6}<u=\frac{\mu^*(-uK_X)|_F\cdot \ell }{\mu^*(-K_X)|_F \cdot \ell }=\frac{(-K_F-\Delta_F)\cdot \ell }{(-K_F+\Delta'_F)\cdot \ell }\leq \frac{2}{3},
    \] 
    which is a contradiction.
    If $F=\mathbb{F}_0=\mathbb{P}^1\times \mathbb{P}^1$, then there exists a ruling $\ell$ on $F$ such that $(\Delta_F\cdot \ell)\geq 1$ and hence
    \[ 
        \frac{5}{6}<u=\frac{\mu^*(-uK_X)|_F\cdot \ell }{\mu^*(-K_X)|_F \cdot \ell }=\frac{(-K_F-\Delta_F)\cdot \ell }{(-K_F+\Delta'_F)\cdot \ell }\leq \frac{1}{2},
    \]
    which is a contradiction.

    So if $S$ is either $\mathbb{P}^2$ or $\mathbb{F}_0$, then $F\to S$ is non-trivial and factors through a blowup of $S$, which dominates $\mathbb{F}_1$. Hence we conclude that there exists a birational morphism $\phi\colon F\to \mathbb F_n$ for some $n\geq 1$.

    Finally we show that $n\leq 2$. Let $\ell$ be a general fiber of the natural map $ \mathbb{F}_n\to \mathbb{P}^1$ and let $\sigma_0$ be the negative section. Let $ \ell', \sigma_0'$ be the strict transforms of $\ell,\sigma_0$ on $F$. Then $\sigma_0'^2\leq \sigma_0^2=-n$ and $(\sigma_0'\cdot \ell')=1$. 
 
    Set $a\coloneq \mult_{\sigma_0'}(\Delta_F)\geq 0$. It follows that 
    \[
        \frac{5}{6}< u=\frac{\mu^*(-uK_X)|_F\cdot \ell'}{\mu^*(-K_X)|_F \cdot \ell'}=\frac{(-K_F-\Delta_F)\cdot \ell'}{(-K_F+\Delta'_F)\cdot \ell'}\leq \frac{2-a}{2},
    \]
    so $a<\frac13$. On the other hand, as $-K_F-\Delta_F$ is nef, we have
    \begin{align*}
        0\leq (-K_F-\Delta_F)\cdot \sigma_0'\leq (-K_F-a \sigma_0')\cdot \sigma_0' =2+(1-a)\sigma_0'^2< 2-\frac{2n}{3},
    \end{align*}
    which implies that $n\leq 2$.
\end{proof}

\section{Ruling out remaining cases: Group C continued}\label{sec group C2}

In this section, we treat the remaining cases in Group C: \textnumero3,  \textnumero6, \textnumero11, \textnumero13, \textnumero21, \textnumero22 in Table~\ref{main tab}. We refer to them as Group C$^+$. 

For an {unmodifiable} canonical Fano $3$-fold $(X, A)$, as $\Cl(X)=\mathbb{Z}A$, we may define $\iota\colon \Cl(X)\to \mathbb{Z}$ such that $D\sim \iota(D)A$ for any $D\in \Cl(X)$. For example, $\iota(-K_X)=\qQ(X)$. We will always denote $q\coloneq \qQ(X)$ in this section. 

First, we show that there exists a natural foliation of rank $2$ on $X$ for Group C$^+$. 
      
\begin{prop}\label{prop.rank2fol} 
    Let $(X,A)$ be an unmodifiable canonical Fano $3$-fold in Group C$^+$. Then there exists an algebraically integrable foliation $\mathcal F$ of rank $2$ on $X$ with $\mu_{c_1(X), \min}(\mathcal F)>0$ such that
    \[
        q>\iota(-K_{\mathcal F})\geq \max\left\{ q-10, \frac{57}{67}q\right\}.
    \]
    All possible values of $\iota(-K_{\mathcal F})$ are listed in Table~\ref{Tab pq}.
\end{prop}

\begin{table}[htbp]
    \caption{Values of $\iota(-K_{\mathcal F})$}\label{Tab pq}
    \begin{tabular}{lcccccc}
	\hline
		\textnumero & $3$ & $6$ & $11$ & $13$ & $21$ & $22$ \\
    \hline
        $q=$ & $70$ & $72$ & $69$ & $68$ &$67$ & $71$\\
        $\iota(-K_{\mathcal F})\geq $ & $66$ & $71$ & $64$ & $61$ & $57$ & $68$\\
	\hline
	\end{tabular}
\end{table}

\begin{proof}
    By Table~\ref{tab group C} and Corollary~\ref{cor xa1=rx}, we have \begin{align*} 
        \delta_X\coloneq \sum_{C\subset \text{\rm Sing}(X)}\left(j_C-\frac{1}{j_C}\right)(-r_XK_X\cdot C) =
        \begin{cases}
            (\frac{3}{2}+\frac{24}{5})r_X =207.9& \text{in \textnumero}3;\\ 
            (\frac{3}{2}+\frac{8}{3})r_X=299+\frac{1}{6} &\text{in \textnumero}6;\\  
            \frac{8}{3} r_X=293+\frac{1}{3} & \text{in \textnumero}11;\\  
            \frac{3}{2} r_X=247.5 &\text{in \textnumero}13;\\
            0& \text{in \textnumero}21, \text{\textnumero}22.
        \end{cases}
    \end{align*}

    If $c_1(X)^3\leq \frac{16}{5}\hat{c}_2(X)\cdot c_1(X)$, then by Theorem~\ref{thm.c1c2 diff}, we have
    \[
        r_Xc_2(X)\cdot c_1(X)-\frac{5}{16}r_Xc_1(X)^3\geq r_Xc_2(X)\cdot c_1(X)-r_X\hat{c}_2(X)\cdot c_1(X)=\delta_X,
    \]
    but this contradicts the data in Table~\ref{tab group C}. So by Theorem~\ref{thm.kmineqrefined}, we are in the case of $(l, r_1)=(2,2)$ or $(3,1)$. Then we can take $\mathcal F=\mathcal E_{l-1}$ in the Harder--Narasimhan filtration \eqref{eq.hnfiltration}  of $\mathcal{T}_X$ which is of rank $2$, where \[\mu_{c_1(X), \min}(\mathcal F)>\mu_{c_1(X)}(\mathcal T_X/\mathcal F)\geq 0\] by definition and \cite{ouwenhao}*{Theorem~1.4}. Hence $\mathcal F$ is an algebraically integrable foliation of rank $2$ on $X$ (see \cite{cp}*{Theorem~1.1} or \cite{ouwenhao}*{Proposition~2.2}). 
 
    By Theorem~\ref{thm.kmineqrefined} where $p=\iota(-K_\mathcal{F})$, we have $p<q$ and 
    \begin{align*}
        c_1(X)^3\leq {}&\max\left\{\frac{4q^2}{p(4q-3p)}, \frac{4q^2}{-4p^2+6pq-q^2}\right\}\hat{c}_2(X)\cdot c_1(X)\\={}& \frac{4q^2}{-4p^2+6pq-q^2}\hat{c}_2(X)\cdot c_1(X).
    \end{align*} 
    Then by Theorem~\ref{thm.c1c2 diff}, we have
    \begin{align*}
        {}&r_Xc_2(X)\cdot c_1(X)-\frac{-4p^2+6pq-q^2}{4q^2}r_Xc_1(X)^3\\
        \geq {}&r_Xc_2(X)\cdot c_1(X)-r_X\hat{c}_2(X)\cdot c_1(X)=\delta_X.
    \end{align*}
    Solving this inequality with the data in Table~\ref{tab group C} while keeping in mind that $p>\frac{1}{2}q$, we get the lower bound of $p$ as in Table~\ref{Tab pq}. This concludes the proof.
\end{proof}

In order to study the foliation $\mathcal{F}$ in Proposition~\ref{prop.rank2fol}, we need to study its leaves, which are divisors on $X$. So we use Lemma~\ref{lem.rrfor<=33} to study the geometry of linear systems $|sA|$.

\begin{lem}\label{lem irreducible in sA}
    Let $(X,A)$ be an unmodifiable canonical Fano $3$-fold in Group C. Take $A_5\in |5A|$ and $A_6\in |6A|$ to be the unique elements. Then the following assertions hold.
    \begin{enumerate}
        \item If $D$ is an effective integral Weil divisor on $X$ with $\iota(D)\leq 34$ such that $\Supp(D)$ does not contain $A_5$ nor $A_6$, then $\iota(D)\in \{0, 22,30, 33\}$. 

        \item If $D$ is an effective integral Weil divisor on $X$ with $0<\iota(D)<22$, then either $D\geq A_5$ or $D\geq A_6$.

        \item If $D$ is a prime divisor on $X$ with $\iota(D)<30$, then $\iota(D)\in \{5,6,22\}$.
    \end{enumerate}   
\end{lem}

\begin{proof}
    When $s=\iota(D)>0$, the condition that $\Supp(D)$ does not contain $A_5$ nor $A_6$ implies that  
    \begin{align*}
        h^0(X, \mathcal{O}_X(sA))>h^0(X, \mathcal{O}_X((s-5)A)), \\
        h^0(X, \mathcal{O}_X(sA))>h^0(X, \mathcal{O}_X((s-6)A)).
    \end{align*}  
    Then the conclusion (1) follows from Lemma~\ref{lem.rrfor<=33}.
    Conclusions (2) and (3) follow from (1).     
\end{proof}

\begin{prop}\label{prop non-reduced in sA}
    Let $(X,A)$ be an unmodifiable canonical Fano $3$-fold in Group C and let $D$ be a non-reduced effective integral Weil divisor on $X$ such that $\iota(D)\leq 59$. Then $D$ satisfies one of the following properties:
    \begin{enumerate}
        \item $\iota(D)\leq 43$ and either $D\geq 2A_5$ or $D\geq 2A_6$.
        \item $45\leq \iota(D)\leq 59$ and one of the following holds:
            \begin{enumerate}
                \item $D\geq 2D_0+A_5$ for some $D_0\in |22A|$;
                \item $D\geq 2D_0+A_6$ for some $D_0\in |22A|$;
                \item $D\geq 2A_5$;
                \item  $D\geq 2A_6$.
            \end{enumerate}
        \item $\iota(D)=44$ and one of the following holds:
            \begin{enumerate}
                \item $D=2D_0$ for some $D_0\in |22A|$;
                \item $D=D_0+2A_5+2A_6$ for some $D_0\in |22A|$;
                \item $D=4A_5+4A_6$.
            \end{enumerate}
    \end{enumerate}
    Here $A_5\in |5A|$ and $A_6\in |6A|$ are the unique elements. 
\end{prop}   

\begin{proof}
    By assumption, there exists a prime divisor $E$ such that $D\geq 2E$. So $\iota(E)\leq \frac{\iota(D)}{2}< 30$. By Lemma~\ref{lem irreducible in sA}(3), we have $\iota(E)\in \{5,6,22\}$. 

    If $\iota(D)\leq 43$, then $\iota(E)\in \{5,6\}$ and hence $E$ is either $A_5$ or $A_6$, which concludes (1). 

    Suppose that $45\leq \iota(D)\leq 59$. If $\iota(E)\in \{5,6\}$, then $D\geq 2A_5$ or $D\geq 2A_6$; if $\iota(E)=22$, then $1\leq \iota(D-2E)\leq 15$, which implies that $D-2E\geq A_5$ or $\geq A_6$ by Lemma~\ref{lem irreducible in sA}(2).

    Suppose that $\iota(D)=44$. If $\iota(E)=22$ then clearly $D=2E$. Now suppose that $E=A_5$ or $A_6$, then we may write 
    \[
        D=D_0+aA_5+bA_6
    \]
    where $a\geq 2$ or $b\geq 2$ and $\Supp(D_0)$ does not contain $A_5$ nor $A_6$. Then $\iota(D_0)\leq 34$, which implies that \[44-5a-6b=\iota(D_0)\in \{0, 22,30, 33\}\] by Lemma~\ref{lem irreducible in sA}(1). If $D_0=0$, then we have $D=4A_5+4A_6$; if $D_0\in |22A|$, then we have $D=D_0+2A_5+2A_6$; $\iota(D_0)\in \{30, 33\}$ is impossible. 
\end{proof}
   
Then we can rule out all candidates in Group C$^+$.

\begin{prop}\label{prop rule out C+}
     All candidates in Group C$^+$ (\textnumero3,  \textnumero6, \textnumero11, \textnumero13, \textnumero21, \textnumero22 in Table~\ref{main tab}) do not exist. 
\end{prop}

\begin{proof}
    Let $(X,A)$ be an unmodifiable canonical Fano $3$-fold in Group C$^+$. By Proposition~\ref{prop.rank2fol}, there exists an algebraically integrable foliation $\mathcal F$ of rank $2$ on $X$ such that $\mu_{c_1(X), \min}(\mathcal F)>0$ and $q>\iota(-K_{\mathcal F})\geq \max\{ q-10, \frac{57}{67}q\}$. Denote $p\coloneq \iota(-K_\mathcal{F})$. 

    We keep the notation in Lemma~\ref{lem.termoffoliation}. 
    Recall that as  \eqref{eq.termoffolation} we have the following commutative diagram
    \[
        \begin{tikzcd}[row sep=large, column sep=large]
		V \arrow[rr,"\mu", bend left=20]\arrow[r,"\pi"] \arrow[dr,"f"] & U \arrow[r,"e"] \arrow[d,"g"] & X \\
		  & T
	\end{tikzcd}
    \]
    where $F$ is a general fiber of $f$ and $G=\pi(F)$ is a general fiber of $g$. 

    \begin{claim}\label{claim.mu(F)}
        $\iota(\mu_*F)\geq 60$, that is, $\mu_*F-60A$ is nef. 
    \end{claim}

    \begin{proof}
        Note that $\mu_*F=e_*G$. Suppose to the contrary that $\iota(e_*G)\leq 59$.
        
        As $X$ is rationally connected by \cite{zhangqi}, $U$ is rationally connected. Hence $T\cong\mathbb{P}^1$ and $-g^*K_T\sim 2G$. By \eqref{eq.PullBackKF}, we have
        \begin{align*}
            e_*R(g)\sim_{\mathbb{Q}} e_*K_U+e_*(2G)-K_{\mathcal F}\sim -qA+2e_*G+pA,
        \end{align*}
        where $R(g)$ is the ramification divisor of $g$. 
        In particular, 
          \begin{align}
            \iota(e_*R(g))=2\iota(e_*G)-q+p. \label{eq:R(g)>G}
        \end{align}
        By the definition of $R(g)$ in \eqref{eq.def Rg}, there are non-reduced divisors $D_1,\dots, D_k\in |e_*G|$ such that
        \begin{align}
            e_*R(g)=\sum_{i=1}^k(D_i-(D_i)_{\text{\rm red}}).\label{eq R=D-D}
        \end{align}
        Here note that pairwisely $D_1, \dots, D_k$ have no common components as they are images of different fibers of $g$. 
        
        First we claim that $k\geq 3$. By \eqref{eq R=D-D}, $\iota(e_*R(g))<k\iota(e_*G)$; on the other hand, $|e_*G|$ is a movable linear system, so $\iota(e_*G)\geq 22$ by Lemma~\ref{lem.rrfor<=33}. This implies that $k>1$ by \eqref{eq:R(g)>G} as $q-p\leq 10$. If $k=2$, as $(D_1)_{\text{\rm red}}$ and $(D_2)_{\text{\rm red}}$ are effective divisors without common components, by Lemma~\ref{lem.rrfor<=33} again, we have 
        \[
            \iota((D_1)_{\text{\rm red}}+(D_2)_{\text{\rm red}})\geq 5+6=11,
        \]
        which implies that 
        \[
            \iota(e_*R(g))\leq 2\iota(e_*G)-5-6<2\iota(e_*G)-q+p,
        \]
        which contradicts \eqref{eq:R(g)>G}. So we proved that $k\geq 3$.
       
        Then $|e_*G|$ contains at least $3$ non-reduced elements $D_1, \dots, D_k$ pairwisely having no common component. By Proposition~\ref{prop non-reduced in sA}, this means that $\iota(e_*G)=44$; moreover, for each $1\leq i\leq k$, $D_i-(D_i)_{\text{\rm red}}$ belongs to the set
        \[\{D_0, A_5+A_6, 3A_5+3A_6\mid D_0\in |22A|\}.
       \]  In particular, both $\iota(e_*R(g))$ and $\iota(e_*G)$ are divided by $11$, but this contradicts \eqref{eq:R(g)>G} as $0<q-p\leq 10$.   
    \end{proof}

    By Lemma~\ref{lem.termoffoliation}, there exists an effective $\mathbb{Q}$-divisor $\Delta_F\coloneq \Delta_V|_F$ such that 
    \begin{align*}
        -K_F-\Delta_F\sim_{\mathbb{Q}} \mu^*(-K_\mathcal{F})|_F\sim_{\mathbb{Q}}\mu^*(pA)|_F
    \end{align*}
    is a nef and big $\mathbb{Q}$-divisor. By the projection formula, Claim~\ref{claim.mu(F)}, and Table~\ref{Tab pq}, we have
    \[
        (-K_F-\Delta_F)^2=(\mu^*(pA)|_F)^2=(pA)^2\cdot \mu_*F\geq 60p^2A^3=\frac{60p^2}{330q}>8.
    \]
    On the other hand, as $p\geq\frac{57}{67}q>\frac{5}{6}q$, by Theorem~\ref{thm.hirzebruch}, there is a birational morphism $\phi\colon F\to \mathbb{F}_n$ for some $n\in\{1,2\}$. Then $-K_{F}- \Delta_F\leq \phi^*(-K_{\mathbb{F}_n})$ by the negativity lemma (\cite{kollar-mori}*{Lemma~3.39}). Also $-K_{\mathbb{F}_n}$ is nef as $n\in\{1,2\}$. So we have 
    \[
        (-K_F-\Delta_F)^2\leq (-K_F-\Delta_F)\cdot \phi^*(-K_{\mathbb{F}_n})\leq (-K_{\mathbb{F}_n})^2=8,
    \]
    which is a contradiction.
\end{proof}

\begin{proof}[Proof of Theorem~\ref{main.thm}]
    By Proposition~\ref{prop.reducetopicard1}, it suffices to show that $\qQ(X)\leq 66$ for an unmodifiable canonical Fano $3$-fold $(X, A)$. 

    Suppose that $\qQ(X)>66$, then by Theorem~\ref{thm.36cases}, the numerical data of all possible candidates of $X$ are listed in Table~\ref{main tab}. Those candidates are divided into Groups A, B, and C. Candidates in Group A are ruled out by Proposition~\ref{prop rule out A}; candidates in Group B are ruled out by Proposition~\ref{prop rule out B1}, Proposition~\ref{prop rule out B2}, Proposition~\ref{prop rule out B3}, Proposition~\ref{prop rule out B4}, and Proposition~\ref{prop rule out B5}; candidates in Group C are ruled out by Proposition~\ref{prop rule out C-} and Proposition~\ref{prop rule out C+}.
\end{proof}

We make a final remark on the case of equality $\qQ(X)=66$. 
\begin{rem}\label{rem q=66}
   By Algorithm~\ref{algo1},
   we can list the numerical data of all possible candidates of unmodifiable canonical Fano $3$-folds $(X, A)$ with $\qQ(X)=66$ in Table~\ref{tab 66}. Here note that  the numerical data of $\mathbb{P}(5,6,22,33)$ are exactly the same as \textnumero1 in Table~\ref{tab 66}. 
     {
    \begin{longtable}{LLLLLLLLL}
        \caption{Candidates for unmodifiable canonical Fano $3$-folds with $\qQ=66$}\label{tab 66}\\
        \hline
        \text{\textnumero} & B_X & \qQ  & r_X & r_Xc_1^3 & r_Xc_2c_1 &  \{p^a\}& \{\LB(p^a)\} & \nabla_X  \\
        \hline
        \endfirsthead
        \multicolumn{4}{l}{{ {\bf \tablename\ \thetable{}} \textrm{-- continued}}}
        \\
        \hline 
        \text{\textnumero} & B_X & \qQ  & r_X & r_Xc_1^3 & r_Xc_2c_1 &  \{p^a\}& \{\LB(p^a)\} & \nabla_X   \\
        \hline 
        \endhead
        \hline
        \hline \multicolumn{4}{c}{{\textrm{Continued on next page}}} \\ \hline
        \endfoot
        
        \hline \hline
        \endlastfoot
        1& \{(5,2)\} & 66 & 5 & 66 & 96 &2,3,11&1,5,5 &79.02 \\
        2& \{(7,2)\} & 66 & 7 & 66 & 120 &2,3,11&1,7,7 &103.02 \\
        3& \{2\times (2,1), (5,1)\} & 66 & 10 & 198 & 162 &2,11& 1,10 &111.05 \\
        4& \{2\times (4,1), (5,1)\} & 66 & 20 & 726 & 234 &2,3 & 1,10 & 47.17 \\
        5& \{2\times (4,1), (5,2),(7,1)\} & 66 & 140 & 2178 & 678 &2 & 1 & 117.5 \\
        6& \{3\times (3,1), (5,2),(7,3)\} & 66 & 105 & 726 & 456 &2,3 & 1,35 & 269.17 \\
        7& \{2\times (2,1), 3\times (3,1), (5,2)\} & 66 & 30 & 726 & 246 &2,3 & 1,10 & 59.17 \\  
    \end{longtable}
    } 
\end{rem}

\section*{Acknowledgments} 
The authors would like to thank Jungkai A. Chen, Masayuki Kawakita, Jie Liu, Yuri Prokhorov, and Miles Reid for helpful discussions. 
We would like to thank Yuri Prokhorov for pointing out an inaccurate statement on crepant divisors in the draft, which leads to the definition of non-split crepant divisors (see Definition~\ref{def nonsplit}).

C.~Jiang was supported by National Key Research and Development Program of China (No. 2023YFA1010600, No. 2020YFA0713200) and NSFC for Innovative Research Groups (No. 12121001). C.~Jiang is a member of the Key Laboratory of Mathematics for Nonlinear Sciences, Fudan University. H.~Liu is supported in part by the National Key Research and Development Program of China (No. 2023YFA1009801) and NSFC (No. 12571048).

\end{document}